\documentclass[11pt,a4paper]{article}

\usepackage[T1]{fontenc}
\usepackage{lmodern,microtype}
\usepackage{amsfonts,amsmath,amssymb,amsthm}
\usepackage{aliascnt}
\usepackage{mathtools,bm,mathrsfs,mleftright}
\usepackage{xcolor,enumitem,autobreak}
\allowdisplaybreaks[4]
\usepackage{wrapfig,subcaption,multirow,placeins,needspace}
\usepackage{hyperref,cleveref}
\usepackage{natbib}

\usepackage{silence}
\crefname{theorem}{Theorem}{Theorems}
\Crefname{theorem}{Theorem}{Theorems}
\crefname{lemma}{Lemma}{Lemmas}
\Crefname{lemma}{Lemma}{Lemmas}
\crefname{corollary}{Corollary}{Corollaries}
\Crefname{corollary}{Corollary}{Corollaries}
\crefname{proposition}{Proposition}{Propositions}
\Crefname{proposition}{Proposition}{Propositions}
\crefname{definition}{Definition}{Definitions}
\Crefname{definition}{Definition}{Definitions}
\crefname{remark}{Remark}{Remarks}
\Crefname{remark}{Remark}{Remarks}
\crefname{example}{Example}{Examples}
\Crefname{example}{Example}{Examples}
\crefname{assumption}{Assumption}{Assumptions}
\Crefname{assumption}{Assumption}{Assumptions}
\crefname{condition}{Condition}{Conditions}
\Crefname{condition}{Condition}{Conditions}
\crefname{section}{Section}{Sections}
\Crefname{section}{Section}{Sections}
\crefname{subsection}{Subsection}{Subsections}
\Crefname{subsection}{Subsection}{Subsections}
\crefname{subsubsection}{Subsection}{Subsections}
\Crefname{subsubsection}{Subsection}{Subsections}

\crefformat{equation}{(#2#1#3)}
\crefrangeformat{equation}{(#3#1#4)--(#5#2#6)}
\crefmultiformat{equation}{(#2#1#3)}{ and (#2#1#3)}{, (#2#1#3)}{, and (#2#1#3)}

\newcommand{\E}{\mathbb{E}}

\DeclareMathOperator{\Ran}{Ran}
\DeclareMathOperator{\Tr}{Tr}
\DeclareMathOperator{\Ker}{Ker}

\newcommand{\xk}[1]{\left(#1\right)}
\newcommand{\zk}[1]{\left[#1\right]}
\newcommand{\dk}[1]{\left\{#1\right\}}

\providecommand{\abs}[1]{\left\lvert{#1}\right\rvert}
\providecommand{\norm}[1]{\left\lVert{#1}\right\rVert}

\providecommand{\dd}{~\mathrm{d}}

\theoremstyle{plain}
\newtheorem{theorem}{Theorem}[section]
\newaliascnt{lemma}{theorem}
\newtheorem{lemma}[lemma]{Lemma}
\aliascntresetthe{lemma}
\newaliascnt{corollary}{theorem}
\newtheorem{corollary}[corollary]{Corollary}
\aliascntresetthe{corollary}

\newaliascnt{proposition}{theorem}
\newtheorem{proposition}[proposition]{Proposition}
\aliascntresetthe{proposition}
\theoremstyle{definition}
\newaliascnt{definition}{theorem}

\aliascntresetthe{definition}
\newaliascnt{remark}{theorem}
\newtheorem{remark}[remark]{Remark}
\aliascntresetthe{remark}

\newtheorem{assumption}{Assumption}
\newtheorem{condition}{Condition}
 
\setcitestyle{numbers,square,comma,sort}
\usepackage{geometry,appendix}
\hypersetup{
  colorlinks=true,
  linkcolor=[rgb]{0.10,0.20,0.35},
  citecolor=[rgb]{0.10,0.30,0.20},
  urlcolor=[rgb]{0.10,0.25,0.55}
}

\title{Risk Equivalence between RKHS Regression and Sequence Models for Lipschitz Spectral Algorithms}
\author{%
  Yicheng Li\textsuperscript{1,*},
  Yuqian Cheng\textsuperscript{2,*},
  Zhuo Chen\textsuperscript{3},
  and Qian Lin\textsuperscript{4,\textdagger}
}
\date{\today}

\hypersetup{
  pdftitle={Risk Equivalence between RKHS Regression and Sequence Models for Lipschitz Spectral Algorithms},
  pdfauthor={Yicheng Li, Yuqian Cheng, Zhuo Chen, and Qian Lin},
  pdfkeywords={kernel spectral algorithms, exact risk, sequence model, Pinsker filter, Pinsker minimax risk, double operator integrals, spectral clipping}
}

\begin{document}
\maketitle
\footnotetext[1]{\texttt{ycli@sfs.ecnu.edu.cn}.
  KLATASDS-MOE, School of Statistics, East China Normal University, Shanghai, China.}
\footnotetext[2]{\texttt{yuqiancheng@hkbu.edu.cn}.
  Department of Mathematics, Hong Kong Baptist University, Hong Kong, China.}
\footnotetext[3]{\texttt{chenzhuo@tsinghua.edu.cn}.
  Department of Mathematics, Tsinghua University, Beijing 100084, China.}
\footnotetext[4]{\texttt{qianlin@tsinghua.edu.cn}.
  Department of Statistics and Data Science, Tsinghua University, Beijing 100084, China.}
\begingroup
  \renewcommand{\thefootnote}{\fnsymbol{footnote}}
  \footnotetext[1]{Yicheng Li and Yuqian Cheng contributed equally to this work.}
  \footnotetext[2]{Corresponding author: Qian Lin.}
\endgroup
\begin{abstract}
  Kernel spectral algorithms are often summarized by convergence rates, which hide how their risk depends jointly on regularization, noise, the population spectrum, target coefficients, and the chosen filter.
Gaussian sequence models arise as a simplified but characteristic setting for studying the interplay of these factors, where the kernel spectral algorithm corresponds to a coordinatewise shrinkage estimator.
Under mild assumptions, we show that the risk of a kernel spectral estimator is asymptotically equivalent to that of the corresponding Gaussian sequence model estimator with the same filter.
The explicit sequence model risk then yields a full characterization of the risk of kernel spectral algorithms in terms of the population spectrum, target coefficients, and filter.
We establish this equivalence for a broad class of spectral filters, covering kernel ridge regression, generalized ridge regression, iterated kernel ridge regression, gradient flow, stable gradient descent, smoothed spectral cutoff, spectral clipping, and Pinsker shrinkage.
Our risk equivalence not only holds in the classical fixed-dimensional regime but also applies to the high dimensional regime where the input dimension scales with the sample size.
As applications, our risk equivalence recovers the minimax upper rates, and establishes the exact Pinsker constant in RKHS regression.
 \end{abstract}

\section{Introduction}
\label{sec:introduction}

Consider the nonparametric regression model
\[
  y=f^*(x)+\varepsilon,\quad x\sim\mu,\quad \E[\varepsilon|x]=0,\quad \E[\varepsilon^2|x]=\sigma^2,
\]
with \(n\) independent observations \((x_i,y_i)\) from this model.
Regression in a reproducing kernel Hilbert space (RKHS) is a classical nonparametric method that has been widely studied in statistics and machine learning.
Often, the RKHS estimator is implemented through a spectral algorithm that applies a scalar filter  \( \varphi_\lambda \)  to the empirical kernel covariance operator.
Its statistical performance is often summarized by a convergence rate, but a rate hides how the risk depends jointly on the regularization scale, the noise level, the population spectrum, the target coefficients, and the chosen filter.
Hence, it is of interest to obtain a more informative characterization of the risk that retains this dependence, rather than only upper and lower rate bounds.

Gaussian sequence models~\citep{johnstone2017_GaussianEstimation} provide a simplified but characteristic setting in which the same statistical factors remain visible.
The eigenbasis \( (e_j)_{j \geq 1} \) of the population covariance operator \( L \) provides a natural coordinate system in which the kernel spectral algorithm becomes a coordinatewise shrinkage rule.
Let \(\lambda_j\) be the corresponding eigenvalues.
Write \(f^*=\sum_{j\ge1} f_j^* e_j\) in this basis and denote \( \bm{f}^* = (f_{j}^*)_{j\geq1} \).
For a regularization function \(\varphi_\lambda\), define the shrinkage profile \(q_\lambda(t)=t\varphi_\lambda(t)\) and the residual profile \(\psi_\lambda(t)=1-q_\lambda(t)\).
The associated Gaussian sequence model is
\[
  z_j=f_j^*+\frac{\sigma}{\sqrt n}\xi_j,
  \qquad
  \xi_j \stackrel{\mathrm{i.i.d.}}{\sim}\mathcal{N}(0,1).
\]
The population counterpart of the kernel spectral algorithm is the coordinatewise shrinkage rule obtained by applying the same profile in the eigenbasis:
\[
  \widehat{f}_{\lambda,j}^{\mathsf{seq}}
  =
  q_\lambda(\lambda_j)z_j,\quad
  \widehat{\bm{f}}_{\lambda}^{\mathsf{seq}} = \xk{\widehat{f}_{\lambda,j}^{\mathsf{seq}}}_{j \geq 1}.
\]
It is easy to see that the risk of this estimator is given by the deterministic expression
\[
  \mathcal{E}_n^{\mathsf{seq}}(q_\lambda;f^*)
  \coloneqq \E \norm{\widehat{\bm{f}}_{\lambda}^{\mathsf{seq}} - \bm{f}^*}_{\ell^2}^2
  =
  \sum_{j\ge1} \psi_\lambda(\lambda_j)^2(f_j^*)^2
  +
  \frac{\sigma^2}{n}\sum_{j\ge1} q_\lambda(\lambda_j)^2.
\]
The first term captures the approximation error through the target coefficients, population spectrum, and residual profile, while the second captures the stochastic contribution of the spectrum, noise level, and shrinkage profile.

The central contribution of this paper is to establish that kernel regression and its sequence model counterpart have asymptotically equivalent risks.
For an estimator \(\widehat{f}_\lambda\) under the nonparametric regression model, we consider the conditional (excess) risk and the risk in expectation
\begin{equation}
  \mathcal{E}_n^{\mathsf{np}}(q_\lambda;f^*\mid X)
  \coloneqq
  \E\zk{
    \norm{\widehat{f}_\lambda-f^*}_{L^2(\mu)}^2
    \,\middle|\, X
  },\quad
  \mathcal{E}_n^{\mathsf{np}}(q_\lambda;f^*) = \E \zk{\norm{\widehat{f}_\lambda-f^*}_{L^2(\mu)}^2}.
\end{equation}
Under mild assumptions, and along a deterministic admissible sequence \(\lambda=\lambda(n)\), our main theorems establish
\[
  \mathcal{E}_n^{\mathsf{np}}(q_\lambda;f^*\mid X)
  =
  \xk{1+o_{\mathbb{P}}(1)}
  \mathcal{E}_n^{\mathsf{seq}}(q_\lambda;f^*),\quad
   \mathcal{E}_n^{\mathsf{np}}(q_\lambda;f^*) = (1 + o(1))\mathcal{E}_n^{\mathsf{seq}}(q_\lambda;f^*).
\]
See \Cref{thm:conditional-risk-equivalence,thm:risk-equivalence-expectation}.
The explicit sequence model risk therefore yields a full characterization of the risk of kernel spectral algorithms in terms of the population spectrum, target coefficients, noise level, and filter function.

The equivalence covers kernel ridge regression, generalized ridge regression, iterated kernel ridge regression, gradient flow, stable gradient descent, smoothed spectral cutoff, spectral clipping, and Pinsker shrinkage.
These examples include both classical smooth filters and continuous nonsmooth profiles such as spectral clipping, so the equivalence is not confined to analytic filters~\citep{li2024_GeneralizationError}.
As an application, we use the equivalence to establish Pinsker's sharp asymptotic constant for the minimax risk in RKHS regression.

Moreover, our equivalence holds uniformly over families of problem instances satisfying the corresponding uniformity conditions.
This allows the input dimension, design distribution, and kernel spectrum to vary with the sample size, and therefore covers high-dimensional regimes in which both \(n\) and \(d\) diverge.
For spherical kernels in the polynomial regime \(n\asymp d^\gamma\), we verify these conditions at suitable regularization scales and use the resulting equivalence to recover optimal rates and sharp Pinsker constants in the regimes specified in \Cref{subsec:high-dimensional-optimality}.

The main difficulty is that the kernel estimator applies the filter to the empirical covariance operator \(T_X\), whereas prediction error is measured through the population covariance operator \(T\).
Because these operators need not commute, the empirical estimator is not diagonal in the population eigenbasis and scalar filter identities do not directly control its bias and variance.
The comparison must also cover filter functions with less regularity.
The analytic functional calculus approach of \citet{li2024_GeneralizationError} handles analytic filters but does not cover filters with kinks.
The key technical contribution is a perturbation analysis for scale-regular Lipschitz filters that remains effective when the empirical and population covariance operators do not commute.
Scalar Lipschitz continuity alone does not ensure operator Lipschitz continuity, so direct operator-norm perturbation bounds are generally unavailable.
We overcome this obstruction through finite Schatten class estimates, using in particular the Schur multiplier criterion of \citet{condeAlonso2023_SchurMultipliersSchatten}.
The resulting comparison bounds accommodate filters with kinks and are sharp enough to preserve the leading constants in the prediction risk.
The techniques developed here may be useful for other problems in the RKHS framework and are of independent interest.

\subsection{Related work}

\subsubsection{Spectral algorithms}

Spectral regularization provides a common algorithmic formulation for kernel ridge regression, iterative regularization, and other supervised-learning procedures through a family of filter functions \citep{gerfo2008_SpectralAlgorithms}.
Classical analyses establish consistency and convergence rates for broad families of linear regularization methods, including Tikhonov regularization and Landweber iteration \citep{bauer2007_RegularizationAlgorithms}.
Optimal rate results were subsequently developed for kernel ridge regression and for spectral regularization in statistical inverse learning problems \citep{caponnetto2007_OptimalRates,blanchard2018_OptimalRates}.
Gaussian sequence models provide a population-coordinate benchmark in which a spectral algorithm becomes a shrinkage rule with an explicit bias--variance risk \citep{johnstone2017_GaussianEstimation}.
Our contribution links these two viewpoints at the level of conditional generalization error by applying the same filter in both models.

\subsubsection{Le Cam equivalence for nonparametric regression}

Le Cam's theory compares statistical experiments through the deficiency distance.
Two sequences of experiments are asymptotically equivalent if randomized transformations in both directions have vanishing approximation error.
The classical equivalence between nonparametric regression and white noise was established by \citet{brown1996_AsymptoticEquivalence}.
For one-dimensional nonparametric regression with random design, \citet{brown2002_AsymptoticEquivalence} construct explicit mappings between the regression and Gaussian white noise experiments and establish global asymptotic equivalence over Lipschitz and Sobolev classes with smoothness greater than \( 1/2 \).
Using approximation spaces, \citet{reiss2008_AsymptoticEquivalence} develops a constructive framework for fixed and random designs, extends the equivalence to multivariate periodic Sobolev classes with smoothness greater than \( d/2 \), and derives explicit bounds on the Le Cam distance.
These results permit statistical procedures to be transferred between regression and white noise experiments with asymptotically matching risks under bounded losses.
Because the resulting Le Cam bounds are additive, they do not by themselves provide a multiplicative risk expansion for a prescribed empirical spectral estimator or an explicit learning curve determined by its filter.
Our results address this algorithm-specific question by identifying the risk of a kernel spectral algorithm with the explicit risk of its sequence model counterpart.

\subsubsection{Learning curves in kernel regression}

Earlier work on Gaussian process regression derived approximations and bounds for learning curves from covariance eigenvalues \citep{sollich2002_LearningCurves}.
More recent analyses study curves that depend on the spectrum for kernel regression and wide neural networks in specific model geometries \citep{bordelon2020_SpectrumDependent} and derive learning curve bounds for kernel ridge regression under broad conditions via a Gaussian equivalent property \citep{cheng2024_ComprehensiveAnalysis}.
Exact learning curves for kernel ridge regression under power-law spectral structure were obtained by \citet{li2023_AsymptoticLearning}.
\citet{li2024_GeneralizationError} developed generalization error curves for analytic spectral algorithms under power-law decay.
\citet{velikanov2024_GeneralizationError} derived generalization error functionals for spectral algorithms in high-dimensional Gaussian and low-dimensional translation-invariant models, again under power-law assumptions.
Complementary work studies rate optimality, misspecification, and saturation for kernel and spectral algorithms \citep{zhang2024_OptimalityMisspecified,lu2024_SaturationEffects}.
Our results give a multiplicative comparison with the population sequence risk for prescribed spectral estimators, including continuous filters with kinks, without requiring power-law eigenvalue decay.

\subsubsection{Operator perturbation theory}

Double operator integrals express differences of functions of self-adjoint operators in terms of scalar divided differences, providing a way to compare empirical and population spectral estimators even when their covariance operators do not commute.
This calculus originates in the work of Birman and Solomyak and is developed in the perturbation theory surveyed by Peller \citep{birman2003_DoubleOperator,peller2016_MultipleOperator}.
A central distinction is that scalar Lipschitz continuity does not in general imply operator Lipschitz continuity \citep{aleksandrov2016_OperatorLipschitz}.
For Schatten classes \(\mathfrak S_r\), however, \citet{potapov2011_OperatorLipschitzFunctions} established Lipschitz perturbation bounds for \(1<r<\infty\), and \citet{caspers2014_BestConstants} determined the optimal order of their dependence on \(r\).
These results allow nonsmooth scalar profiles to be treated through finite Schatten norms.

Of particular relevance here is the Schur multiplier theorem of \citet{condeAlonso2023_SchurMultipliersSchatten}.
Their criterion gives complete boundedness on \(\mathfrak S_r\) from H\"ormander--Mikhlin derivative bounds on the multiplier kernel, with explicit dependence of order \(r^2/(r-1)\).
It applies to kernels beyond the divided differences of a single Lipschitz function.
In our setting, such kernels arise when comparing the filtered operators in the population prediction norm.
The one-dimensional criterion controls an auxiliary multiplier in logarithmic spectral coordinates; combined with the estimates of divided difference above, it yields the interaction bound needed for the bias comparison.
Its linear growth in \(r\ge2\) is essential to our use of Schatten interpolation, which converts these bounds into a perturbation estimate with only a logarithmic loss.
 \section{Preliminaries}
\label{sec:preliminaries}

\subsection{Reproducing kernel Hilbert spaces and kernel operators}
\label{subsec:prelim-kernel}

Let \((\mathcal{X},\mu)\) be a probability space and let \(k:\mathcal{X}\times\mathcal{X}\to\mathbb{R}\) be a measurable positive definite kernel.
Write \(\mathcal{H}\) for its associated separable RKHS and \(k_x=k(x,\cdot)\).
Throughout, we assume that
\[
  \sup_{x\in\mathcal{X}} k(x,x)\le \kappa^2.
\]
Under these assumptions, RKHS functions are measurable and the feature map \(x\mapsto k_x\) is strongly measurable.

Let \(S:\mathcal{H}\to L^2(\mu)\) be the canonical embedding, understood as the map sending an RKHS function to its \(\mu\)-equivalence class.
This map need not be injective.
The reproducing property and the diagonal bound give
\[
  \norm{Sh}_{L^2(\mu)}^2
  =
  \int_{\mathcal{X}}|h(x)|^2\dd\mu(x)
  \le
  \int_{\mathcal{X}} k(x,x)\dd\mu(x)\norm{h}_{\mathcal{H}}^2
  \le
  \kappa^2 \norm{h}_{\mathcal{H}}^2.
\]
Thus \(S\) is bounded with \(\norm{S}\le\kappa\).
Its adjoint \(S^*:L^2(\mu)\to \mathcal{H}\) is the Bochner integral map
\[
  S^*g
  =
  \int_{\mathcal{X}} k_x g(x)\dd\mu(x),
  \qquad
  g\in L^2(\mu).
\]
We further define
\[
  L=SS^*:L^2(\mu)\to L^2(\mu),
  \qquad
  T=S^*S:\mathcal{H}\to\mathcal{H}.
\]
Writing \(\mathfrak{S}_2\) for the Hilbert--Schmidt class, if \((v_r)\) is an orthonormal basis of \(\mathcal{H}\), then Tonelli's theorem and Parseval's identity give
\[
  \norm{S}_{\mathfrak{S}_2}^2
  =
  \sum_r \norm{Sv_r}_{L^2(\mu)}^2
  =
  \int_{\mathcal{X}} k(x,x)\dd\mu(x)
  \le
  \kappa^2.
\]
Consequently, \(S\) is Hilbert--Schmidt and compact, while \(T\) and \(L\) are positive, self-adjoint, trace-class operators satisfying
\[
  \Tr(T)=\Tr(L)=\norm{S}_{\mathfrak{S}_2}^2.
\]
Their integral representations are
\[
  (Lf)(x)
  =
  \int_{\mathcal{X}} k(x,x')f(x')\dd\mu(x'),
  \qquad
  Th
  =
  \int_{\mathcal{X}} k_x h(x)\dd\mu(x).
\]
Equivalently, as a trace-class Bochner integral,
\[
  T
  =
  \int_{\mathcal{X}} k_x \otimes_{\mathcal{H}} k_x \dd\mu(x)
  =
  \E\zk{k_X \otimes_{\mathcal{H}} k_X },
  \qquad X\sim\mu.
\]
Moreover,
\[
  \overline{\Ran(T)}
  =
  \overline{\Ran(S^*)}
  =
  \Ker(S)^\perp,
  \qquad
  \overline{\Ran(L)}
  =
  \overline{\Ran(S)}.
\]
The separability of \(\mathcal{H}\) allows one to choose a common full-measure set on which every element of \(\Ker(S)\) vanishes.
Hence
\[
  k_x \in\overline{\Ran(S^*)}
  \qquad
  \text{for \(\mu\)-almost every \(x\)}.
\]
The operators \(T\) and \(L\) have the same nonzero eigenvalues, including multiplicities.

The nonzero spectrum is finite or countable.
Write its distinct eigenvalues as
\[
  \mu_1>\mu_2>\cdots>0.
\]
All sums below are understood to terminate in the finite-rank case.
Let \(V_m=\operatorname{Ker}(L-\mu_m I)\) be the eigenspace of \(L\) corresponding to \(\mu_m\), \(d_m=\dim(V_m)\), and \(P_m\) be the \(L^2(\mu)\)-orthogonal projection onto \(V_m\).
The nonzero eigenvalues of \(T\) are the same \(\mu_m\), with the same multiplicities \(d_m\).
Choose an orthonormal basis \(e_{m,1},\dots,e_{m,d_m}\) of \(V_m\), so that
\[
  L e_{m,l}=\mu_m e_{m,l},
  \qquad
  l=1,\dots,d_m.
\]
The vectors
\[
  u_{m,l}
  =
  \mu_m^{-1/2} S^*e_{m,l}
\]
form an orthonormal eigenbasis of \(\overline{\Ran(T)}\), with
\[
  Tu_{m,l}=\mu_m u_{m,l},
  \qquad
  Su_{m,l}=\mu_m^{1/2} e_{m,l}.
\]
The canonical measurable representative is
\[
  e_{m,l}(x)=\mu_m^{-1/2} u_{m,l}(x).
\]
For \(\mu\)-almost every \(x\), the feature expansion
\[
  k_x
  =
  \sum_{m\ge1} \sum_{l=1}^{d_m}
  \mu_m^{1/2} e_{m,l}(x)u_{m,l}
\]
holds in \(\mathcal{H}\).
The Hilbert--Schmidt spectral theorem therefore gives the kernel expansion
\[
  k(x,x')
  = \sum_{m\ge1} \mu_m
  \sum_{l=1}^{d_m} e_{m,l}(x)e_{m,l}(x'),
\]
with convergence in \(L^2(\mu\otimes\mu)\) and equality for \(\mu\otimes\mu\)-almost every \((x,x')\).
For each \(m\), the shell sum
\(
\sum_{l=1}^{d_m} e_{m,l}(x)e_{m,l}(x')
\)
is independent of the choice of orthonormal basis of \(V_m\).
For later formulas, we flatten the multiplicities into a single sequence.
Write \( (\lambda_j)_{j \geq 1} \) for the descending eigenvalues counting multiplicity and \( (e_j)_{j \geq 1} \) for the corresponding eigenfunctions.

For a target \(f^*\in\overline{\Ran(L)}\), write
\[
  f^*
  =
  \sum_{m\ge1} \sum_{l=1}^{d_m} f_{m,l} e_{m,l},
  \qquad
  \bar{f}_m^2
  \coloneqq
  \sum_{l=1}^{d_m} f_{m,l}^2
  =
  \norm{P_m f^*}_{L^2(\mu)}^2
\]
on this closed range.

For every bounded Borel function \(r\) on \([0,\kappa^2]\),
\begin{equation}
  \label{eq:spectral-intertwining}
  S^*r(L)=r(T)S^*,
  \qquad
  Sr(T)=r(L)S,
  \qquad
  Sr(T)S^*=Lr(L).
\end{equation}
Thus population-coordinate statements are moved between \(L^2(\mu)\) and the RKHS only through the canonical embedding \(S\) and its adjoint \(S^*\).
In particular, an RKHS estimator is evaluated in the prediction norm by applying \(S\), not by identifying RKHS spectral coefficients with \(L^2(\mu)\) coefficients.

\subsection{RKHS regression and spectral algorithms}
\label{subsec:spectral-algorithms}

The RKHS regression model consists of independent pairs
\[
  y_i=f^*(x_i)+\varepsilon_i,
  \qquad
  i=1,\dots,n,
\]
where \(x_i \sim\mu\), \(\E[\varepsilon_i \mid x_i]=0\), and \(\E[\varepsilon_i^2 \mid x_i]=\sigma^2>0\).
Here \(f^*\) denotes a fixed measurable representative of the regression function.
Write
\[
  Z=((x_i,y_i))_{i=1}^n,
  \qquad
  X=(x_1,\dots,x_n),
  \qquad
  y=(y_1,\dots,y_n).
\]
Equip \(\mathbb{R}^n\) with the empirical inner product
\[
  \langle a,b\rangle_n
  =
  \frac{1}{n}\sum_{i=1}^n a_i b_i,
\]
and denote the resulting Hilbert space by \(\mathbb{R}_X^n\).
Define the sampling operator \(S_X:\mathcal{H}\to\mathbb{R}_X^n\) and its adjoint by
\[
  (S_X h)_i=h(x_i),
  \qquad
  S_X^*a=\frac{1}{n}\sum_{i=1}^n a_i k_{x_i}.
\]
The empirical covariance operator and sample representer are
\[
  T_X=S_X^*S_X
  =
  \frac{1}{n}\sum_{i=1}^n k_{x_i} \otimes_{\mathcal{H}} k_{x_i}
  \quad\text{on }\mathcal{H},
  \qquad
  \widehat{g}_Z=S_X^*y
  =
  \frac{1}{n}\sum_{i=1}^n y_i k_{x_i}.
\]

The kernel ridge regression estimator is
\begin{align}
  \widehat{f}_{\lambda}^{\mathsf{KR}}=(T_{X}+\lambda)^{-1}\widehat{g}_{Z}.
\end{align}

A spectral algorithm is specified by a regularization function \(\varphi_\lambda:[0,\kappa^2]\to[0,\infty)\) and returns
\[
  \widehat{f}_\lambda=\varphi_\lambda(T_X)\widehat{g}_{Z},
\]
where \(\lambda>0\) is a regularization parameter and \( \varphi_\lambda(T_X) \) is defined by the spectral functional calculus.

It is useful to separate the regularized inverse from the shrinkage profile:
\[
  q_\lambda(t)=t\varphi_\lambda(t),
  \qquad
  \psi_\lambda(t)=1-q_\lambda(t).
\]
We focus on filters satisfying \(0\le q_\lambda(t)\le1\).
This covers the filters of interest in this manuscript, including smoothed spectral cutoff, hard spectral cutoff, and spectral clipping.

For example, kernel ridge regression, gradient flow at time \(\lambda^{-1}\), and spectral soft thresholding correspond respectively to
\begin{align*}
  \varphi_\lambda^{\mathsf{KR}}(t)
  &=\frac{1}{t+\lambda},
  &
  q_\lambda^{\mathsf{KR}}(t)
  &=\frac{t}{t+\lambda},
  &
  \psi_\lambda^{\mathsf{KR}}(t)
  &=\frac{\lambda}{t+\lambda},
  \\
  \varphi_\lambda^{\mathsf{GF}}(t)
  &=\frac{1-e^{-t/\lambda}}{t},
  &
  q_\lambda^{\mathsf{GF}}(t)
  &=1-e^{-t/\lambda},
  &
  \psi_\lambda^{\mathsf{GF}}(t)
  &=e^{-t/\lambda},
  \\
  \varphi_\lambda^{\mathsf{ST}}(t)
  &=\frac{(t-\lambda)_+}{t^2},
  &
  q_\lambda^{\mathsf{ST}}(t)
  &=\xk{1-\frac{\lambda}{t}}_+,
  &
  \psi_\lambda^{\mathsf{ST}}(t)
  &=\min\xk{\frac{\lambda}{t},1}.
\end{align*}
Here \((u)_+=\max\{u,0\}\), and the removable or continuous values at \(t=0\) are \(\varphi_\lambda^{\mathsf{GF}}(0)=\lambda^{-1}\), \(\varphi_\lambda^{\mathsf{ST}}(0)=0\), \(q_\lambda^{\mathsf{ST}}(0)=0\), and \(\psi_\lambda^{\mathsf{ST}}(0)=1\).

\subsection{Gaussian sequence model and spectral algorithms}
\label{subsec:sequence-model}

Gaussian sequence models are simple but fundamental tools in nonparametric statistics.
A useful way to interpret a spectral algorithm is as a coordinatewise shrinkage rule in the population eigenbasis \citep{johnstone2017_GaussianEstimation,gerfo2008_SpectralAlgorithms}.
On the closed span of the nonzero eigenspaces of \(L\), write
\[
  f^*
  =
  \sum_{j\ge1} f_j^* e_j,
  \qquad
  f_j^*
  =
  \langle f^*,e_j \rangle_{L^2(\mu)}.
\]
Replacing the empirical quantities \(T_X\) and \(\widehat{g}_Z\) by their population counterparts \(T\) and \(S^*f^*\) gives the idealized estimator
\[
  \widetilde{f}_\lambda^{\mathsf{pop}}
  =
  \varphi_\lambda(T)S^*f^*.
\]
By \cref{eq:spectral-intertwining}, its prediction function satisfies
\[
  S\widetilde{f}_\lambda^{\mathsf{pop}}
  =
  L\varphi_\lambda(L)f^*
  =
  q_\lambda(L)f^*,
\]
and hence
\[
  \langle e_j,S\widetilde{f}_\lambda^{\mathsf{pop}}\rangle_{L^2(\mu)}
  =
  q_\lambda(\lambda_j)f_j^*,
  \qquad
  j\ge1.
\]
Thus \(q_\lambda\) acts as a shrinkage profile on the population coefficients, while \(\psi_\lambda=1-q_\lambda\) is the corresponding residual profile.

To retain the stochastic error while removing the random-design perturbation, consider the sequence observations
\[
  z_j=f_j^*+\frac{\sigma}{\sqrt n}\xi_j,
  \quad
  \xi_j \stackrel{\mathrm{i.i.d.}}{\sim}\mathcal{N}(0,1)
  \qquad
  j\ge1.
\]
Applying the same shrinkage profile gives the sequence estimator
\begin{equation}
  \widehat{f}_{\lambda,j}^{\mathsf{seq}}
  =
  q_\lambda(\lambda_j)z_j,
  \qquad
  \widehat{\bm{f}}_{\lambda}^{\mathsf{seq}} = \xk{\widehat{f}_{\lambda,j}^{\mathsf{seq}}}_{j \geq 1}.
\end{equation}
Its quadratic risk is
\begin{equation}
  \label{eq:sequence-risk}
  \mathcal{E}_n^{\mathsf{seq}}(q_\lambda;f^*)
  \coloneqq
  \E \norm{\widehat{\bm{f}}_{\lambda}^{\mathsf{seq}} - \bm{f}^*}_{\ell^2}^2
  =
  \sum_{j\ge1}
  \E_\xi \zk{
           \xk{
             \widehat{f}_{\lambda,j}^{\mathsf{seq}}-f_j^*
           }^2
  }
  =
  \mathcal{B}_n(q_\lambda;f^*)
  +
  \mathcal{V}_n(q_\lambda),
\end{equation}
where the deterministic bias and variance are given by
\begin{equation}
  \label{eq:sequence-bias}
  \mathcal{B}_n(q_\lambda;f^*)
  =
  \sum_{j\ge1} \psi_\lambda(\lambda_j)^2 (f_j^*)^2
  =
  \norm{\psi_\lambda(L)f^*}_{L^2(\mu)}^2
\end{equation}
and
\begin{equation}
  \label{eq:sequence-variance}
  \mathcal{V}_n(q_\lambda)
  =
  \frac{\sigma^2}{n}\mathcal{N}_q(\lambda),
  \qquad
  \mathcal{N}_q(\lambda)
  \coloneqq\sum_{j\ge1}q_\lambda(\lambda_j)^2.
\end{equation}
The first term is the deterministic approximation error of population shrinkage, whereas the second is the accumulated coordinatewise noise.
The sequence model therefore retains the population eigen-coordinates, target coefficients, noise level, and shrinkage rule while removing the random perturbation of the empirical covariance operator.

\subsection{Double operator integrals}
\label{subsec:doi-background}

Double operator integrals~\citep{birman2003_DoubleOperator} compare functions of noncommuting self-adjoint operators and provide a functional-calculus analogue of the scalar divided-difference identity \(f(a)-f(b)=f^{[1]}(a,b)(a-b)\).
For the real RKHS operators above, we use the complex theory through their canonical complexifications; the operator, Schatten, and trace quantities below are unchanged.
The perturbation formula we need goes back to Birman--Solomyak and is surveyed in modern form by Peller \citep{birman2003_DoubleOperator,peller2016_MultipleOperator}.

Let \(A\) and \(B\) be bounded self-adjoint operators on a complex Hilbert space, with spectral measures \(E_A\) and \(E_B\) supported on intervals \(I\) and \(J\).
For a kernel \(m:I\times J\to\mathbb{C}\), the formal double operator integral is the linear map
\[
  H
  \mapsto
  \operatorname{DOI}_{m}^{A,B}(H)
  \coloneqq
  \iint_{I\times J}
  m(x,y) \dd E_A(x)H \dd E_B(y).
\]
For every bounded Borel kernel \(m\), this integral defines a bounded operator on \(\mathfrak{S}_2\), with
\[
  \norm{\operatorname{DOI}_{m}^{A,B}(H)}_{\mathfrak{S}_2}
  \le
  \sup_{(x,y)\in I\times J}|m(x,y)|\,
  \norm{H}_{\mathfrak{S}_2}.
\]
A DOI kernel is called a multiplier on a Schatten class if its DOI map acts boundedly on that class.
Every bounded Borel kernel is a multiplier on \(\mathfrak{S}_2\), whereas boundedness alone does not suffice on \(\mathfrak{S}_r\) for \(r\ne2\).
The multiplier spaces and bounds used in our perturbation analysis are introduced in \Cref{sec:lipschitz-schatten}.

The following perturbation formula is useful; see \citet[Theorem 1.2.1 and Corollary 1.2.2]{peller2016_MultipleOperator}.

\begin{theorem}[Birman--Solomyak perturbation formula]
  \label{thm:birman-solomyak-doi}
  Let \(A\) and \(B\) be bounded self-adjoint operators with spectra in a compact interval \(I\), and suppose that \(A-B\in\mathfrak{S}_2\).
  Let \(f:I\to\mathbb{R}\) be Lipschitz and define the zero-diagonal divided difference
  \[
    f^{[1]}_0(x,y)
    =
    \begin{cases}
      \dfrac{f(x)-f(y)}{x-y}, & x\ne y,\\[1.2ex]
      0, & x=y.
    \end{cases}
  \]
  Then \(f(A)-f(B)\in\mathfrak{S}_2\) and
  \[
    f(A)-f(B)
    =
    \operatorname{DOI}_{f^{[1]}_0}^{A,B}(A-B).
  \]
  Moreover,
  \[
    \norm{f(A)-f(B)}_{\mathfrak{S}_2}
    \le
    \operatorname{Lip}_I(f)\norm{A-B}_{\mathfrak{S}_2},
    \qquad
    \operatorname{Lip}_I(f)
    \coloneqq\sup_{\substack{x,y\in I\\x\ne y}}
    \frac{|f(x)-f(y)|}{|x-y|}.
  \]
\end{theorem}

The interval formulation follows by extending \(f\) to a Lipschitz function on \(\mathbb{R}\).
Changing the divided difference to any bounded Borel value on the diagonal leaves its DOI applied to \(A-B\) unchanged.
In particular, no differentiability assumption is needed.
This Hilbert--Schmidt estimate does not imply an operator-norm Lipschitz estimate, which need not hold for scalar Lipschitz functions \citep{aleksandrov2016_OperatorLipschitz}.
 \section{Main Results}
\label{sec:main-results}

This section presents conditional and expectation risk equivalence for a fixed target.
We first identify the risks to be compared, then state the filter and scale assumptions for their equivalence.

\subsection{Risk comparison}

\begin{assumption}[Model]
  \label{ass:basic-model}
  In the kernel and RKHS setting of \Cref{sec:preliminaries}, assume
  \[
    \sup_{x\in\mathcal{X}}k(x,x)\le\kappa^2<\infty.
  \]
  Suppose \(f^*\in\overline{\Ran(L)}\), the sample pairs \((x_i,y_i)\), \(i=1,\dots,n\), are independent, with \(x_i\sim\mu\) and
  \[
    y_i=f^*(x_i)+\varepsilon_i,
    \qquad
    \E[\varepsilon_i\mid x_i]=0,
    \qquad
    \E[\varepsilon_i^2\mid x_i]=\sigma^2\in(0,\infty).
  \]
\end{assumption}

Independence of the sample pairs implies that, conditional on \(X\), the noise variables are independent and satisfy
\[
  \E[\varepsilon_i\mid X]=0,
  \qquad
  \E[\varepsilon_i\varepsilon_j\mid X]
  =\sigma^2\mathbf{1}_{\{i=j\}}.
\]
Consequently,
\begin{equation}
  \label{eq:conditional-mean}
  \widetilde{g}_X
  \coloneqq \E[\widehat{g}_Z \mid X]
  = S_X^*f_X^*
  = \frac{1}{n}\sum_{i=1}^n f^*(x_i)k_{x_i},
\end{equation}
where \(f_X^*=(f^*(x_1),\dots,f^*(x_n))\).

We begin with the conditional risk.
Under \Cref{ass:basic-model}, it admits the exact decomposition
\begin{equation}
  \label{eq:np-risk}
  \mathcal{E}_n^{\mathsf{np}}(q_\lambda;f^*\mid X)
  \coloneqq
  \E \zk{
    \norm{S\widehat{f}_\lambda-f^*}_{L^2(\mu)}^2
    \,\middle|\, X
  }
  =
  \mathsf{B}_X(q_\lambda)
  +
  \mathsf{V}_X(\varphi_\lambda),
\end{equation}
where
\begin{equation}
  \label{eq:empirical-bias}
  \mathsf{B}_X(q_\lambda)
  =
  \norm{S\varphi_\lambda(T_X)\widetilde{g}_X-f^*}_{L^2(\mu)}^2
\end{equation}
and
\begin{equation}
  \label{eq:empirical-variance}
    \mathsf{V}_X(\varphi_\lambda)
    =\frac{\sigma^2}{n^2}
      \sum_{i=1}^n
      \norm{S\varphi_\lambda(T_X)k_{x_i}}_{L^2(\mu)}^2
    =\frac{\sigma^2}{n}
      \Tr_{\mathcal{H}}\xk{T\varphi_\lambda(T_X)^2T_X}.
\end{equation}
The trace identity follows from \(\norm{Sh}_{L^2(\mu)}^2=\langle h,Th\rangle_{\mathcal{H}}\).

The risk in expectation is obtained by averaging over the design:
\begin{equation}
  \label{eq:risk-in-expectation}
  \mathcal{E}_n^{\mathsf{np}}(q_\lambda;f)
  \coloneqq
  \E_X \mathcal{E}_n^{\mathsf{np}}(q_\lambda;f\mid X)
  =
  \E_{X,\varepsilon}
  \norm{S\widehat{f}_\lambda-f}_{L^2(\mu)}^2.
\end{equation}
The sequence risk in \cref{eq:sequence-risk} serves as the population benchmark.
Our aim is to replace the empirical bias and variance in \cref{eq:empirical-bias,eq:empirical-variance} by their population counterparts, with an error negligible relative to the total risk.
All asymptotic statements are taken as \(n\to\infty\) along a prescribed deterministic sequence \(\lambda=\lambda(n)\downarrow0\).

\subsection{Filter assumptions}
\label{subsec:lipschitz-main-result}

The finite-Schatten argument requires uniform Lipschitz control after the spectrum is expressed in logarithmic coordinates.
We state this direct condition first; a simple sufficient condition for standard filter families follows afterward.

\begin{assumption}[Scale-regular Lipschitz filter]
  \label{ass:lipschitz-filter}
  For each \(\lambda\), let \(\varphi_\lambda:[0,\kappa^2]\to[0,\infty)\).
  Assume
  \[
    0\le q_\lambda(t)\le1,
    \qquad
    0\le t\le\kappa^2.
  \]
  Define the dimensionless profiles
  \[
    \phi_\lambda^\sharp(u)
    =
    \lambda\varphi_\lambda(\lambda u),
    \qquad
    \psi_\lambda^\sharp(u)
    =
    \psi_\lambda(\lambda u),
    \qquad
    0\le u\le\kappa^2/\lambda.
  \]
  There exist a fixed residual order \(\rho_\psi>0\) and a constant \(A<\infty\), both independent of \(\lambda\), with the following property.
  On the interval in logarithmic spectral coordinates
  \[
    J_\lambda
    =
    [0,\log(1+\kappa^2/\lambda)],
  \]
  define
  \[
    F_\lambda(t)
    =
    e^t \phi_\lambda^\sharp(e^t-1),
    \qquad
    P_{\lambda,\rho_\psi}(t)
    =
    e^{\rho_\psi t} \psi_\lambda^\sharp(e^t-1).
  \]
  These transformed profiles satisfy
  \begin{equation}
    \label{eq:scale-lipschitz-size}
    |F_\lambda(t)|+|P_{\lambda,\rho_\psi}(t)|
    \le
    A
  \end{equation}
  and, for all \(t,s\in J_\lambda\),
  \begin{equation}
    \label{eq:scale-lipschitz-difference}
    |F_\lambda(t)-F_\lambda(s)|
    +
    |P_{\lambda,\rho_\psi}(t)-P_{\lambda,\rho_\psi}(s)|
    \le
    A|t-s|.
  \end{equation}
\end{assumption}

The transformed size bound yields the regularized inverse estimate and a qualification of at least \(\rho_\psi\):
\[
  \sup_{0\le t\le\kappa^2}
  (t+\lambda)\varphi_\lambda(t)
  \le A,
  \qquad
  \sup_{0\le t\le\kappa^2}
  t^\tau|\psi_\lambda(t)|
  \le
  A\lambda^\tau,
  \quad
  0\le\tau\le\rho_\psi.
\]
The residual order \(\rho_\psi\) is a property of the filter family and is independent of the source exponent.
The formulation is scalar and does not impose operator Lipschitzness.

The following derivative bounds provide a convenient sufficient condition for \Cref{ass:lipschitz-filter}.

\begin{condition}
  \label{cond:smooth-filter}
  For each regularization scale \(\lambda\), let \(\varphi_\lambda:[0,\kappa^2]\to[0,\infty)\) be the regularization function.
  Define the shrinkage and residual profiles by
  \begin{equation}
    \label{eq:smooth-filter-profiles}
    q_\lambda(t)=t\varphi_\lambda(t),
    \qquad
    \psi_\lambda(t)=1-q_\lambda(t),
    \qquad 0\le t\le\kappa^2.
  \end{equation}
  At \(t=0\), the expression \(\varphi_\lambda=q_\lambda/t\) is interpreted through its removable extension.
  Assume that \(0\le q_\lambda \le1\), that \(\varphi_\lambda\) and
  \(\psi_\lambda\) are absolutely continuous, and that there are constants
  \(B<\infty\) and \(\rho_\psi>0\), independent of \(\lambda\), for which
  \[
    \begin{aligned}
      (t+\lambda)\varphi_\lambda(t)
      &\le B,
      &
      |\varphi_\lambda'(t)|
      &\le B(t+\lambda)^{-2},
      \\
      |\psi_\lambda(t)|
      &\le B\xk{\frac{\lambda}{t+\lambda}}^{\rho_\psi},
      &
      |\psi_\lambda'(t)|
      &\le
      B\lambda^{\rho_\psi}(t+\lambda)^{-\rho_\psi-1}
    \end{aligned}
  \]
  almost everywhere on \([0,\kappa^2]\).
\end{condition}

For the dimensionless profiles in \Cref{ass:lipschitz-filter}, these bounds give
\[
  \begin{aligned}
    |\phi_\lambda^\sharp(u)|
    &\le B(1+u)^{-1},
    &
    |(\phi_\lambda^\sharp)'(u)|
    &\le B(1+u)^{-2},
    \\
    |\psi_\lambda^\sharp(u)|
    &\le B(1+u)^{-\rho_\psi},
    &
    |(\psi_\lambda^\sharp)'(u)|
    &\le B(1+u)^{-\rho_\psi-1}.
  \end{aligned}
\]
By the chain rule, the log-spectral profiles \(F_\lambda\) and
\(P_{\lambda,\rho_\psi}\) are uniformly bounded and Lipschitz, with constants
independent of \(\lambda\).
Thus \Cref{cond:smooth-filter} implies \Cref{ass:lipschitz-filter}.
For concrete filter functions, \Cref{cond:smooth-filter} is more convenient to verify.

For every fixed \(\rho_\psi>0\), the residual family
\[
  \psi_\lambda(x)
  =
  \xk{\frac{\lambda}{x+\lambda}}^{\rho_\psi},
  \qquad
  q_\lambda(x)=1-\psi_\lambda(x),
\]
for which
\[
  P_{\lambda,\rho_\psi}(t)=1,
  \qquad
  F_\lambda(t)
  =
  \frac{1-e^{-\rho_\psi t}}{1-e^{-t}},
\]
with the removable value \(F_\lambda(0)=\rho_\psi\), satisfies \Cref{ass:lipschitz-filter}.
Residual orders below one half are therefore allowed when the balance condition in \cref{eq:leverage-fixed-target-scale} is satisfied.

\subsection{Risk equivalence}
\label{subsec:main-results}

For \(\lambda>0\), define the effective dimension by
\[
  \mathcal{N}_1(\lambda)
  \coloneqq
  \sum_{j\ge1} \frac{\lambda_j}{\lambda_j+\lambda}.
\]
The maximal regularized leverage is
\[
  \mathcal{L}_\lambda
  \coloneqq
  \operatorname*{ess\,sup}_{x\in\mathcal{X}}
  \norm{(T+\lambda I)^{-1/2}k_x}_{\mathcal{H}}^2.
\]
It is finite under the kernel condition, and
\begin{equation}
  \label{eq:leverage-scale-bounds}
  \mathcal{N}_1(\lambda)
  \le
  \mathcal{L}_\lambda
  \le
  \frac{\kappa^2}{\lambda}.
\end{equation}
The effective dimension is the average regularized leverage, whereas \(\mathcal{L}_\lambda\) controls its largest value.

Set \(\ell_\lambda=\log\xk{e+\mathcal{N}_1(\lambda)}\) and define the perturbation scale
\begin{equation}
  \label{eq:lipschitz-schatten-scale}
  \gamma_\lambda
  =
  \sqrt{
    \frac{\mathcal{L}_\lambda\ell_\lambda}{n}
  }
  \log\xk{
    e+
    \sqrt{
      \frac{\mathcal{N}_1(\lambda)}{\ell_\lambda}
    }
  }.
\end{equation}
For the residual order \(\rho_\psi\) in \Cref{ass:lipschitz-filter}, set
\(\alpha_\psi=\min\{2\rho_\psi,1\}\).
For a fixed target, define the weighted spectral tail
\begin{equation}
  \label{eq:fixed-target-spectral-tail}
  \mathcal{S}_{f^*}(\lambda)
  =\sum_{j\ge1}
  \xk{\frac{\lambda}{\lambda_j+\lambda}}^{\alpha_\psi}
  (f_j^*)^2
  =\norm{
    \xk{\frac{\lambda}{L+\lambda I}}^{\alpha_\psi/2}f^*
  }_{L^2(\mu)}^2.
\end{equation}
The weights emphasize target coefficients at small population eigenvalues.
Together, \(\mathcal{S}_{f^*}(\lambda)\) and \(\gamma_\lambda\) control the perturbation of the bias caused by the random design.

\begin{assumption}[Scale conditions]
  \label{ass:conditional-equivalence-scale}
  Along the prescribed sequence \(\lambda=\lambda(n)\downarrow0\), assume:
  \begin{itemize}
    \item \(\mathcal{N}_q(\lambda)>0\) for all sufficiently large \(n\), and the variance scales satisfy
    \begin{equation}
      \label{eq:leverage-variance-compatibility}
      \frac{\mathcal{N}_1(\lambda)}{\mathcal{N}_q(\lambda)}
      \sqrt{
        \frac{\mathcal{L}_\lambda\ell_\lambda}{n}
      }
      =
      o(1);
    \end{equation}

    \item the target satisfies the spectral balance condition
    \begin{equation}
      \label{eq:leverage-fixed-target-scale}
      \mathcal{S}_{f^*}(\lambda)\gamma_\lambda^2
      =o\xk{\frac{\mathcal{N}_q(\lambda)}{n}}.
    \end{equation}
  \end{itemize}
\end{assumption}

Condition \cref{eq:leverage-variance-compatibility} makes the variance comparison error negligible relative to the population variance.
Condition \cref{eq:leverage-fixed-target-scale} requires the target-dependent perturbation term to be negligible on the same scale.

\begin{theorem}
  \label{thm:conditional-risk-equivalence}
  Under the structural assumptions
  \Cref{ass:basic-model}, the filter assumption
  \Cref{ass:lipschitz-filter}, and the scale assumption
  \Cref{ass:conditional-equivalence-scale},
  \begin{equation}
    \label{eq:conditional-risk-equivalence}
    \mathcal{E}_n^{\mathsf{np}}(q_\lambda;f^*\mid X)
    =
    (1+o_{\mathbb{P}}(1))
    \mathcal{E}_n^{\mathsf{seq}}(q_\lambda;f^*).
  \end{equation}
  The \(o_{\mathbb{P}}(1)\) remainder is uniform over every family of problem instances for which the assumptions of the theorem hold uniformly in the sense of \Cref{rem:triangular-array-equivalence}.
\end{theorem}

To obtain the expectation result, we impose the following condition on the population scale.
It excludes rare, nearly singular designs whose conditional risks are too large to average.

\begin{assumption}[Additional scale condition for expectation]
  \label{ass:risk-equivalence-expectation}
  Assume that the prescribed deterministic regularization sequence satisfies
  \begin{equation}
    \label{eq:leverage-expectation-scale}
    \frac{\mathcal{L}_\lambda}{n}
    \log\xk{
      e+\frac{n}{\lambda\mathcal{N}_q(\lambda)}
    }
    =
    o(1).
  \end{equation}
\end{assumption}

\begin{theorem}
  \label{thm:risk-equivalence-expectation}
  Suppose the assumptions of \Cref{thm:conditional-risk-equivalence}
  hold along a prescribed deterministic sequence \(\lambda=\lambda(n)\), and
  that \Cref{ass:risk-equivalence-expectation} also holds.
  Then
  \begin{equation}
    \label{eq:risk-equivalence-expectation}
    \mathcal{E}_n^{\mathsf{np}}(q_\lambda;f^*)
    =
    \xk{1+o(1)}
    \mathcal{E}_n^{\mathsf{seq}}(q_\lambda;f^*).
  \end{equation}
  The \(o(1)\) remainder is uniform over every family of problem instances for which the assumptions of the theorem hold uniformly in the sense of \Cref{rem:triangular-array-equivalence}.
\end{theorem}

\begin{remark}[Uniformity and triangular arrays]
  \label{rem:triangular-array-equivalence}
  For a family \(\Theta_n\) of problem instances, we say that the assumptions hold uniformly if their nonasymptotic bounds hold throughout the family with constants independent of the instance, the noise variances are uniformly bounded away from zero, and the supremum over \(\Theta_n\) of every deterministic quantity stated as \(o(1)\) converges to zero.
  A random remainder \(Z_{n,\theta}=o_{\mathbb{P}}(1)\) uniformly over \(\theta\in\Theta_n\) means that, for every \(\epsilon>0\),
  \[
    \sup_{\theta\in\Theta_n}
    \mathbb{P}_\theta\dk{
      \abs{Z_{n,\theta}}>\epsilon
    }
    \longrightarrow0.
  \]

  The conditional conclusion in \Cref{thm:conditional-risk-equivalence}
  remains valid if the sampling distribution \(\mu\), kernel \(k\), and target
  \(f^*\) vary with \(n\), provided the assumptions hold uniformly.
  The expectation conclusion also remains valid when
  \Cref{ass:risk-equivalence-expectation} holds uniformly.
  This extension covers regimes in which the dimension grows with the sample size.
\end{remark}

\subsection{Verifying the scale conditions}

To verify the scale assumptions, fix \(s>0\) and \(R\in(0,\infty)\), and define the source ball
\[
  \mathcal{F}_s(R)
  \coloneqq
  \dk{
    L^{s/2} h:
    h\in\overline{\Ran(L)},
    \ \norm{h}_{L^2(\mu)} \le R
  }.
\]

\subsubsection{Polynomial eigenvalue decay}
\label{sec:polynomial-scale-verification}
Suppose \Cref{ass:lipschitz-filter} holds, \(\lambda_j\asymp j^{-\beta}\) for some \(\beta>1\), and \(f^*\in\mathcal{F}_s(R)\) for fixed \(s,R>0\).
Suppose in addition that \(\mathcal{L}_\lambda\lesssim\mathcal{N}_1(\lambda)\).
If
\[
  \lambda\asymp n^{-\theta},
  \qquad 0<\theta<\beta,
\]
then all conditions in \Cref{ass:conditional-equivalence-scale} and \Cref{ass:risk-equivalence-expectation} hold.
Hence, together with the structural assumptions, risk equivalence holds.
See \Cref{prop:polynomial-scale-verification} for the verification.

If no leverage control beyond the universal bound \(\mathcal{L}_\lambda\le\kappa^2/\lambda\) in \cref{eq:leverage-scale-bounds} is available,
the polynomial conclusion still holds in the smaller region
\[
  \lambda\asymp n^{-\theta},
  \qquad
  0<\theta<1,
  \qquad
  \min\{s,\alpha_\psi\}>1-1/\beta.
\]
See \Cref{prop:polynomial-universal-leverage-verification} for the verification.

\subsubsection{Exponential eigenvalue decay}
\label{sec:exponential-scale-verification}
Suppose \Cref{ass:lipschitz-filter} holds,
\(
  \lambda_j\asymp \exp(-c_0j^\nu)
\)
for some \(c_0,\nu>0\), and \(f^*\in\mathcal{F}_s(R)\) for fixed
\(s,R>0\).
Suppose in addition that
\(\mathcal{L}_\lambda\lesssim\mathcal{N}_1(\lambda)\).
If \(\lambda\asymp n^{-\theta}\) for any \(\theta>0\), then all conditions in
\Cref{ass:conditional-equivalence-scale} and
\Cref{ass:risk-equivalence-expectation} hold.
Hence, together with the structural assumptions, risk equivalence holds.
See \Cref{prop:exponential-scale-verification} for the verification.

\subsubsection{High-dimensional spherical kernels}
\label{sec:high-dimensional-scale-verification}
Consider the spherical setting of \citet{zhang2025_OptimalRates}, with
\(x_i\) uniformly distributed on \(\mathbb S^{d-1}\), \(d\to\infty\), and
\[
  n\asymp d^\gamma,
  \qquad
  k_d(x,y)=\Phi(\langle x,y\rangle).
\]
Here \(\Phi\in C^\infty([-1,1])\) is fixed independently of \(d\) and has the expansion
\[
  \Phi(t)=\sum_{j\ge0}a_jt^j,
  \qquad a_j>0\quad(j\ge0).
\]
Thus the kernel bound \(\kappa^2=\Phi(1)\) is independent of \(d\).
Assume \Cref{ass:basic-model,ass:lipschitz-filter} hold with a fixed positive noise variance and filter constants independent of \(d\).
Let \(f_d^*=L_d^{s/2}h_d\), where \(L_d\) is the population integral operator and \(\norm{h_d}_{L^2(\mu_d)}\le R\) for fixed \(s,R>0\).
For any fixed integer \(m\) satisfying \(1\le m<\gamma\), choose
\[
  \lambda=c\,d^{-m},
\]
where \(c>0\) is a sufficiently small fixed constant depending only on \(m\), \(\Phi\), and the filter constants.
Then \Cref{ass:conditional-equivalence-scale,ass:risk-equivalence-expectation} hold uniformly over this source ball, so both conditional and expectation risk equivalence follow.
In particular, \(m=1\) gives an admissible scale for every fixed \(\gamma>1\) and \(s>0\).
See \Cref{prop:high-dimensional-scale-verification} for the verification.
\Cref{subsec:high-dimensional-optimality} shows that suitable filters at these integer scales attain the minimax rate and the Pinsker constant when \(\gamma\ge s+1\).
 \section{Applications}
\label{sec:applications}

We now give several applications of our main results.

\subsection{Examples of spectral filters}
\label{subsec:filter-examples}

We first provide several examples of spectral filters satisfying our assumptions.
Throughout this subsection, implicit constants may depend on fixed algorithmic parameters, but not on \(\lambda\).
These parameters include the exponent in generalized ridge, the iteration order \(\nu\), the step size stability margin, the cutoff profile, and \(\kappa\).

\subsubsection{Kernel ridge regression}
\label{subsubsec:krr-filter-example}

Kernel ridge regression uses
\begin{equation}
  \label{eq:krr-filter}
  \varphi_\lambda^{\mathsf{KR}}(t)=\frac{1}{t+\lambda},
  \qquad
  q_\lambda^{\mathsf{KR}}(t)=\frac{t}{t+\lambda},
  \qquad
  \psi_\lambda^{\mathsf{KR}}(t)=\frac{\lambda}{t+\lambda}.
\end{equation}
With \(u=t/\lambda\), the dimensionless profiles are
\[
  \phi_\lambda^\sharp(u)=\psi_\lambda^\sharp(u)=(1+u)^{-1}.
\]
The shrinkage lies in \([0,1]\), and these profiles, together with their first
derivatives, satisfy the four bounds in
\Cref{cond:smooth-filter} with \(\rho_\psi=1\).
Thus kernel ridge regression satisfies that elementary condition.

\subsubsection{Generalized ridge regression}
\label{subsubsec:generalized-ridge}

Fix \(p\ge2\).
The generalized ridge filter is
\begin{equation}
  \label{eq:generalized-ridge-filter}
  \varphi_{\lambda,p}^{\mathsf{GR}}(t)
  =
  \frac{t^{p-1}}{t^p+\lambda^p},
  \qquad
  q_{\lambda,p}^{\mathsf{GR}}(t)
  =
  \frac{t^p}{t^p+\lambda^p},
  \qquad
  \psi_{\lambda,p}^{\mathsf{GR}}(t)
  =
  \frac{\lambda^p}{t^p+\lambda^p}.
\end{equation}
The endpoint \(p=1\) recovers kernel ridge regression in
\Cref{eq:krr-filter}.
With \(u=t/\lambda\), the dimensionless profiles are
\[
  \phi_{\lambda,p}^\sharp(u)
  =
  \frac{u^{p-1}}{1+u^p},
  \qquad
  \psi_{\lambda,p}^\sharp(u)
  =
  \frac{1}{1+u^p}.
\]
They satisfy
\[
  |\phi_{\lambda,p}^\sharp(u)|
  \lesssim
  (1+u)^{-1},
  \qquad
  |\psi_{\lambda,p}^\sharp(u)|
  \lesssim
  (1+u)^{-p}.
\]
Moreover,
\[
  (\phi_{\lambda,p}^\sharp)'(u)
  =
  \frac{(p-1)u^{p-2}-u^{2p-2}}{(1+u^p)^2},
  \qquad
  (\psi_{\lambda,p}^\sharp)'(u)
  =
  -\frac{p u^{p-1}}{(1+u^p)^2}.
\]
Since \(p\ge2\), splitting into \(0\le u\le1\) and \(u\ge1\) gives
\[
  |(\phi_{\lambda,p}^\sharp)'(u)|
  \lesssim
  (1+u)^{-2},
  \qquad
  |(\psi_{\lambda,p}^\sharp)'(u)|
  \lesssim
  (1+u)^{-p-1}.
\]
The shrinkage lies in \([0,1]\), and the value
\(\varphi_{\lambda,p}^{\mathsf{GR}}(0)=0\) supplies the removable extension.
Rescaling the preceding bounds from \(u=t/\lambda\) to \(t\) verifies the four bounds in
\Cref{cond:smooth-filter} with residual order \(\rho_\psi=p\).
Hence generalized ridge regression satisfies the elementary scale-regular filter condition for every fixed \(p\ge2\).

\subsubsection{Iterated kernel ridge regression}
\label{subsubsec:iterated-krr-filter-example}

Fix an integer \(\nu\ge1\).
The \(\nu\)-fold iterated kernel ridge filter is
\begin{equation}
  \label{eq:iterated-krr-filter}
  \psi_\lambda^{\mathsf{IKR},\nu}(t)
  =
  \xk{\frac{\lambda}{t+\lambda}}^\nu,
  \qquad
  q_\lambda^{\mathsf{IKR},\nu}(t)=1-\psi_\lambda^{\mathsf{IKR},\nu}(t),
\end{equation}
and
\begin{equation}
  \label{eq:iterated-krr-phi}
  \varphi_\lambda^{\mathsf{IKR},\nu}(t)
  =
  \frac{1-\xk{\lambda/(t+\lambda)}^\nu}{t}
  =
  \sum_{\ell=1}^\nu
  \lambda^{\ell-1}(t+\lambda)^{-\ell}.
\end{equation}
The final expression supplies the removable value at \(t=0\).
With \(u=t/\lambda\),
\[
  \phi_\lambda^\sharp(u)
  =
  \sum_{\ell=1}^\nu(1+u)^{-\ell},
  \qquad
  \psi_\lambda^\sharp(u)=(1+u)^{-\nu}.
\]
The shrinkage lies in \([0,1]\), and these profiles and their first derivatives
satisfy the bounds in
\Cref{cond:smooth-filter} with \(\rho_\psi=\nu\).

\subsubsection{Gradient flow}
\label{subsubsec:gradient-flow-filter-example}

At time \(t_\lambda=\lambda^{-1}\), gradient flow has
\begin{equation}
  \label{eq:gradient-flow-filter}
  \psi_\lambda^{\mathsf{GF}}(t)=e^{-t/\lambda},
  \qquad
  q_\lambda^{\mathsf{GF}}(t)=1-e^{-t/\lambda},
  \qquad
  \varphi_\lambda^{\mathsf{GF}}(t)=\frac{1-e^{-t/\lambda}}{t},
\end{equation}
with removable value \(\varphi_\lambda^{\mathsf{GF}}(0)=\lambda^{-1}\).
With \(u=t/\lambda\), write
\[
  \phi_\lambda^\sharp(u)=g(u),
  \qquad
  \psi_\lambda^\sharp(u)=e^{-u},
  \qquad
  g(u)=\frac{1-e^{-u}}{u},\quad g(0)=1.
\]
The function \(g\) and its first derivative satisfy
\[
  |g(u)|\lesssim(1+u)^{-1},
  \qquad
  |g'(u)|\lesssim(1+u)^{-2}.
\]
The corresponding bounds for \(e^{-u}\) and its first derivative are
immediate; together with \(0\le q_\lambda^{\mathsf{GF}}\le1\),
\Cref{cond:smooth-filter} holds with \(\rho_\psi=1\).

\subsubsection{Stable gradient descent}
\label{subsubsec:gradient-descent}

Assume that the step size satisfies \(\eta\kappa^2 \le\rho<1\).
Let \(m=(\eta\lambda)^{-1}\), initially assuming that \(m\) is an integer.
The gradient descent filter is
\begin{equation}
  \label{eq:gradient-descent-filter}
  \psi_\lambda^{\mathsf{GD}}(t)=(1-\eta t)^m,
  \qquad
  q_\lambda^{\mathsf{GD}}(t)=1-(1-\eta t)^m,
  \qquad
  \varphi_\lambda^{\mathsf{GD}}(t)=\frac{1-(1-\eta t)^m}{t},
\end{equation}
with removable value \(\varphi_\lambda^{\mathsf{GD}}(0)=m\eta=\lambda^{-1}\).
On \(0\le t\le\kappa^2\), stability gives \(0\le1-\eta t\le1\), hence
\(0\le q_\lambda^{\mathsf{GD}}\le1\).
With \(u=t/\lambda\), the dimensionless profiles are
\[
  \phi_\lambda^\sharp(u)
  =
  \int_0^1(1-su/m)^{m-1}\dd s,
  \qquad
  \psi_\lambda^\sharp(u)=(1-u/m)^m.
\]
Since \(u\le\rho m\), the integrand and its first \(u\)-derivative are
bounded by an exponential in \(-su\), uniformly for all sufficiently small
\(\lambda\).  Integrating gives the first two bounds in
\Cref{cond:smooth-filter}; the same exponential bound for
\(\psi_\lambda^\sharp\) and its derivative gives the remaining two with
\(\rho_\psi=1\).
Thus stable gradient descent satisfies \Cref{cond:smooth-filter}.
For noninteger \(m=(\eta\lambda)^{-1}\), the same argument applies using the real branch on \(1-\eta t>0\).

\subsubsection{Smoothed spectral cutoff}
\label{subsubsec:smoothed-cutoff-filter-example}

Let \(\chi\in C^1(\mathbb{R})\) be nondecreasing, with \(0\le\chi\le1\), \(\chi(u)=0\) for \(u\le1\), and \(\chi(u)=1\) for \(u\ge2\).
The smoothed spectral cutoff is
\begin{equation}
  \label{eq:smoothed-cutoff-filter}
  q_\lambda^{\mathsf{SC}}(t)=\chi(t/\lambda),
  \qquad
  \psi_\lambda^{\mathsf{SC}}(t)=1-\chi(t/\lambda),
  \qquad
  \varphi_\lambda^{\mathsf{SC}}(t)=\frac{\chi(t/\lambda)}{t},
\end{equation}
where \(\varphi_\lambda^{\mathsf{SC}}(t)=0\) for \(t\le\lambda\).
Since \(0\le\chi\le1\), we have \(0\le q_\lambda^{\mathsf{SC}}\le1\).
The dimensionless profiles are
\[
  \phi_\lambda^\sharp(u)=\frac{\chi(u)}{u},
  \qquad
  \psi_\lambda^\sharp(u)=1-\chi(u),
\]
where the first profile is zero on \([0,1]\).
They are absolutely continuous; their values and first derivatives have the
required bounds in \Cref{cond:smooth-filter} with \(\rho_\psi=1\), by treating
\([0,1]\), \([1,2]\), and \([2,\infty)\) separately.

\subsubsection{Spectral clipping}
\label{subsubsec:spectral-clipping}

For spectral clipping,
\[
  \varphi_\lambda^{\mathsf{clip}}(t)=\frac{1}{\max(t,\lambda)},
  \qquad
  q_\lambda^{\mathsf{clip}}(t)=\min(t/\lambda,1),
\]
and
\[
  \psi_\lambda^{\mathsf{clip}}(t)=\xk{1-\frac{t}{\lambda}}_+.
\]
The basic filter conditions hold since \(0\le q_\lambda^{\mathsf{clip}}(t)\le1\).
Its dimensionless profiles are
\[
  \phi_\lambda^\sharp(u)
  =
  \begin{cases}
    1,&0\le u\le1,\\
    u^{-1},&u\ge1,
  \end{cases}
  \qquad
  \psi_\lambda^\sharp(u)=(1-u)_+.
\]
These profiles are absolutely continuous.
With residual order \(\rho_\psi=1\), they satisfy
\[
  |\phi_\lambda^\sharp(u)|
  \lesssim
  (1+u)^{-1},
  \qquad
  |(\phi_\lambda^\sharp)'(u)|
  \lesssim
  (1+u)^{-2},
\]
and
\[
  |\psi_\lambda^\sharp(u)|
  \lesssim
  (1+u)^{-1},
  \qquad
  |(\psi_\lambda^\sharp)'(u)|
  \lesssim
  (1+u)^{-2}
\]
almost everywhere, with constants independent of \(\lambda\).
Hence spectral clipping satisfies \Cref{cond:smooth-filter}.

\subsection{Saturation and interpolation}
\label{subsec:related-phenomena}

The sequence risk representation also clarifies two phenomena that are less visible in analyses based only on rates.
Under the hypotheses of the relevant equivalence theorem, both conclusions carry over to the kernel risk.

\subsubsection{Saturation effects}
\label{subsubsec:saturation-effects}

For a source ball \(\mathcal{F}_s(R)\), the worst-case bias in the sequence model is exactly
\[
  \sup_{f\in\mathcal{F}_s(R)}
  \sum_{j\ge1}
  \psi_\lambda(\lambda_j)^2 f_j^2
  =
  R^2
  \sup_{j\ge1}
  \lambda_j^s \psi_\lambda(\lambda_j)^2.
\]
This identity shows the role of qualification directly.
If the residual has qualification at least \(s/2\), the preceding bias is \(O(R^2 \lambda^s)\).
When the filter has only finite qualification, additional target smoothness need not improve this order.
For example, the kernel ridge residual in \Cref{eq:krr-filter} gives the worst-case bias bound \(O(\lambda^{\min(s,2)})\), while the \(\nu\)-fold iterated residual in \Cref{eq:iterated-krr-filter} gives \(O(\lambda^{\min(s,2\nu)})\).
Above their respective thresholds, these bounds are sharp because any fixed positive population eigenvalue contributes at orders \(\lambda^2\) and \(\lambda^{2\nu}\), respectively.
Thus iteration postpones the saturation threshold.
Gradient flow, stable gradient descent, and smoothed spectral cutoff have arbitrary qualification, so they do not suffer from saturation caused by finite qualification.

\subsubsection{Interpolation regime}
\label{subsubsec:interpolation-regime}

The variance comparison yields a lower bound for every filter covered by the main theorem, even without the conditions required for a full bias comparison.
Suppose that, for some \(\rho>0\) and \(F_\rho \in(0,\infty)\), its residual satisfies the qualification bound
\[
  0\le \psi_\lambda(t)\le
  F_\rho \xk{\frac{\lambda}{t}}^\rho,
  \qquad t>0.
\]
This holds under \Cref{ass:lipschitz-filter}, with its positive residual order.
Setting \(c_\rho=(2F_\rho)^{1/\rho}\), we have
\[
  q_\lambda(t)=1-\psi_\lambda(t)\ge \frac{1}{2},
  \qquad t\ge c_\rho \lambda,
\]
and therefore
\[
  \mathcal{N}_q(\lambda)
  =
  \sum_{j\ge1} q_\lambda(\lambda_j)^2
  \ge
  \frac{1}{4}\#\{j:\lambda_j \ge c_\rho \lambda\}.
\]
Under \Cref{ass:basic-model,ass:lipschitz-filter} and
\cref{eq:leverage-variance-compatibility},
\Cref{prop:lipschitz-variance-comparison} and nonnegativity of the squared bias yield
\[
  \mathcal{E}_n^{\mathsf{np}}(q_\lambda;f^*\mid X)
  \ge
  \mathsf{V}_X(\varphi_\lambda)
  \ge
  \xk{1-o_{\mathbb{P}}(1)}
  \frac{\sigma^2}{4n}
  \#\{j:\lambda_j \ge c_\rho \lambda\}.
\]
Thus every admissible filter incurs a variance lower bound determined by the number of population eigenvalues resolved at scale \(\lambda\); as the regularization scale decreases, this bound marks the onset of the interpolation regime.
The empirical pseudoinverse, for which \(q^{\mathsf{int}}(t)=\mathbf{1}\{t>0\}\), is the limiting extreme rather than the only filter to which the conclusion applies.

\subsection{Pinsker minimax constant}
\label{subsec:pinsker-constants-interpretation}

We now ask a sharper question than rate optimality: what is the exact asymptotic constant for the minimax risk over \(\mathcal{F}_s(R)\), averaged over the design and noise, and can a concrete kernel spectral filter attain it?
Balancing the worst-case bias and variance identifies a regularization scale of the optimal order but not the sharp constant, which depends on the exact asymptotics of the spectral sums and the shape of the shrinkage profile.

Motivated by Pinsker's sharp Gaussian sequence result \citep{pinsker1980_OptimalFiltering}, we consider the following Pinsker profile and transfer the resulting sequence model calculation to the kernel estimator with random design.

Fix \(s\ge1\) and \(R>0\).
For \(\lambda>0\), define the Pinsker shrinkage profile
\[
  q_{\lambda,s}^{\mathsf{Pin}}(t)
  =
  \begin{cases}
    0,
    &0\le t\le \lambda,
    \\
    1-(\lambda/t)^{s/2},
    &t>\lambda,
  \end{cases}
\]
its regularization function
\[
  \varphi_{\lambda,s}^{\mathsf{Pin}}(t)
  =
  \begin{cases}
    0,
    &0\le t\le \lambda,
    \\
    t^{-1} \xk{1-(\lambda/t)^{s/2}},
    &t>\lambda,
  \end{cases}
\]
and the residual
\[
  \psi_{\lambda,s}^{\mathsf{Pin}}(t)
  =
  \begin{cases}
    1,
    &0\le t\le \lambda,
    \\
    (\lambda/t)^{s/2},
    &t>\lambda.
  \end{cases}
\]
The corresponding kernel estimator is
\[
  \widehat{f}_{\lambda,s}^{\mathsf{Pin}}
  =
  \varphi_{\lambda,s}^{\mathsf{Pin}}(T_X)\widehat{g}_Z.
\]

\begin{assumption}
\label{ass:gaussian-pinsker-spectrum}
The noise variables are i.i.d.\ \(\mathcal{N}(0,\sigma^2)\) and independent of the design.
Moreover, the eigenvalue counting function, with multiplicities included, satisfies
\[
  N_L(t)
  \coloneqq
  \#\{j:\lambda_j \ge t\}
  \sim
  Dt^{-1/\beta},
  \qquad
  t\downarrow0,
\]
for fixed constants \(D>0\) and \(\beta>1\).
\end{assumption}

Let
\(
  \mathcal{G}=\overline{\Ran(L)}\subset L^2(\mu)
\),
and let \(\mathfrak{D}_n\) be the class of all jointly measurable prediction rules
\[
  \widehat{g}_n:
  \mathcal{X}^n \times\mathbb{R}^n
  \longrightarrow
  \mathcal{G}.
\]
For \(\widehat{g}_n \in\mathfrak{D}_n\), set
\[
  \mathcal{R}_n(\widehat{g}_n,f)
  =
  \E_{X,\varepsilon} \zk{
    \norm{\widehat{g}_n(X,y)-f}_{L^2(\mu)}^2
  }.
\]
The minimax risk is
\begin{equation}
  \label{eq:pinsker-minimax-risk}
  \mathfrak{R}_n
  \zk{\mathcal{F}_s(R)}
  \coloneqq
  \inf_{\widehat{g}_n \in\mathfrak{D}_n}
  \sup_{f\in\mathcal{F}_s(R)}
  \mathcal{R}_n(\widehat{g}_n,f).
\end{equation}

Put \(\rho=s\beta\), and define the deterministic oracle scale
\begin{equation}
  \label{eq:pinsker-oracle-scale}
  \lambda_n^*
  =
  \xk{
    \frac{
      \sigma^2 D \rho
    }{
      nR^2(\rho+1)(\rho+2)
    }
  }^{\beta/(\rho+1)}.
\end{equation}
This is an oracle scale because it depends on the fixed class and model parameters \(s,\beta,R,D,\sigma^2\).

Set
\begin{equation}
  \label{eq:pinsker-upper-constant}
  C_{\mathsf{Pin}}
  =
  D^{\rho/(\rho+1)}
  R^{2/(\rho+1)}
  (\sigma^2)^{\rho/(\rho+1)}
  (\rho+1)^{1/(\rho+1)}
  \xk{
    \frac{\rho}{\rho+2}
  }^{\rho/(\rho+1)}.
\end{equation}

\begin{corollary}[Pinsker upper bound]
\label{cor:pinsker-upper}
Under \Cref{ass:basic-model,ass:gaussian-pinsker-spectrum}, for every fixed \(s\ge1\) and \(R>0\),
\begin{align}
  &
  \sup_{f\in\mathcal{F}_s(R)}
  \mathcal{E}_n^{\mathsf{np}}
  (q_{\lambda_n^*,s}^{\mathsf{Pin}};f)
  \notag
  \\
  &\quad=
  \xk{1+o(1)}
  \zk{
    R^2(\lambda_n^*)^s
    +
    \frac{\sigma^2}{n}
    \sum_{j\ge1}
    \xk{
      1-\xk{\frac{\lambda_n^*}{\lambda_j}}^{s/2}
    }_+^2
  }
  \notag
  \\
  &\quad=
  \xk{C_{\mathsf{Pin}}+o(1)}
  n^{-\rho/(\rho+1)}.
  \label{eq:pinsker-upper}
\end{align}
\end{corollary}

\begin{theorem}[Pinsker minimax constant]
\label{thm:pinsker-minimax}
Under \Cref{ass:basic-model,ass:gaussian-pinsker-spectrum}, for every fixed \(s\ge1\) and \(R>0\),
\begin{equation}
  \label{eq:pinsker-minimax}
  n^{\rho/(\rho+1)}
  \mathfrak{R}_n
  \zk{\mathcal{F}_s(R)}
  \longrightarrow
  C_{\mathsf{Pin}}.
\end{equation}
Moreover,
\begin{equation}
  \label{eq:pinsker-minimax-attainment}
  \sup_{f\in\mathcal{F}_s(R)}
  \mathcal{E}_n^{\mathsf{np}}
  (q_{\lambda_n^*,s}^{\mathsf{Pin}};f)
  =
  \xk{1+o(1)}
  \mathfrak{R}_n
  \zk{\mathcal{F}_s(R)}.
\end{equation}
\end{theorem}
Thus the Pinsker spectral estimator with the oracle scale attains the sharp minimax constant over \(\mathcal{F}_s(R)\).

\subsection{Optimal rates and Pinsker constants in high dimensions}
\label{subsec:high-dimensional-optimality}

The integer scales in \Cref{sec:high-dimensional-scale-verification} also yield sharp minimax procedures in the spherical setting.
The gaps between successive spherical eigenvalues allow a smooth filter to retain entire spectral subspaces, with partial shrinkage of the last retained subspace at a transition point.
We transfer the resulting sequence risk through \Cref{thm:risk-equivalence-expectation} and use the matching lower bound of \citet{lu2024_PinskerBound}.

Retain the spherical kernel and source ball of \Cref{sec:high-dimensional-scale-verification}, writing \(\mathcal F_{s,d}(R)\) for the ball at dimension \(d\).
In this subsection, let \(s,R>0\), let the noise be i.i.d.\ \(\mathcal N(0,\sigma^2)\) independently of the design, and assume
\[
  n=\eta d^\gamma(1+o(1)),
  \qquad \eta>0,
  \qquad \gamma\ge s+1.
\]
Let \(\mathfrak R_{n,d}(\mathcal F_{s,d}(R))\) denote the minimax risk in \cref{eq:pinsker-minimax-risk} for this dimension-dependent model.
Set
\[
  p=\left\lfloor\frac{\gamma}{s+1}\right\rfloor\ge1,
  \qquad \delta=\gamma-p(s+1)\in[0,s+1),
  \qquad \zeta=\min\{\gamma-p,(p+1)s\},
\]
and define
\[
  A_j=R^2(a_jj!)^s,
  \qquad B_p=\frac{\sigma^2}{\eta p!}.
\]
For a fixed \(w\in(0,1]\), choose a smooth nondecreasing profile \(\chi_w:[0,\infty)\to[0,1]\) with
\[
  \chi_w(u)=0\quad(u\le1/2),
  \qquad \chi_w(1)=w,
  \qquad \chi_w(u)=1\quad(u\ge2).
\]
Use the filter \(q_\lambda(t)=\chi_w(t/\lambda)\), with \(\varphi_\lambda(t)=q_\lambda(t)/t\) for \(t>0\) and \(\varphi_\lambda(0)=0\), at the scale
\begin{equation}
  \label{eq:high-dimensional-optimal-filter}
  \lambda_d=\mu_{d,p}\asymp d^{-p},
  \qquad
  w=
  \begin{cases}
    A_p/(A_p+B_p),&\delta=0,\\
    1,&\delta>0.
  \end{cases}
\end{equation}
Here \(\mu_{d,p}\) is the population eigenvalue of degree \(p\).
The choice is deterministic and depends on the kernel and the fixed model parameters.

\begin{corollary}[High-dimensional Pinsker constant]
  \label{cor:high-dimensional-pinsker}
  Under the preceding assumptions, the filter in \cref{eq:high-dimensional-optimal-filter} satisfies
  \begin{align}
    \sup_{f\in\mathcal F_{s,d}(R)}
      \mathcal E_n^{\mathsf{np}}(q_{\lambda_d};f)
    =\bigl(C_{s,\gamma}+o(1)\bigr)d^{-\zeta}
    =\bigl(1+o(1)\bigr)
      \mathfrak R_{n,d}(\mathcal F_{s,d}(R)),
    \label{eq:high-dimensional-pinsker-attainment}
  \end{align}
  where
  \[
    C_{s,\gamma}=
    \begin{cases}
      A_pB_p/(A_p+B_p),&\delta=0,\\
      B_p,&0<\delta<s,\\
      B_p+A_{p+1},&\delta=s,\\
      A_{p+1},&s<\delta<s+1.
    \end{cases}
  \]
\end{corollary}
The rate \(d^{-\zeta}\asymp n^{-\zeta/\gamma}\) has the plateau structure of the high-dimensional minimax risk.
Away from \(\delta=0\), a smooth cutoff already attains the sharp constant; at \(\delta=0\), partial shrinkage of degree \(p\) balances its bias and variance.
All these filters satisfy the comparison scales because \(1\le p<\gamma\).
The proof is given in \Cref{subsec:high-dimensional-pinsker-proof}.

The same comparison also recovers the optimal KRR rate in the range \(0<s\le1\) studied by \citet{zhang2025_OptimalRates}.
\begin{corollary}[High-dimensional KRR rate]
  \label{cor:high-dimensional-krr-rate}
  Under the assumptions of this subsection, suppose \(0<s\le1\) and choose
  \[
    \lambda_d\asymp d^{-a},
    \qquad a=p+\frac12\min\{\delta,s\}.
  \]
  Then
  \[
    \sup_{f\in\mathcal F_{s,d}(R)}
      \mathcal E_n^{\mathsf{np}}(q_{\lambda_d}^{\mathsf{KR}};f)
    \asymp d^{-\zeta}
    \asymp\mathfrak R_{n,d}(\mathcal F_{s,d}(R)).
  \]
\end{corollary}
This is a statement about the worst-case risk over the source ball; it does not require every target to have nonvanishing energy in each spectral subspace.
The verification is also given in \Cref{subsec:high-dimensional-pinsker-proof}.
 \section{Numerical Experiments}
\label{sec:numerical-experiments}

This section presents numerical illustrations of the risk equivalence results.
Let \(x=(x_1,\ldots,x_d)\) have independent Rademacher coordinates and let \(e_j(x)=x_j\).
We use \(d=255\) and a bounded kernel with population eigenvalues
\[
  \mu_j=j^{-3/2},
  \qquad
  j=1,\ldots,d.
\]
Thus, the \(e_j\)'s are bounded orthonormal eigenfunctions, and the covariance matrix in population coordinates is \(\Lambda=\operatorname{diag}(\mu_1,\ldots,\mu_d)\).
We take \(\sigma^2=1\) and, for source order \(s\), use the normalized fixed target
\[
  f_j^*=c_s \mu_j^{s/2} j^{-1},
  \qquad
  \sum_{j=1}^d(f_j^*)^2=1.
\]

For a design matrix \(E_X=(e_j(x_i))_{i,j}\), write
\[
  \widehat{T}_X
  =
  \Lambda^{1/2}
  \xk{\frac{E_X^{\mathsf{T}} E_X}{n}}
  \Lambda^{1/2}.
\]
We evaluate the finite-rank conditional risk exactly:
\[
  \left\|
    \Lambda^{1/2} q_\lambda(\widehat{T}_X)
    \Lambda^{-1/2} \bm{f}^*-
    \bm{f}^*
  \right\|_2^2
  +
  \frac{\sigma^2}{n}
  \Tr\xk{
    \Lambda\varphi_\lambda(\widehat{T}_X)^2\widehat{T}_X
  }.
\]
This is the coordinate form of \cref{eq:np-risk}, which we compare with the corresponding finite-rank sequence risk in \cref{eq:sequence-risk}.
For each displayed parameter setting, we draw 200 paired, independent designs.
The marker gives the mean exact conditional risk across these designs, and the shaded band shows its 10th--90th percentiles.
The regularization parameters are fixed from population quantities before the designs are drawn.

\subsection{Risk equivalence for regular filters}
\label{subsec:regular-filter-numerics}

We begin with \(s=1\) and the prescribed scale path \(\lambda_n=n^{-3/5}\).
We compare kernel ridge regression, gradient flow, the smoothed
spectral cutoff in \Cref{subsubsec:smoothed-cutoff-filter-example}, and spectral
clipping from \Cref{subsubsec:spectral-clipping}.
This comparison covers regular filters under the main theorem.

\begin{figure}[htpb]
  \centering
  \includegraphics[width=\textwidth]{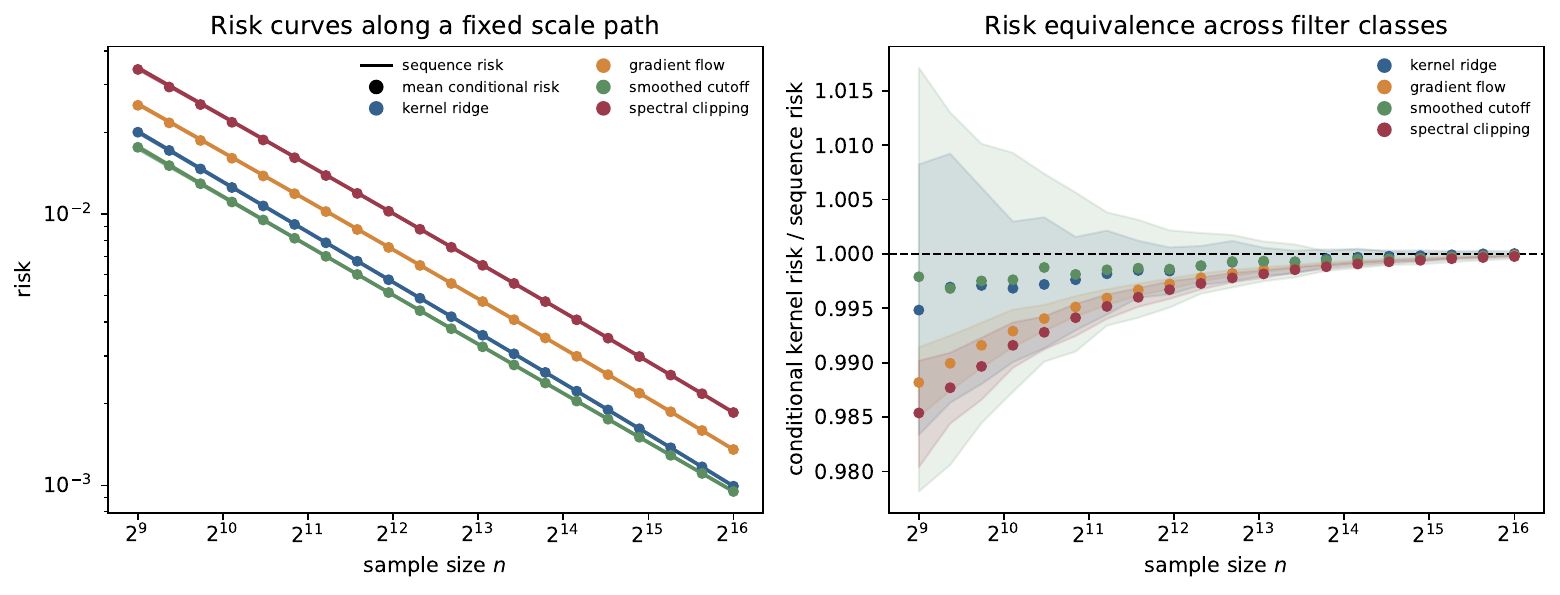}
  \caption{
    Exact finite-rank risks for four regular spectral filters over 200 paired designs at each displayed sample size.
    In the left panel, solid curves give the deterministic sequence risks, markers give mean conditional kernel risks, and bands give their 10th--90th conditional percentiles.
    The right panel shows the corresponding ratios of conditional risks.
  }
\label{fig:regular-filter-equivalence}
\end{figure}
\FloatBarrier

In \Cref{fig:regular-filter-equivalence}, the sequence curves closely track the learning curves of the kernel estimators across the displayed range of sample sizes.
The ratios approach one along the prescribed scale path for all four filters, including spectral clipping, whose kink lies outside the analytic filter class.

\subsection{Learning curves}
\label{subsec:filter-learning-curves-numerics}

We next fix \(n=8192\), retain \(s=1\), and vary the regularization over a deterministic logarithmic grid.
The comparison includes the four filters above and the Pinsker profile \(q^{\mathsf{Pin}}_{\lambda,1}\) in \Cref{subsec:pinsker-constants-interpretation}.
All filters use the same designs and fixed target under the common regularization grid.

\begin{figure}[htpb]
  \centering
  \includegraphics[width=0.5\textwidth]{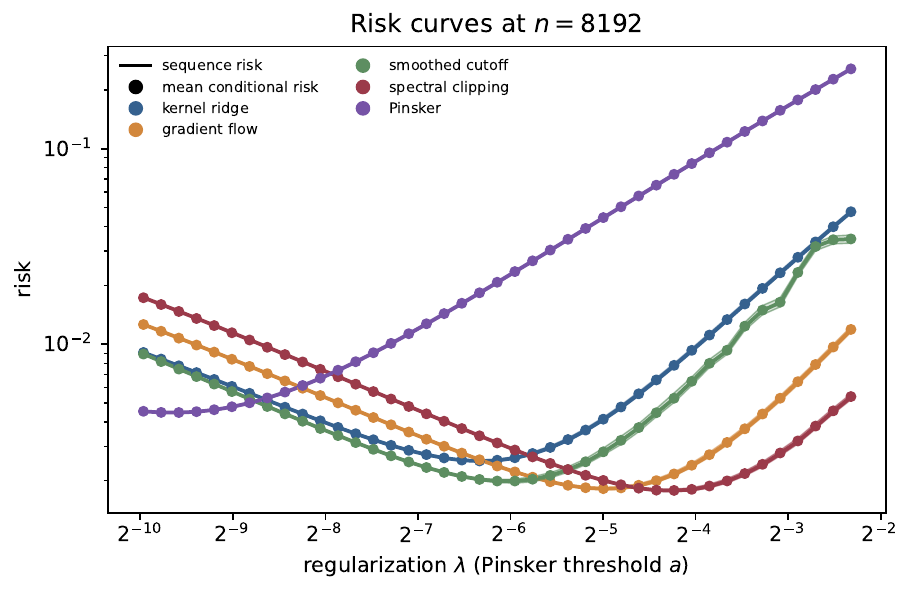}
  \caption{
    Exact risk curves at a fixed sample size over 200 paired designs.
    Solid curves give deterministic sequence risks, markers give mean conditional kernel risks, and bands give the 10th--90th conditional percentiles over the same regularization grid.
  }
\label{fig:learning-curves}
\end{figure}
\FloatBarrier

The solid curves in \Cref{fig:learning-curves} predict the risk profiles of the kernel estimators throughout the displayed regularization range.
The Pinsker curve compares spectral profiles for a fixed target, whereas its sharp optimality concerns the source-ball minimax risk under the additional assumptions of \Cref{thm:pinsker-minimax}.

\subsection{Worst-case risks over a source ball}
\label{subsec:source-ball-worst-case-numerics}

To examine the worst-case risk underlying Pinsker optimality, we consider the finite-rank source ball
\[
  \mathcal{F}^{(d)}_{s,R}
  =
  \dk{
    f=L^{s/2} h:\ \sum_{j=1}^d h_j^2 \le R^2
  }.
\]
For a spectral profile \(q_\lambda\), the exact worst-case sequence risk over this ball is
\[
  \mathcal{W}_{n,R}(q_\lambda)
  =
  R^2 \max_{1\le j\le d}
  \mu_j^s \xk{1-q_\lambda(\mu_j)}^2
  +
  \frac{\sigma^2}{n}\sum_{j=1}^d q_\lambda(\mu_j)^2.
\]
The maximum arises because the worst-case source element concentrates its \(h\)-energy on the coordinate with the largest residual bias.
We use \(s=R=1\), \(n=8192\), and the same spectrum and noise variance as above.

\begin{figure}[htpb]
  \centering
  \includegraphics[width=0.5\textwidth]{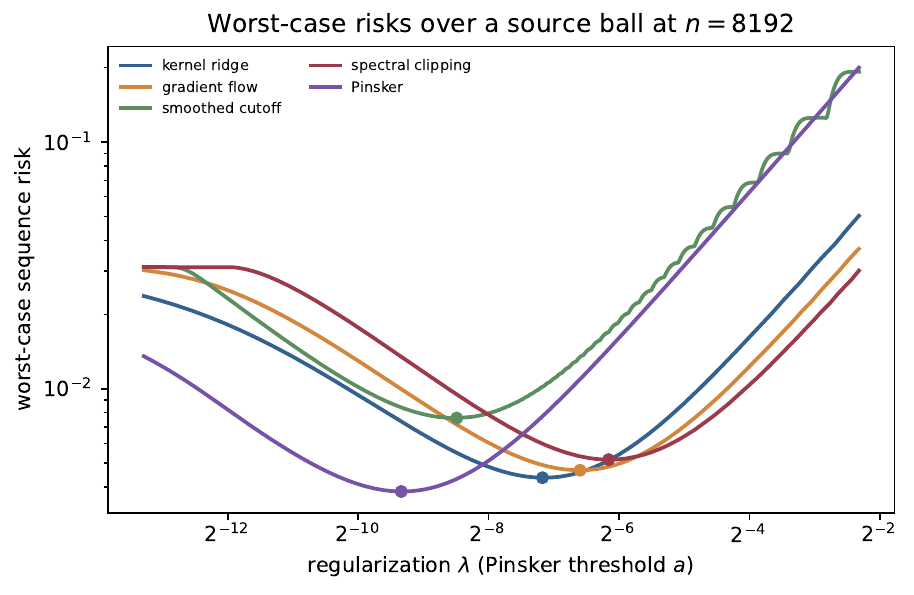}
  \caption{
    Exact finite-rank worst-case sequence risks over the source ball \(\mathcal{F}^{(255)}_{1,1}\) at \(n=8192\).
    Each curve is \(\mathcal{W}_{n,R}(q_\lambda)\), and its marker identifies the minimum over the displayed regularization grid.
  }
\label{fig:source-ball-worst-case-risk}
\end{figure}
\FloatBarrier

As shown in \Cref{fig:source-ball-worst-case-risk}, the Pinsker profile has the smallest minimum risk among the filters compared over the displayed regularization grid.
This exact finite-rank sequence-model calculation complements the conditional kernel risk curves in \Cref{fig:learning-curves}.

\subsection{Hard spectral cutoff}
\label{subsec:hard-cutoff-numerics}

The equivalence results above require regular spectral filters.
For a discontinuous filter such as the hard spectral cutoff \(q_\lambda(t)=\mathbf{1}\{t\ge\lambda\}\), upper and lower polynomial eigenvalue envelopes do not control how the empirical spectrum crosses the discontinuity.
Without finer local information near the threshold---for example, control of the threshold eigenspace and the signal it carries---learning curve equivalence need not hold.

We examine this boundary in a finite-rank example.
Let \(e_j(x)=x_j\) be uniformly bounded orthonormal coordinate functions on a product Rademacher space, and retain the first \(d=4^5-1=1023\) coordinates.
We compare a smooth polynomial model with a block model whose geometric blocks are
\[
  I_\ell
  =
  \{4^\ell,\ldots,4^{\ell+1}-1\},
  \qquad
  \ell=0,\ldots,4,
\]
and whose coefficients are
\[
  \mu_j=4^{-3\ell/2},
  \qquad
  (f_j^*)^2=c_\ell4^{-2\ell},
  \qquad
  j\in I_\ell,
  \qquad
  c_\ell=
  \begin{cases}
    4,&\ell\ \text{even},\\
    1,&\ell\ \text{odd}.
  \end{cases}
\]
The smooth model uses \(\mu_j=j^{-3/2}\) and \((f_j^*)^2=j^{-2}\), while both models have matching upper and lower polynomial orders:
\[
  \mu_j \asymp j^{-3/2},
  \qquad
  (f_j^*)^2\asymp j^{-2}.
\]
Both targets satisfy \(f^*=L^{s/2} h\) for every \(s<2/3\), in particular for \(s=1/2\).

Choose \(\lambda\) so that the population rule retains coordinates \(1,\ldots,63\).
In the block model, the 48-dimensional block \(16,\ldots,63\) lies exactly at the threshold.
For each independent Rademacher design, we evaluate both sides of \cref{eq:np-risk,eq:sequence-risk} exactly, with \(\sigma^2=1\).
The number of retained population coordinates is \(63\), below every sample size considered: \(n\in\{1024,4096,16384\}\).

\begin{figure}[t]
  \centering
  \includegraphics[width=\textwidth]{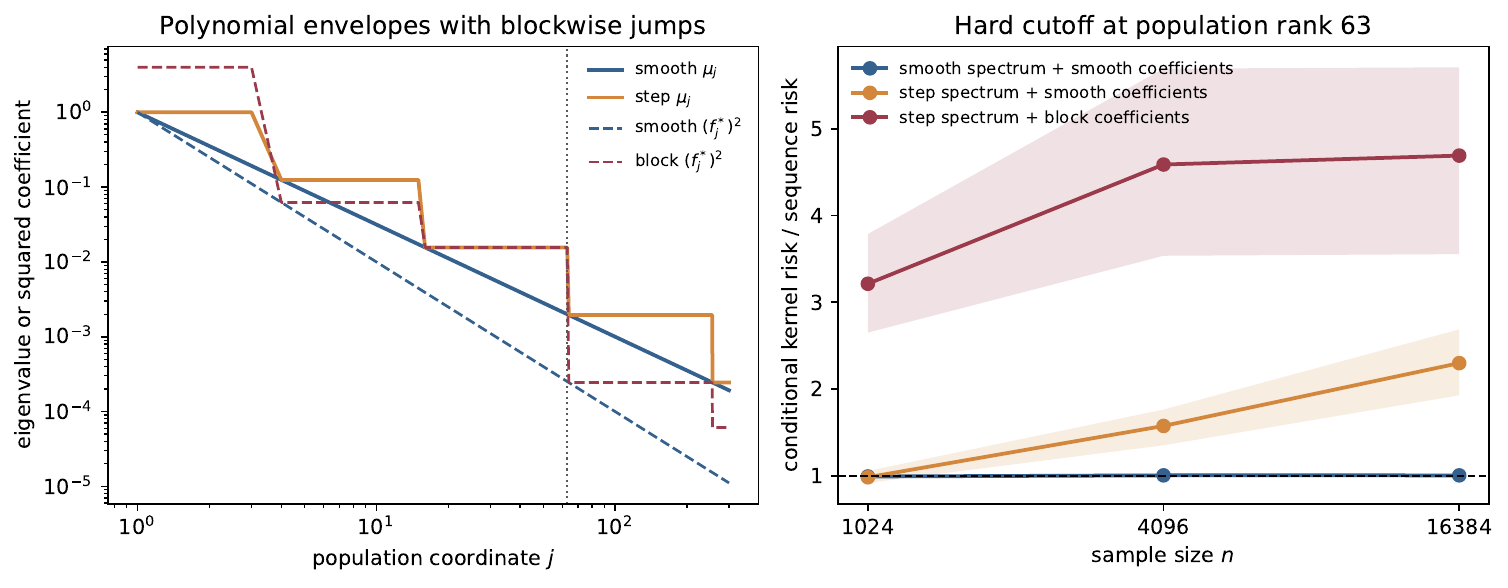}
  \caption{
    Risk comparison for hard spectral cutoff over 40 independent designs.
    The left panel shows the two spectral and coefficient constructions together with the population cutoff.
    The right panel shows the mean ratio of the conditional risk of the kernel estimator to the sequence-model risk, with shaded 10th--90th percentile bands.
  }
\label{fig:hard-cutoff-plateau}
\end{figure}
\FloatBarrier

\Cref{fig:hard-cutoff-plateau} makes the mechanism explicit.
In the smooth model, the empirical and population cutoffs produce nearly identical risks.
In the block model, the empirical cutoff splits the spectral plateau and retains fewer directions than the population rule, creating a persistent discrepancy that increases when the signal coefficients vary across blocks.
The example shows that vanishing errors in the empirical eigenvalues alone do not ensure learning curve equivalence for a discontinuous filter: the local spectral configuration at the threshold determines which directions are retained.
 \section{Conclusion}
\label{sec:conclusion}

This paper characterizes the risk of random-design kernel spectral regression through a population sequence model.
Under the stated structural, filter, and scale conditions, the risk of the kernel spectral estimator is asymptotically equivalent to that of the corresponding population Gaussian sequence estimator.
The resulting sequence risk gives an explicit first-order learning curve determined by the population eigenvalues, target coefficients, noise level, and shrinkage profile.
The proof combines finite-Schatten perturbation bounds with scale-regular Lipschitz bounds in logarithmic spectral coordinates, and therefore covers continuous nonsmooth filters such as spectral clipping.
The remainders are uniform over families of problem instances that satisfy the corresponding uniformity conditions.

Under Gaussian noise and exact polynomial asymptotics for the eigenvalue counting function, the Pinsker profile attains the minimax constant over every fixed source ball with \(s\ge1\), among all measurable prediction rules.
The sequence risk also shows how finite filter qualification produces saturation and how the number of population eigenvalues resolved at the regularization scale marks the onset of interpolation.

The risk equivalence concerns prescribed spectral estimators along deterministic regularization sequences, with uniformity over admissible families of problem instances.
Hard spectral cutoff remains outside the present equivalence theory because global spectral decay alone does not determine how empirical eigenvalues cross a discontinuity.
Equivalence for such filters will require finer control of the local spectral projector and the signal near the threshold.
When the regularization parameter is selected using the responses, the filter depends on the noise, so the conditional bias--variance identity for a prescribed filter no longer applies directly.
 \section*{Acknowledgements}
We thank Yuxuan Hou for helpful suggestions on this work.
 \phantomsection

\clearpage
\appendix
\clearpage{}\section{Empirical Operator Control}
\label{sec:operator-concentration}

This section proves the estimates for random designs used in the
Schatten argument for scale-regular Lipschitz filters.
All statements are consequences of the structural assumptions in \Cref{sec:main-results}.
Set
\[
  T_\lambda \coloneqq T+\lambda I,
  \qquad
  T_{X,\lambda} \coloneqq T_X+\lambda I,
\]
and define
\[
  \delta_\lambda
  \coloneqq
  \norm{
    T_\lambda^{-1/2}(T_X-T)T_\lambda^{-1/2}
  },
  \qquad
  \eta_\lambda
  \coloneqq
  \norm{
    T_\lambda^{-1/2}(T_X-T)T_\lambda^{-1/2}
  }_{\mathfrak{S}_2}.
\]
Here \(\mathfrak{S}_2\) denotes the Hilbert--Schmidt class.
We also write
\[
  v_\lambda
  \coloneqq
  \frac{\mathcal{L}_\lambda\ell_\lambda}{n}.
\]
For the operator norm bound, define
\[
  \varrho_\lambda
  \coloneqq
  \norm{T_\lambda^{-1/2} TT_\lambda^{-1/2}},
  \qquad
  \mathfrak{d}_\lambda
  \coloneqq
  \begin{cases}
    \mathcal{N}_1(\lambda)/\varrho_\lambda,
      & \varrho_\lambda>0,\\
    1,
      & \varrho_\lambda=0,
  \end{cases}
  \qquad
  \ell_\lambda^{\mathrm{int}}
  \coloneqq
  \log(e+\mathfrak{d}_\lambda).
\]
When \(\varrho_\lambda>0\), \(\mathfrak{d}_\lambda\) is the intrinsic dimension
of the regularized population covariance; the convention
\(\mathfrak{d}_\lambda=1\) covers \(T=0\).

\subsection{Concentration of the regularized covariance operator}
\label{subsec:covariance-concentration}

\begin{proposition}[Concentration of the regularized covariance operator]
\label{prop:covariance-concentration}
Under \Cref{ass:basic-model} and
\cref{eq:leverage-variance-compatibility},
\[
  \delta_\lambda
  =
  O_{\mathbb{P}} \xk{
    \sqrt{
      \frac{
        \mathcal{L}_\lambda
        \varrho_\lambda
        \ell_\lambda^{\mathrm{int}}
      }{n}
    }
    +
    \frac{
      \mathcal{L}_\lambda
      \ell_\lambda^{\mathrm{int}}
    }{n}
  },
  \qquad
  \eta_\lambda
  =
  O_{\mathbb{P}}\xk{\sqrt{\frac{\mathcal{L}_\lambda\mathcal{N}_1(\lambda)}{n}}}.
\]
If \(T\) is fixed and nonzero, then
\(\varrho_\lambda \to1\) and
\(\ell_\lambda^{\mathrm{int}}\asymp\ell_\lambda\).
If, in addition, \(\sup_{0\le t\le\kappa^2}(t+\lambda)\varphi_\lambda(t)\)
is bounded uniformly in \(\lambda\), then
\[
  \delta_\lambda=O_{\mathbb{P}}(\sqrt{v_\lambda})
  =o_{\mathbb{P}}(1).
\]
The bounds also hold when the population operator varies with \(n\), with
the displayed leverage scale evaluated rowwise.
\end{proposition}

\begin{proof}
Write
\[
  Z_i
  =
  T_\lambda^{-1/2}
  \bigl(k_{x_i} \otimes k_{x_i}-T\bigr)
  T_\lambda^{-1/2}.
\]
Then
\[
  T_\lambda^{-1/2}(T_X-T)T_\lambda^{-1/2}
  =
  \frac{1}{n}\sum_{i=1}^n Z_i,
  \qquad
  \E Z_i=0.
\]
For \(\mu\)-almost every \(x\), the definition of \(\mathcal{L}_\lambda\)
gives
\[
  \norm{T_\lambda^{-1/2} k_x}_{\mathcal{H}}^2
  =
  \sum_{j\ge1} \frac{\lambda_j}{\lambda_j+\lambda}e_j(x)^2
  \le
  \mathcal{L}_\lambda.
\]
Hence \(\norm{Z_i}\lesssim \mathcal{L}_\lambda\).
Moreover,
\[
  A_x
  \coloneqq
  T_\lambda^{-1/2}(k_x \otimes k_x)T_\lambda^{-1/2},
  \qquad
  A\coloneqq\E A_x
  =T_\lambda^{-1/2} TT_\lambda^{-1/2}.
\]
Then \(\Tr A=\mathcal{N}_1(\lambda)\), \(0\preceq A\preceq I\), and
\[
  A_x^2
  =
  \norm{T_\lambda^{-1/2} k_x}_{\mathcal{H}}^2 A_x
  \preceq
  \mathcal{L}_\lambda A_x.
\]
Since \(Z_i=A_{x_i}-A\),
\[
  \E Z_i^2
  =
  \E A_x^2-A^2
  \preceq
  \mathcal{L}_\lambda A.
\]
Thus \(n\mathcal{L}_\lambda A\) is a variance majorant for the sum.
If \(\varrho_\lambda=0\), then \(T=0\), hence \(A_x=0\) almost surely and
both conclusions are immediate.
Suppose henceforth that \(\varrho_\lambda>0\).
The variance majorant has operator norm
\[
  n\mathcal{L}_\lambda\varrho_\lambda
\]
and intrinsic dimension
\[
  \frac{
    \Tr\xk{n\mathcal{L}_\lambda A}
  }{
    \norm{n\mathcal{L}_\lambda A}
  }
  =
  \frac{\mathcal{N}_1(\lambda)}{\varrho_\lambda}
  =
  \mathfrak{d}_\lambda.
\]
Combining this calculation with
\(\norm{Z_i}\lesssim\mathcal{L}_\lambda\), the matrix Bernstein inequality
with intrinsic dimension \citep[Section~7.3]{tropp2015_IntroductionMatrix}
yields
\[
  \delta_\lambda
  =
  O_{\mathbb{P}} \xk{
    \sqrt{
      \frac{
        \mathcal{L}_\lambda
        \varrho_\lambda
        \ell_\lambda^{\mathrm{int}}
      }{n}
    }
    +
    \frac{
      \mathcal{L}_\lambda
      \ell_\lambda^{\mathrm{int}}
    }{n}
  }.
\]
To pass from finite-dimensional compressions to the compact operator setting,
let \(P_m\) be increasing spectral projections of finite rank for \(A\), each
containing its leading eigenspace.
Then
\[
  \E(P_m Z_i P_m)^2
  \preceq
  P_m(\E Z_i^2)P_m
  \preceq
  \mathcal{L}_\lambda P_m A P_m,
\]
so the corresponding bounds on the norm, trace, and intrinsic dimension are
uniform in \(m\).
Apply the finite-dimensional inequality to these compressions.
Since \(\sum_i Z_i\) is compact,
\(P_m(\sum_i Z_i)P_m \to\sum_i Z_i\) in operator norm, and the same probability bound passes to the limit.
For the Hilbert--Schmidt bound, independence and centering give
\[
  \E \eta_\lambda^2
  =
  \frac{1}{n}\E\norm{Z_i}_{\mathfrak{S}_2}^2
  =
  \frac{1}{n}\Tr(\E Z_i^2)
  \lesssim
  \frac{\mathcal{L}_\lambda\mathcal{N}_1(\lambda)}{n},
\]
where the last step uses the variance majorant and \(\Tr A=\mathcal{N}_1(\lambda)\).
Markov's inequality gives
\[
  \eta_\lambda
  =
  O_{\mathbb{P}}\xk{\sqrt{\frac{\mathcal{L}_\lambda\mathcal{N}_1(\lambda)}{n}}}.
\]
Finally, if \(T\) is fixed and nonzero, then
\[
  \varrho_\lambda
  =
  \frac{\lambda_1}{\lambda_1+\lambda}
  \longrightarrow1.
\]
Consequently,
\(\mathfrak{d}_\lambda \asymp\mathcal{N}_1(\lambda)\) and
\(\ell_\lambda^{\mathrm{int}}\asymp\ell_\lambda\).
Under the additional regularized inverse bound, variance compatibility
implies \(v_\lambda\to0\), and the linear Bernstein term is absorbed by
the square root term, as shown below.
\end{proof}

\begin{corollary}[Covariance concentration under variance compatibility]
\label{cor:variance-compatible-concentration}
Under \Cref{ass:basic-model} and
\cref{eq:leverage-variance-compatibility},
suppose
\(q_\lambda(t)=t\varphi_\lambda(t)\) and, for some constant \(E<\infty\)
independent of \(n\) and \(\lambda\),
\[
  \sup_{0\le t\le\kappa^2}
  \left|(t+\lambda)\varphi_\lambda(t)\right|
  \le E.
\]
No lower bound on \(\varrho_\lambda\) is required, and
\[
  \delta_\lambda
  =
  O_{\mathbb{P}} \xk{
    \sqrt{
      \frac{\mathcal{L}_\lambda\ell_\lambda}{n}
    }
  }
  =o_{\mathbb{P}}(1).
\]
\end{corollary}

\begin{proof}
Write
\[
  N_\lambda=\mathcal{N}_1(\lambda),
  \qquad
  Q_\lambda=\mathcal{N}_q(\lambda),
  \qquad
  a_\lambda=\frac{\mathcal{L}_\lambda \ell_\lambda}{n}.
\]
The regularized-inverse bound gives
\[
  Q_\lambda
  \le
  E^2 \sum_{j\ge1}
  \xk{\frac{\lambda_j}{\lambda_j+\lambda}}^2
  \le
  E^2 \varrho_\lambda N_\lambda.
\]
Since \(Q_\lambda>0\) eventually by assumption, also
\(\varrho_\lambda>0\) eventually.
Moreover, \cref{eq:leverage-variance-compatibility} implies
\[
  \frac{\sqrt{a_\lambda}}{\varrho_\lambda}
  \le
  E^2 \frac{N_\lambda}{Q_\lambda}\sqrt{a_\lambda}
  \longrightarrow0.
\]
Since \(Q_\lambda \le E^2 N_\lambda\), variance compatibility also gives
\(a_\lambda \to0\).

Because \(0<\varrho_\lambda \le1\),
\[
  \ell_\lambda^{\mathrm{int}}
  =
  \log\xk{e+\frac{N_\lambda}{\varrho_\lambda}}
  \le
  \ell_\lambda+\log(1/\varrho_\lambda).
\]
Using \(x\log(1/x)\le1/e\) for \(0<x\le1\), we obtain
\[
  \varrho_\lambda \ell_\lambda^{\mathrm{int}}
  \lesssim
  \ell_\lambda.
\]
Thus the square root term in
\Cref{prop:covariance-concentration} is
\(O(\sqrt{a_\lambda})\), while the linear term satisfies
\begin{align*}
  \frac{
    \mathcal{L}_\lambda \ell_\lambda^{\mathrm{int}}/n
  }{
    \sqrt{a_\lambda}
  }
  &=
  \frac{
    \sqrt{a_\lambda}\ell_\lambda^{\mathrm{int}}
  }{
    \ell_\lambda
  }
  \\
  &\le
  \sqrt{a_\lambda}
  +
  \frac{\sqrt{a_\lambda}}{\varrho_\lambda}
  \frac{
    \varrho_\lambda \log(1/\varrho_\lambda)
  }{
    \ell_\lambda
  }
  =o(1).
\end{align*}
\Cref{prop:covariance-concentration} therefore gives the stated
rate.
Since \(a_\lambda \to0\), this rate is \(o_{\mathbb{P}}(1)\).
\end{proof}

\subsection{Fluctuation of the empirical right-hand side}
\label{subsec:rhs-fluctuation}

\begin{lemma}[Fluctuation of the empirical right-hand side]
\label{lem:rhs-fluctuation}
Let
\[
  h_\lambda
  \coloneqq
  \varphi_\lambda(T)S^*f^*,
  \qquad
  r_\lambda
  \coloneqq
  f^*-Sh_\lambda
  =
  \psi_\lambda(L)f^*.
\]
Define
\[
  Z_\lambda
  \coloneqq
  \xk{\widetilde{g}_X-T_X h_\lambda }
  -
  \xk{S^*f^*-Th_\lambda }.
\]
Under \Cref{ass:basic-model}, for any bounded Borel filter \(\varphi_\lambda\),
\[
  \norm{T_\lambda^{-1/2} Z_\lambda}_{\mathcal{H}}
  =
  O_{\mathbb{P}} \xk{
    \sqrt{\frac{\mathcal{L}_\lambda}{n}}
    \norm{\psi_\lambda(L)f^*}_{L^2(\mu)}
  }.
\]
In particular, if \(\mathcal{L}_\lambda/n\to0\),
\[
  \norm{T_\lambda^{-1/2} Z_\lambda}_{\mathcal{H}}
  =
  o_{\mathbb{P}} \xk{
    \norm{\psi_\lambda(L)f^*}_{L^2(\mu)}
  }.
\]
\end{lemma}

\begin{proof}
The identity \(Sh_\lambda=q_\lambda(L)f^*\) follows from
\Cref{eq:spectral-intertwining}, and hence
\(r_\lambda=\psi_\lambda(L)f^*\).
The decomposition using the sampling operator, used in similar form in analyses of kernel ridge regression such as
\citet{li2023_AsymptoticLearning,zhang2023_OptimalityMisspecified,zhang2025_OptimalRates}, gives
\[
  Z_\lambda
  =
  \frac{1}{n}\sum_{i=1}^n
  \zk{
    k_{x_i} r_\lambda(x_i)-S^*r_\lambda
  }.
\]
Indeed, with \(r_{\lambda,X}=(r_\lambda(x_1),\dots,r_\lambda(x_n))\),
\(\widetilde{g}_X-T_X h_\lambda=S_X^*r_{\lambda,X}\), while
\(S^*f^*-Th_\lambda=S^*r_\lambda\).
Set
\[
  \xi_i
  =
  T_\lambda^{-1/2}
  \zk{
    k_{x_i} r_\lambda(x_i)-S^*r_\lambda
  }.
\]
Then \(\E \xi_i=0\), and by the definition of \(\mathcal{L}_\lambda\),
\[
  \E\norm{\xi_i}_{\mathcal{H}}^2
  \le
  \E\zk{
    \norm{T_\lambda^{-1/2} k_x}_{\mathcal{H}}^2
    r_\lambda(x)^2
  }
  \le
  \mathcal{L}_\lambda\norm{r_\lambda}_{L^2(\mu)}^2.
\]
Therefore
\[
  \E\norm{T_\lambda^{-1/2} Z_\lambda}_{\mathcal{H}}^2
  \le
  \frac{\mathcal{L}_\lambda}{n}
  \norm{\psi_\lambda(L)f^*}_{L^2(\mu)}^2,
\]
and Markov's inequality gives the stochastic bound.
The final estimate follows whenever \(\mathcal{L}_\lambda/n\to0\).
Under the filter assumption, this condition follows from
\cref{eq:leverage-variance-compatibility} by
\Cref{cor:variance-compatible-concentration}.

\end{proof}

\subsection{Loewner and mixed perturbation bounds}
\label{subsec:loewner-comparisons}

\Needspace{8\baselineskip}
\begin{lemma}[Comparison of regularized covariance operators]
\label{lem:regularized-loewner-comparison}
On the event \(\{\delta_\lambda<1\}\),
\[
  (1-\delta_\lambda)T_\lambda
  \preceq
  T_{X,\lambda}
  \preceq
  (1+\delta_\lambda)T_\lambda.
\]
In particular, on \(\{\delta_\lambda \le1/2\}\),
\[
  \frac{1}{2}T_\lambda
  \preceq
  T_{X,\lambda}
  \preceq
  \frac{3}{2}T_\lambda.
\]
\end{lemma}

\begin{proof}
By definition of \(\delta_\lambda\),
\[
  -\delta_\lambda I
  \preceq
  T_\lambda^{-1/2}(T_X-T)T_\lambda^{-1/2}
  \preceq
  \delta_\lambda I.
\]
Conjugating by \(T_\lambda^{1/2}\) gives
\[
  -\delta_\lambda T_\lambda
  \preceq
  T_X-T
  \preceq
  \delta_\lambda T_\lambda.
\]
Adding \(T_\lambda\) proves the claim.
\end{proof}

\Needspace{8\baselineskip}
\begin{lemma}[Bound on the empirical effective dimension]
\label{lem:empirical-effective-dimension}
On \(\{\delta_\lambda \le1/4\}\),
\[
  \widehat{\mathcal{N}}_{1,X}(\lambda)
  \coloneqq
  \Tr\xk{T_X T_{X,\lambda}^{-1} }
  =
  O_{\mathbb{P}}\xk{\mathcal{N}_1(\lambda)}.
\]
\end{lemma}

\begin{proof}
By \Cref{lem:regularized-loewner-comparison},
\[
  T_{X,\lambda}^{-1}
  \preceq
  2T_\lambda^{-1}
\]
on \(\{\delta_\lambda \le1/4\}\).
Hence
\[
  \widehat{\mathcal{N}}_{1,X}(\lambda)
  =
  \Tr\xk{T_X^{1/2} T_{X,\lambda}^{-1} T_X^{1/2} }
  \le
  2\Tr\xk{T_\lambda^{-1/2} T_X T_\lambda^{-1/2} }.
\]
The last trace equals
\[
  \frac{1}{n}\sum_{i=1}^n
  \norm{T_\lambda^{-1/2} k_{x_i}}_{\mathcal{H}}^2,
\]
whose expectation is \(\mathcal{N}_1(\lambda)\).
Markov's inequality gives the claimed stochastic order.
\end{proof}

\begin{lemma}[Bounds for the mixed perturbation]
\label{lem:mixed-perturbation}
On \(\{\delta_\lambda \le1/4\}\), define
\[
  K_\lambda
  \coloneqq
  T_{X,\lambda}^{-1/2}(T_X-T)T_\lambda^{-1/2}.
\]
Then
\[
  \norm{K_\lambda}\le2\delta_\lambda,
  \qquad
  \norm{K_\lambda}_{\mathfrak{S}_2} \le2\eta_\lambda.
\]
\end{lemma}

\begin{proof}
Factor
\[
  K_\lambda
  =
  T_{X,\lambda}^{-1/2} T_\lambda^{1/2}
  \bigl[
    T_\lambda^{-1/2}(T_X-T)T_\lambda^{-1/2}
  \bigr].
\]
By \Cref{lem:regularized-loewner-comparison},
\[
  T_{X,\lambda}^{-1}
  \preceq
  (1-\delta_\lambda)^{-1}T_\lambda^{-1},
\]
and therefore
\[
  \norm{T_{X,\lambda}^{-1/2} T_\lambda^{1/2}}
  \le
  (1-\delta_\lambda)^{-1/2}
  \le2
\]
on \(\{\delta_\lambda \le1/4\}\).
The operator-norm and Hilbert--Schmidt bounds now follow from the ideal property of
\(\mathfrak{S}_2\).
\end{proof}
\clearpage{}
\clearpage{}\section{General Risk Comparison Bounds}
\label{sec:common-risk-transfer}

This section gives deterministic risk comparison bounds that will be combined
with the Schatten perturbation estimates.
We use the filter, risk, and effective dimension notation from
\Cref{subsec:spectral-algorithms,subsec:sequence-model}.
Throughout this section, the filter profile \(\varphi_\lambda\) is Borel and
satisfies
\[
  0\le q_\lambda\le1,
  \qquad
  \sup_{0\le t\le\kappa^2}
  (t+\lambda)\varphi_\lambda(t)
  \le E
\]
for a constant \(E<\infty\) independent of \(n\) and \(\lambda\).
For a fixed target \(f^*\), abbreviate
\[
  R_\lambda^2
  =
  \mathcal{B}_n(q_\lambda;f^*),
  \qquad
  V_\lambda
  =
  \frac{\sigma^2}{n}\mathcal{N}_q(\lambda).
\]
For later comparison, write
\[
  G_\lambda
  =
  \varphi_\lambda(T_X)\psi_\lambda(T)
  -
  \psi_\lambda(T_X)\varphi_\lambda(T),
  \qquad
  d_{\lambda,X}
  =
  \norm{q_\lambda(T_X)-q_\lambda(T)}_{\mathfrak{S}_2}.
\]
\begin{lemma}[Monotonicity of the conditional variance]
\label{lem:variance-monotonicity}
If \(0\le \varphi_1 \le \varphi_2\) on the spectrum of \(T_X\), then
\[
  \mathsf{V}_X(\varphi_1)
  \le
  \mathsf{V}_X(\varphi_2).
\]
\end{lemma}

\begin{proof}
The operator
\[
  C_X
  =
  T_X \xk{\varphi_2(T_X)^2-\varphi_1(T_X)^2}
\]
is positive semidefinite by functional calculus.
Therefore
\[
  \mathsf{V}_X(\varphi_2)-\mathsf{V}_X(\varphi_1)
  =
  \frac{\sigma^2}{n}\Tr(TC_X)
  =
  \frac{\sigma^2}{n}\Tr(T^{1/2} C_X T^{1/2})
  \ge0.
\]
\end{proof}

For \(0<\alpha\le1\), define the weight for the bias comparison
\begin{equation}
  \label{eq:source-bias-weight}
  a_{\lambda,\alpha}(t)
  \coloneqq
  \psi_\lambda(t)
  +
  \xk{
    \frac{\lambda}{t+\lambda}
  }^{\alpha/2}.
\end{equation}

\begin{lemma}[General variance comparison bound]
\label{lem:common-variance-transfer}
Assume \(\sigma^2>0\) and
\(0<\mathcal{N}_q(\lambda)<\infty\) for all sufficiently large \(n\).
On \(\{\delta_\lambda \le1/4\}\),
\begin{align}
  \label{eq:common-variance-transfer}
  \left|
    \frac{n}{\sigma^2}\mathsf{V}_X(\varphi_\lambda)
    -
    \mathcal{N}_q(\lambda)
  \right|
  &\le
  d_{\lambda,X}
  \xk{
    2\sqrt{\mathcal{N}_q(\lambda)}
    +
    d_{\lambda,X}
  }
  \nonumber\\
  &\quad+
  C_E \delta_\lambda
  \xk{
    \xk{
      \sqrt{\mathcal{N}_q(\lambda)}
      +
      d_{\lambda,X}
    }^2
    +
    \widehat{\mathcal{N}}_{1,X}(\lambda)
  }.
\end{align}
Consequently, if
\[
  \delta_\lambda=o_{\mathbb{P}}(1),
  \qquad
  d_{\lambda,X}
  =
  o_{\mathbb{P}}\xk{\sqrt{\mathcal{N}_q(\lambda)}},
  \qquad
  \delta_\lambda
  \widehat{\mathcal{N}}_{1,X}(\lambda)
  =
  o_{\mathbb{P}}\xk{\mathcal{N}_q(\lambda)},
\]
then
\[
  \mathsf{V}_X(\varphi_\lambda)
  =
  \xk{1+o_{\mathbb{P}}(1)}
  \frac{\sigma^2}{n}\mathcal{N}_q(\lambda).
\]
\end{lemma}

\begin{proof}
Let
\[
  V_X
  =
  \Tr\xk{T\varphi_\lambda(T_X)^2T_X },
  \qquad
  W_X
  =
  \Tr q_\lambda(T_X)^2.
\]
The definition of the conditional variance and the identity
\(
T_X=n^{-1} \sum_i k_{x_i} \otimes k_{x_i}
\)
give
\[
  \mathsf{V}_X(\varphi_\lambda)
  =
  \frac{\sigma^2}{n}V_X.
\]
Set
\[
  C_X
  =
  \varphi_\lambda(T_X)T_X \varphi_\lambda(T_X).
\]
Then \(V_X=\Tr(TC_X)\), \(W_X=\Tr(T_X C_X)\), and
\[
  |V_X-W_X|
  \le
  \delta_\lambda \Tr(T_\lambda C_X)
  =
  \delta_\lambda
  \xk{
    V_X+\lambda\Tr(C_X)
  }.
\]
The regularized inverse bound gives
\[
  \lambda\Tr(C_X)
  =
  \sum_b
  \lambda b\varphi_\lambda(b)^2
  \le
  E^2
  \sum_b
  \frac{\lambda b}{(b+\lambda)^2}
  \le
  E^2 \widehat{\mathcal{N}}_{1,X}(\lambda),
\]
where the sums run over the eigenvalues of \(T_X\), with multiplicity.
Since \(\delta_\lambda \le1/4\), rearranging yields
\[
  |V_X-W_X|
  \le
  C_E \delta_\lambda
  \xk{
    W_X+\widehat{\mathcal{N}}_{1,X}(\lambda)
  }.
\]

Moreover,
\[
  \left|
    \sqrt{W_X}
    -
    \sqrt{\mathcal{N}_q(\lambda)}
  \right|
  \le
  d_{\lambda,X},
\]
and hence
\[
  |W_X-\mathcal{N}_q(\lambda)|
  \le
  d_{\lambda,X}
  \xk{
    2\sqrt{\mathcal{N}_q(\lambda)}
    +
    d_{\lambda,X}
  },
\]
while
\[
  W_X
  \le
  \xk{
    \sqrt{\mathcal{N}_q(\lambda)}
    +
    d_{\lambda,X}
  }^2.
\]
Combining the last three bounds proves
\cref{eq:common-variance-transfer}.
The asymptotic conclusion follows by dividing by
\(\mathcal{N}_q(\lambda)\).
\end{proof}

\begin{lemma}[Scalar bound for relative risk differences]
\label{lem:scalar-quadratic-risk-transfer}
Let \(\Theta\) be a nonempty index set, let \(V>0\), and let
\(u,v\ge0\) be independent of \(\theta\in\Theta\).
If \(b_\theta,d_\theta \ge0\) satisfy
\[
  d_\theta \le u b_\theta+v,
  \qquad
  \theta\in\Theta,
\]
then
\[
  \sup_{\theta\in\Theta}
  \frac{
    2b_\theta d_\theta+d_\theta^2
  }{
    b_\theta^2+V
  }
  \le
  2u+u^2
  +
  (1+u)\frac{v}{\sqrt V}
  +
  \frac{v^2}{V}.
\]
\end{lemma}

\begin{proof}
For every \(\theta\), monotonicity in \(d_\theta\) gives
\[
  2b_\theta d_\theta+d_\theta^2
  \le
  (2u+u^2)b_\theta^2
  +
  2(1+u)b_\theta v
  +
  v^2.
\]
Divide by \(b_\theta^2+V\) and use
\[
  \frac{b_\theta^2}{b_\theta^2+V}\le1,
  \qquad
  \frac{2b_\theta}{b_\theta^2+V}\le\frac{1}{\sqrt V},
  \qquad
  \frac{1}{b_\theta^2+V}\le\frac{1}{V}.
\]
\end{proof}

\begin{proposition}[Bias comparison via an interaction bound]
\label{prop:common-bias-transfer}
Assume \Cref{ass:basic-model}.
Let \(f^*\in\overline{\Ran(L)}\), and fix \(0<\alpha\le1\).
Suppose that, on \(\{\delta_\lambda \le1/4\}\), an \(X\)-dependent
\(\beta_\lambda \ge0\) satisfies
\begin{equation}
  \label{eq:interaction-certificate}
  \norm{SG_\lambda S^*f^*}_{L^2(\mu)}
  \le
  \beta_\lambda
  \norm{a_{\lambda,\alpha}(L)f^*}_{L^2(\mu)}.
\end{equation}
Suppose that \(0<\mathcal{N}_q(\lambda)<\infty\),
\(\delta_\lambda=o_{\mathbb{P}}(1)\), and \(\mathcal L_\lambda/n\to0\).
If
\begin{equation}
  \label{eq:pointwise-bias-transfer-conditions}
  \beta_\lambda^2
  \norm{a_{\lambda,\alpha}(L)f^*}_{L^2(\mu)}^2
  =
  o_{\mathbb{P}}\xk{
    V_\lambda
  },
\end{equation}
then
\begin{equation}
  \label{eq:common-pointwise-bias-transfer}
  \mathsf{B}_X(q_\lambda)
  =
  \mathcal{B}_n(q_\lambda;f^*)
  +
  o_{\mathbb{P}}\xk{
    \mathcal{E}_n^{\mathsf{seq}}(q_\lambda;f^*)
  }.
\end{equation}

\end{proposition}

\begin{proof}[Proof of \Cref{prop:common-bias-transfer}]
It suffices to work on \(\{\delta_\lambda \le1/4\}\), whose probability tends
to one by assumption.
Set
\[
  h_\lambda
  =
  \varphi_\lambda(T)S^*f^*,
  \qquad
  r_\lambda
  =
  \psi_\lambda(L)f^*,
  \qquad
  e_\lambda
  =
  S\xk{
    \varphi_\lambda(T_X)\widetilde{g}_X-h_\lambda
  }.
\]
Then
\[
  S\varphi_\lambda(T_X)\widetilde{g}_X-f^*
  =
  -r_\lambda+e_\lambda,
\]
so
\[
  \left|
    \mathsf{B}_X(q_\lambda)-R_\lambda^2
  \right|
  \le
  2R_\lambda \norm{e_\lambda}_{L^2(\mu)}
  +
  \norm{e_\lambda}_{L^2(\mu)}^2.
\]

With \(Z_\lambda\) from \Cref{lem:rhs-fluctuation}, direct algebra gives
\[
  \varphi_\lambda(T_X)\widetilde{g}_X-h_\lambda
  =
  \varphi_\lambda(T_X)Z_\lambda
  +
  G_\lambda S^*f^*.
\]
The Loewner comparison and the regularized inverse bound imply
\begin{align*}
  \norm{
    T^{1/2} \varphi_\lambda(T_X)T_\lambda^{1/2}
  }
  &=
  \norm{
    T^{1/2} T_{X,\lambda}^{-1/2}
    \xk{
      T_{X,\lambda}^{1/2}
      \varphi_\lambda(T_X)
      T_{X,\lambda}^{1/2}
    }
    T_{X,\lambda}^{-1/2} T_\lambda^{1/2}
  }
  \\
  &\le
  \norm{T^{1/2} T_{X,\lambda}^{-1/2}}
  \norm{
    T_{X,\lambda}^{1/2}
    \varphi_\lambda(T_X)
    T_{X,\lambda}^{1/2}
  }
  \norm{T_{X,\lambda}^{-1/2} T_\lambda^{1/2}}
  \\
  &\le
  C_E.
\end{align*}
The middle factor is controlled by the regularized inverse bound, and the two
outside factors by \Cref{lem:regularized-loewner-comparison}.
Therefore
\[
  \norm{
    S\varphi_\lambda(T_X)Z_\lambda
  }_{L^2(\mu)}
  \le
  C_E
  \norm{
    T_\lambda^{-1/2} Z_\lambda
  }_{\mathcal{H}}.
\]
\Cref{lem:rhs-fluctuation} therefore gives a random
\[
  \omega_\lambda
  =
  O_{\mathbb{P}}\xk{
    \sqrt{
      \frac{\mathcal L_\lambda}{n}
    }
  }
  =
  o_{\mathbb{P}}(1)
\]
such that
\[
  \norm{
    S\varphi_\lambda(T_X)Z_\lambda
  }_{L^2(\mu)}
  \le
  \omega_\lambda R_\lambda.
\]
If \(R_\lambda=0\), the proof of
\Cref{lem:rhs-fluctuation} gives \(Z_\lambda=0\) almost surely, and we take
\(\omega_\lambda=0\).
Using \cref{eq:interaction-certificate}, we obtain
\[
  \norm{e_\lambda}_{L^2(\mu)}
  \le
  \omega_\lambda R_\lambda
  +
  \beta_\lambda
  \norm{a_{\lambda,\alpha}(L)f^*}_{L^2(\mu)}.
\]
Thus the preceding bound holds with
\[
  u_\lambda=\omega_\lambda,
  \qquad
  c_\lambda
  =
  \beta_\lambda
  \norm{a_{\lambda,\alpha}(L)f^*}_{L^2(\mu)}.
\]
The bias expansion above and
\Cref{lem:scalar-quadratic-risk-transfer}, applied to the singleton index set,
give
\[
  \frac{
    \left|
      \mathsf{B}_X(q_\lambda)-R_\lambda^2
    \right|
  }{
    R_\lambda^2+V_\lambda
  }
  \le
  2u_\lambda+u_\lambda^2
  +
  (1+u_\lambda)\frac{c_\lambda}{\sqrt{V_\lambda}}
  +
  \frac{c_\lambda^2}{V_\lambda}.
\]
The bound on \(\omega_\lambda\) gives \(u_\lambda=o_{\mathbb{P}}(1)\), while
\cref{eq:pointwise-bias-transfer-conditions} gives
\(c_\lambda^2/V_\lambda=o_{\mathbb{P}}(1)\).
Hence \(c_\lambda/\sqrt{V_\lambda}=o_{\mathbb{P}}(1)\), proving
\cref{eq:common-pointwise-bias-transfer}.

\end{proof}
\clearpage{}
\clearpage{}\section{Scale-Regular Lipschitz Filters via Schatten Class Bounds}
\label{sec:lipschitz-schatten}

This section proves \Cref{thm:conditional-risk-equivalence}.
Scalar Lipschitz regularity does not give a uniform DOI bound on
$\mathfrak{B}(\mathcal{H})$, but it does give a bound on every finite Schatten class.
Optimizing the Schatten exponent against the operator and Hilbert--Schmidt
concentration scales yields the comparison rate in
\cref{eq:lipschitz-schatten-scale}.

Recall that for a Hilbert space $\mathcal{H}$, we denote by $\mathfrak{B}=\mathfrak{B}(\mathcal{H})$ the space of all bounded operators on $\mathcal{H}$. For any spectral measures $E$ and $F$, denote by
$\mathfrak{M}_{\mathfrak{B}}=\mathfrak{M}_{\mathfrak{B}}(E,F)$ the space of all DOI kernel functions $\phi$ such that $\operatorname{DOI}_\phi^{E,F}$ is a bounded operator on $\mathfrak{B}$, with norm
\[
    \norm{\phi}_{\mathfrak{M}_{\mathfrak{B}}}
    =
    \norm{\operatorname{DOI}_\phi^{E,F}}_{\mathfrak{B}}.
\]

Likewise, for any spectral measures $E$, $F$, and any $1\leq r<\infty$, we denote by
$\mathfrak{M}_r=\mathfrak{M}_r(E,F)$ the space of all DOI kernel functions $\phi$ such that $\operatorname{DOI}_\phi^{E,F}$ is a bounded operator on the Schatten class $\mathfrak{S}_r=\mathfrak{S}_r(\mathcal{H})$, with norm
\[
    \norm{\phi}_{\mathfrak{M}_r}
    =
    \norm{\operatorname{DOI}_\phi^{E,F}}_{\mathfrak{S}_r}.
\]

The transformed size bounds in \Cref{ass:lipschitz-filter} give
\begin{equation}
  \label{eq:lipschitz-filter-basic-consequences}
  \sup_{0\le x\le\kappa^2}
  (x+\lambda)\varphi_\lambda(x)
  \le A,
  \qquad
  0\le\psi_\lambda(x)
  \le
  A\xk{\frac{\lambda}{x+\lambda}}^{\rho_\psi}.
\end{equation}
In particular, the residual has qualification at least
\(\rho_\psi\), and
\(\mathcal{N}_q(\lambda)<\infty\).

\begin{lemma}[Contractivity of the diagonal DOI]
\label{lem:contraction_on_diag}
Let $\mathcal H$ be a separable Hilbert space, and let
\[
E,F:\mathcal B(\mathbb R)\longrightarrow \mathfrak B(\mathcal H)
\]
be two spectral measures. Denote $D=\{(t,s)\in\mathbb R^2:t=s\}$, and for $X\in\mathfrak S_2(\mathcal H)$, define
\[
P_D(X):=
\operatorname{DOI}^{E,F}_{\mathbf 1_D}(X).
\]
For every \(1\le r<\infty\), the restriction of \(P_D\) to finite-rank operators
extends uniquely to a contraction on \(\mathfrak S_r(\mathcal H)\), agreeing
with the original DOI on \(\mathfrak S_2\cap\mathfrak S_r\):
\[
\|P_D(X)\|_{\mathfrak S_r}
\le
\|X\|_{\mathfrak S_r},
\qquad
X\in\mathfrak S_r(\mathcal H).
\]
\end{lemma}

\begin{proof}
Let
\[
\mathcal A
:=
\{\tau\in\mathbb R:
E(\{\tau\})\neq0,\;
F(\{\tau\})\neq0\}
\]
be the set of common atoms of $E$ and $F$. Note that $\mathcal H$ is separable, and the nonzero projections
$E(\{\tau\})$ corresponding to distinct atoms have mutually orthogonal
ranges, and similarly for $F$, hence $\mathcal A$ is at most countable.

Write $E_\tau:=E(\{\tau\})$, $
F_\tau:=F(\{\tau\})$. For any finite subset $A\subset\mathcal A$, define
\[
T_A(X):=
\sum_{\tau\in A}E_\tau X F_\tau.
\]
Since the projections $\{E_\tau\}_{\tau\in A}$ and
$\{F_\tau\}_{\tau\in A}$ are pairwise orthogonal, for any
$X\in\mathfrak B(\mathcal H)$,
\[
\begin{aligned}
\|T_A(X)\xi\|^2
&=
\sum_{\tau\in A}
\|E_\tau X F_\tau\xi\|^2
\le
\sum_{\tau\in A}\|XF_\tau\xi\|^2
\\
&\le
\|X\|^2
\sum_{\tau\in A}\|F_\tau\xi\|^2
=
\|X\|^2\left\|\sum_{\tau\in A}F_\tau\xi\right\|^2
\le
\|X\|^2\|\xi\|^2.
\end{aligned}
\]
Therefore, 
\[
\|T_A\|_{\mathfrak B\to\mathfrak B}\le1.
\]

Next, we use the duality between the trace class $\mathfrak S_1$ and $\mathfrak S_\infty=\mathfrak B$. For any $X\in\mathfrak S_1$, $Y\in\mathfrak B$, we have 
\[
\operatorname{Tr}(T_A(X)Y)
=
\operatorname{Tr}(XT_A^\sharp(Y)),\quad T_A^\sharp(Y)
=
\sum_{\tau\in A}F_\tau Y E_\tau.
\]
Thus the map dual to $T_A$ under the trace pairing is $T_A^\sharp$.
The same argument as for $T_A$ gives
\[
\|T_A^\sharp\|_{\mathfrak B\to\mathfrak B}\le1,
\]
so the dual characterization of the trace norm gives
\[
  \norm{T_A(X)}_{\mathfrak S_1}
  =
  \sup_{\|Y\|\leq 1}
  |\operatorname{Tr}(T_A(X)Y)|
  =
  \sup_{\|Y\|\leq 1}
  |\operatorname{Tr}(XT_A^\sharp(Y))|
  \leq 
  \sup_{\|Y\|\leq 1}
  \norm{X}_{\mathfrak S_1}\|T_A^\sharp(Y)\|\leq\norm{X}_{\mathfrak S_1}.
\]
Equivalently,
\[
  \norm{T_A}_{\mathfrak S_1\to\mathfrak S_1}\leq 1.
\]
Interpolating between $\mathfrak S_1$ and
$\mathfrak S_\infty=\mathfrak B(\mathcal H)$ yields
\[
  \norm{T_A}_{\mathfrak S_r\to\mathfrak S_r}\leq 1,\quad 1\leq r<\infty.
\]

We next identify the diagonal DOI on \(\mathfrak S_2\).
For a rank-one operator \(X=|u\rangle\langle v|\), set
\(\mu_u(B)=\|E(B)u\|^2\) and \(\nu_v(B)=\|F(B)v\|^2\).
The scalar spectral measure of \(X\) for the commuting left and right
spectral actions on \(\mathfrak S_2\) is \(\mu_u\otimes\nu_v\).
By Fubini's theorem,
\[
  (\mu_u\otimes\nu_v)(D)
  =\int\nu_v(\{t\})\,d\mu_u(t)
  =\sum_{\tau\in\mathcal A}
    \mu_u(\{\tau\})\nu_v(\{\tau\}).
\]
Thus this measure gives no mass to \(D\setminus D_{\rm at}\), where
\(D_{\rm at}=\{(\tau,\tau):\tau\in\mathcal A\}\).
By linearity and finite-rank density, the corresponding DOI vanishes
on all of \(\mathfrak S_2\).
Choose increasing finite sets \(A_N\) exhausting \(\mathcal A\), taking
\(A_N=\mathcal A\) if \(\mathcal A\) is finite, including the empty case.
Countable additivity of the spectral measure on \(\mathfrak S_2\) gives
\[
  T_{A_N}(X)\longrightarrow P_D(X)
  \quad\text{in }\mathfrak S_2,
  \qquad X\in\mathfrak S_2.
\]

To obtain convergence in every finite Schatten class, first let
\(X=|u\rangle\langle v|\).
Orthogonality and Cauchy--Schwarz give
\[
  \sum_{\tau\in\mathcal A}\|E_\tau X F_\tau\|_{\mathfrak S_1}
  =\sum_{\tau\in\mathcal A}\|E_\tau u\|\,\|F_\tau v\|
  \le
  \left(\sum_{\tau\in\mathcal A}\|E_\tau u\|^2\right)^{1/2}
  \left(\sum_{\tau\in\mathcal A}\|F_\tau v\|^2\right)^{1/2}
  \le\|u\|\,\|v\|.
\]
Hence the series converges absolutely in \(\mathfrak S_1\), and therefore
in every \(\mathfrak S_r\), to its already identified \(\mathfrak S_2\)-limit.
The same holds for every finite-rank \(X\) by linearity.
Passing to the limit in the contraction estimate for \(T_{A_N}\) gives
\(\|P_D(X)\|_{\mathfrak S_r}\le\|X\|_{\mathfrak S_r}\) on finite-rank operators.
Finite-rank density yields a unique contractive extension to \(\mathfrak S_r\).
Moreover, the uniform contraction bound for \(T_{A_N}\), together with
finite-rank approximation, gives convergence to this extension in
\(\mathfrak S_r\) for every \(X\in\mathfrak S_r\).
For \(X\in\mathfrak S_2\cap\mathfrak S_r\), both limits agree, since convergence
in either Schatten norm implies convergence in operator norm.
This proves the claimed compatibility with the original DOI.
\end{proof}

\begin{lemma}[Divided differences with two spectral measures]
\label{lem:two-measure-lipschitz-schatten}
Let \(f:\mathbb{R}\to\mathbb{R}\) be Lipschitz and define
\[
  f_0^{[1]}(x,y)
  =
  \frac{f(x)-f(y)}{x-y}\mathbf{1}_{\{x\ne y\}}(x,y).
\]
For \(2\le r<\infty\), compactly supported spectral measures \(E_1,E_2\),
and \(H\in\mathfrak{S}_r\),
\[
  \norm{
    \operatorname{DOI}_{f_0^{[1]}}^{E_1,E_2}(H)
  }_{\mathfrak{S}_r}
  \le
  C\frac{r^2}{r-1}
  \operatorname{Lip}(f)
  \norm{H}_{\mathfrak{S}_r}.
\]
The constant is independent of \(r,f,E_1,E_2\), and \(H\).
\end{lemma}

\begin{proof}
We first consider the case of one spectral measure \(E_1=E_2=E\). Let
\(P_D^E=\operatorname{DOI}_{\mathbf{1}_D}^{E,E}\) be the diagonal projection, where
\(D=\{(x,y):x=y\}\) is the diagonal set.
By \Cref{lem:contraction_on_diag}, \(P_D^E\) is a contraction on
\(\mathfrak{S}_r\).

If \(H\in\mathfrak{S}_2 \cap\mathfrak{S}_r=\mathfrak{S}_2\) is off-diagonal (i.e. \(P_D^E H=0\)), then the one-dimensional specialization of
\citet[Theorem~5.1]{caspers2014_BestConstants} gives
\begin{equation}\label{eq:off-diagonal-original-estimate}
  \norm{
    \operatorname{DOI}_{f_0^{[1]}}^{E,E}(H)
  }_{\mathfrak{S}_r}
  \le
  C\frac{r^2}{r-1}
  \operatorname{Lip}(f)
  \norm{H}_{\mathfrak{S}_r},
\end{equation}
where the constant $C>0$ is independent of $r$, $f$, $E$, and $H$. For general $H$, note that
\[
    \operatorname{DOI}^{E,E}_{f_0^{[1]}}\circ\operatorname{DOI}^{E,E}_{\mathbf{1}_D}
    =
    \operatorname{DOI}^{E,E}_{f_0^{[1]} \cdot\mathbf{1}_D}=\operatorname{DOI}^{E,E}_0=0
\]
on $\mathfrak{S}_2$; moreover, since $P_D$ is a contractive projection, we have
\[
    P_D(I-P_D)H=(P_D-P_D^2)H=(P_D-P_D)H=0,
\]
hence 
\begin{equation}\label{eq:non-off-diagonal-estimate}
  \begin{aligned}
    \norm{
    \operatorname{DOI}_{f_0^{[1]}}^{E,E}(H)
    }_{\mathfrak{S}_r}
    &=
    \norm{
    \operatorname{DOI}_{f_0^{[1]}}^{E,E}(P_D H+(I-P_D)H)
    }_{\mathfrak{S}_r}\\
    &=\norm{
    \operatorname{DOI}_{f_0^{[1]}}^{E,E}((I-P_D)H)}_{\mathfrak{S}_r}\\
    &\le
    C\frac{r^2}{r-1}\operatorname{Lip}(f)\norm{(I-P_D)H}_{\mathfrak{S}_r}\\
    &\le
    2C\frac{r^2}{r-1}\operatorname{Lip}(f)\norm{H}_{\mathfrak{S}_r},
  \end{aligned}
\end{equation}
where we use (\ref{eq:off-diagonal-original-estimate}) in the first inequality and the contractivity of $P_D$ in the last one.
Since \(\mathfrak{S}_2\) is dense in \(\mathfrak{S}_r\) for \(r\ge2\), this estimate extends to all \(H\in\mathfrak{S}_r\).

Next, for two spectral measures $E_1$ and $E_2$ on a Hilbert space $\mathcal{H}$, define
$\widetilde{E}=E_1 \oplus E_2$, which is a spectral measure on $\widetilde{\mathcal{H}}=\mathcal{H}\oplus\mathcal{H}$. Accordingly, define
\[
  \widetilde{H}
  =
  \begin{pmatrix}
    0&H\\
    0&0
  \end{pmatrix}:\widetilde{\mathcal{H}}\to\widetilde{\mathcal{H}},\quad \widetilde{H}(h_1,h_2)=(H(h_2),0),\quad\forall (h_1,h_2)\in\widetilde{\mathcal{H}}.
\]
Then $\widetilde{H}\in\mathfrak{S}_r(\widetilde{\mathcal{H}})$, and $\norm{\widetilde{H}}_{\mathfrak{S}_r(\widetilde{\mathcal{H}})}=\norm{H}_{\mathfrak{S}_r({\mathcal{H}})}$. Applying the extended estimate (\ref{eq:non-off-diagonal-estimate}) to $\widetilde{H}$ and $\widetilde{E}$, we obtain
\[
    \norm{
    \operatorname{DOI}^{\widetilde{E},\widetilde{E}}_{f_0^{[1]}}(\widetilde{H})
    }_{\mathfrak{S}_r(\widetilde{\mathcal{H}})}
    \leq 
    C\frac{r^2}{r-1}\operatorname{Lip}(f)\norm{\widetilde{H}}_{\mathfrak{S}_r(\widetilde{\mathcal{H}})}.
\]
Finally, note that 
\[
    \operatorname{DOI}^{\widetilde{E},\widetilde{E}}_{f_0^{[1]}}(\widetilde{H})
    =
    \begin{pmatrix} 
        0 & \operatorname{DOI}^{E_1,E_2}_{f_0^{[1]}}(H) \\ 
        0 & 0
    \end{pmatrix},
\]
hence 
\[
\begin{aligned}
    \norm{
    \operatorname{DOI}^{E_1,E_2}_{f_0^{[1]}}(H)
    }_{\mathfrak{S}_r(\mathcal{H})}
    &=
    \norm{\operatorname{DOI}^{\widetilde{E},\widetilde{E}}_{f_0^{[1]}}(\widetilde{H})}_{\mathfrak{S}_r(\widetilde{\mathcal{H}})}\\
    &\le
    C\frac{r^2}{r-1}\operatorname{Lip}(f)\norm{\widetilde{H}}_{\mathfrak{S}_r(\widetilde{\mathcal{H}})}\\
    &=
    C\frac{r^2}{r-1}\operatorname{Lip}(f)\norm{H}_{\mathfrak{S}_r(\mathcal{H})}.
\end{aligned}
\]
This proves the lemma.
\end{proof}

\begin{lemma}[Universal Schur multiplier]
\label{lem:universal_schur_multiplier}
    For $0<\alpha\leq 1$, define
    \[
    \vartheta_\alpha(z)
    =
    \frac{
    z e^{(1-\alpha)z/2}
    }{
    2\sinh(z/2)
    },
    \qquad
    \vartheta_\alpha(0)=1,
    \]
    then $\Phi_\alpha(t,s)=\vartheta_\alpha(t-s)$ is a universal Schur multiplier; in other words, for any $r\geq 1$ and any spectral measures $E,F$ defined on $\mathbb{R}$, 
    \[
        \norm{
            \operatorname{DOI}_{\Phi_\alpha}^{E,F}
        }_{\mathfrak{S}_r}
        \leq 
        C_\alpha,
    \]
    where $C_\alpha$ is independent of $r$, $E$ and $F$. 
\end{lemma} 

\begin{proof}
Let 
\[
    \hat{\vartheta}_\alpha(\xi)=\int_{\mathbb{R}} \vartheta_\alpha(z) e^{-iz\xi} dz
\]
be the Fourier transform of $\vartheta_\alpha(z)$. Direct differentiation shows that $\vartheta_\alpha$ is smooth on $\mathbb{R}$ and
\[
    \left|\vartheta_\alpha^{(k)}(z)\right|
    \leq
    C_\alpha(1+|z|)
    \begin{cases}
      e^{-\alpha z/2}, & z\ge0,\\
      e^{(2-\alpha)z/2}, & z<0,
    \end{cases}
    \quad 
    k=0,1,2
\]
for some constant $C_\alpha>0$ depending only on $\alpha$. Therefore, the Fourier transform of $\vartheta_\alpha(z)$ satisfies 
\[\begin{aligned}
    \norm{\hat{\vartheta}_\alpha}_{L^1(\mathbb{R})}&=\int_{\mathbb{R}}|\hat{\vartheta}_\alpha(\xi)|d\xi=\int_{\mathbb{R}}(1+\xi^2)^{-1/2}\cdot(1+\xi^2)^{1/2}|\hat{\vartheta}_\alpha(\xi)|d\xi\\
    &\leq
    \xk{\int_{\mathbb{R}} \frac{d\xi}{1+\xi^2}}^{1/2}
    \cdot
    \xk{\int_{\mathbb{R}}(1+\xi^2)|\hat{\vartheta}_\alpha(\xi)|^2d\xi}^{1/2}\\
    &=\sqrt{\pi}\cdot\zk{2\pi\xk{\norm{\vartheta_\alpha}_{L^2(\mathbb{R})}^2+\norm{\vartheta_\alpha'}_{L^2(\mathbb{R})}^2 }}^{1/2}\\
    &\leq C_\alpha<\infty,
\end{aligned}\]
where the last equality follows from Plancherel's theorem.

Define $d\mu_\alpha(\xi)=\frac{1}{2\pi}\hat{\vartheta}_\alpha(\xi)d\xi$. Then, by the Fourier inversion formula,
\[
    \Phi_\alpha(t,s)=\vartheta_\alpha(t-s)=\int_{\mathbb{R}} e^{i\xi t} e^{-i\xi s} d\mu_\alpha(\xi).
\]
The measure \(\mu_\alpha\) may be complex, but its total variation satisfies
\[
  \|\mu_\alpha\|_{\mathrm{TV}}
  =\frac{1}{2\pi}\|\widehat\vartheta_\alpha\|_{L^1(\mathbb R)}
  \le C_\alpha.
\]
Let \(U_E(\xi)=\int e^{i\xi t}\,dE(t)\) and
\(U_F(\xi)=\int e^{i\xi s}\,dF(s)\), which are unitary operators.
For \(X\in\mathfrak S_2\), Fourier inversion and Fubini's theorem give
\[
  \operatorname{DOI}_{\Phi_\alpha}^{E,F}(X)
  =\int_{\mathbb R}U_E(\xi)XU_F(\xi)^*\,d\mu_\alpha(\xi).
\]
For every \(1\le r<\infty\) and \(X\in\mathfrak S_r\), the integrand is
continuous in \(\mathfrak S_r\), as follows first for finite-rank operators
from strong continuity of the unitaries and then by finite-rank approximation.
Thus the integral defines a bounded map on \(\mathfrak S_r\), with
\[
  \left\|\int_{\mathbb R}U_E(\xi)XU_F(\xi)^*\,d\mu_\alpha(\xi)\right\|_{\mathfrak S_r}
  \le\int_{\mathbb R}\|X\|_{\mathfrak S_r}\,d|\mu_\alpha|(\xi)
  \le C_\alpha\|X\|_{\mathfrak S_r}.
\]
It agrees with the Hilbert--Schmidt DOI on the intersection.
The same integral, interpreted in the weak operator sense, gives the
operator-norm bound on \(\mathfrak B(\mathcal H)\).
This proves the lemma.
\end{proof}
\begin{theorem}[One-dimensional HMS Schur multiplier]
\label{thm:external-hms-schur-multiplier}
Let \(M\in C^1(\mathbb R^2\setminus\{t=s\})\), and set
\[
  \mathcal H_1(M)
  :=
  \sum_{k=0}^1
  \left\|
    |t-s|^k
    \left(
      |\partial_t^kM(t,s)|
      +
      |\partial_s^kM(t,s)|
    \right)
  \right\|_{L^\infty(\mathbb R^2\setminus\{t=s\})}.
\]
If \(\mathcal H_1(M)<\infty\), then, for every \(1<r<\infty\), the Schur
multiplier \(S_M\), initially defined on
\(\mathfrak S_2(L^2(\mathbb R))\cap\mathfrak S_r(L^2(\mathbb R))\),
extends to a completely bounded map on \(\mathfrak S_r(L^2(\mathbb R))\),
and
\[
  \left\|
    S_M:
    \mathfrak{S}_r(L^2(\mathbb{R}))
    \to
    \mathfrak{S}_r(L^2(\mathbb{R}))
  \right\|_{\mathrm{cb}}
  \leq
  C\frac{r^2}{r-1}\mathcal H_1(M).
\]
This is the \(n=1\) specialization of Theorem~A and Remark~2.1 of
\citet{condeAlonso2023_SchurMultipliersSchatten}.
\end{theorem}

\begin{lemma}[Scalar multiplier in logarithmic coordinates]
\label{lem:log-spectral-schatten-multipliers}
For \(0<\alpha\le1\), define
\[
  \chi_\alpha(z)=\frac{e^{-z/2}-e^{(1-\alpha)z/2}}{2\sinh(z/2)},\quad \chi_\alpha(0)=\frac{\alpha}{2}-1.
\]
Then, for any spectral measures $E$ and $F$ with compact support, the multiplier function 
\[
  \chi_{\alpha,0}(t,s)
  =
  \chi_\alpha(t-s)\mathbf{1}_{\{t\ne s\}}(t,s)
\]
satisfies
\[
  \norm{
    \chi_{\alpha,0}
  }_{\mathfrak{M}_r}
  \le
  C_\alpha r,
  \qquad
  2\le r<\infty,
\]
where the constant $C_\alpha>0$ is independent of $r$, $E$, and $F$. 
\end{lemma}

\begin{proof}
The singularity of \(\chi_\alpha\) at zero is removable, and its Taylor expansion is
\[
  \chi_\alpha(z)=\frac{\alpha}{2}-1
  +\frac{\alpha(2-\alpha)}8z+O_\alpha(z^2).
\]
For the tails, the identities
\[
  \chi_\alpha(z)
  =\frac{e^{-z}-e^{-\alpha z/2}}{1-e^{-z}},\qquad
  \chi_\alpha(-w)+1
  =\frac{e^{\alpha w/2}-1}{e^w-1}
  \quad(z,w>0)
\]
give, upon differentiation,
\[
  \begin{aligned}
    |\chi_\alpha(z)|+|\chi_\alpha'(z)|
    &\le C_\alpha e^{-\alpha z/2}, &&z\ge1,\\
    |\chi_\alpha(z)+1|+|\chi_\alpha'(z)|
    &\le C_\alpha e^{(1-\alpha/2)z}, &&z\le-1.
  \end{aligned}
\]
Consequently, for \(M(t,s)=\chi_\alpha(t-s)\),
\[
  \mathcal H_1(M)
  \leq
  2\sup_{z\in\mathbb R}
  \left(
    |\chi_\alpha(z)|
    +
    |z\chi_\alpha'(z)|
  \right)
  \leq C_\alpha.
\]
Thus \Cref{thm:external-hms-schur-multiplier} yields
\begin{equation}\label{eq:log-spectral-cb-input}
  \left\|
    S_{\chi_\alpha(t-s)}:
    \mathfrak{S}_r(L^2(\mathbb{R}))
    \to
    \mathfrak{S}_r(L^2(\mathbb{R}))
  \right\|_{\mathrm{cb}}
  \leq
  C_\alpha \frac{r^2}{r-1},
  \qquad
  2\le r<\infty.
\end{equation}
Here, \(\|\cdot\|_{\mathrm{cb}}\) denotes the completely bounded norm,
so the estimate holds with the same constant after every finite matrix amplification.
Since \(r^2/(r-1)\leq2r\) for \(r\geq2\), we absorb this factor into
\(C_\alpha r\) below.

\emph{Step 1: finite point sets.}
Let \(t_1,\ldots,t_m\) and \(s_1,\ldots,s_n\) be distinct within
each list, and let
\[
  A=(A_{ij})_{i\le m,j\le n}
  \in
  \mathfrak{S}_r(
    \mathbb{C}^n\otimes\mathbb{C}^d,
    \mathbb{C}^m\otimes\mathbb{C}^d
  ).
\]
Choose interval neighborhoods \(U_i^\varepsilon\ni t_i\) and
\(V_j^\varepsilon\ni s_j\) whose lengths tend to zero as \(\varepsilon\to0\).
Require the intervals to be pairwise disjoint within each family;
intervals from different families may overlap.
With
\[
  e_i^\varepsilon
  :=
  |U_i^\varepsilon|^{-1/2}\mathbf{1}_{U_i^\varepsilon},
  \qquad
  f_j^\varepsilon
  :=
  |V_j^\varepsilon|^{-1/2}\mathbf{1}_{V_j^\varepsilon},
\]
define isometric embeddings 
\[
  \mathcal U_\varepsilon:\mathbb C^m\to L^2(\mathbb R),
  \quad
  \delta_i\mapsto e_i^\varepsilon,
  \qquad
  \mathcal V_\varepsilon:\mathbb C^n\to L^2(\mathbb R),
  \quad
  \delta_j\mapsto f_j^\varepsilon.
\]
Then, the finite-rank operator
\[
  \widetilde A_\varepsilon
  :=
  (\mathcal U_\varepsilon\otimes I_d)
  A
  (\mathcal V_\varepsilon\otimes I_d)^*
  =
  \sum_{i,j}
  |e_i^\varepsilon\rangle
  \langle f_j^\varepsilon|
  \otimes A_{ij}
\]
has the same nonzero singular values as \(A\), hence $\norm{\widetilde{A}_\varepsilon}_{\mathfrak S_r}=\norm{A}_{\mathfrak S_r}$. Here, $|\cdot\rangle$ and $\langle\cdot|$ denote the Dirac brackets:
\[
  |e_i^\varepsilon\rangle
  \langle f_j^\varepsilon|:\,L^2(\mathbb R)\to L^2(\mathbb R),\quad h\mapsto\langle f_j^\varepsilon,h\rangle_{L^2(\mathbb R)}e_i^\varepsilon.
\]
Define 
\[
  B_\varepsilon
  :=
  (\mathcal U_\varepsilon\otimes I_d)^*
  S_{\chi_\alpha(t-s)}(\widetilde A_\varepsilon)
  (\mathcal V_\varepsilon\otimes I_d),  
\]
where $S_{\chi_\alpha(t-s)}$ is understood as its amplification, acting entrywise on $\widetilde{A}_\varepsilon=(\widetilde{A}_\varepsilon^{ab})_{a,b=1}^d$: 
\[
  S_{\chi_\alpha(t-s)}(\widetilde{A}_\varepsilon):=(S_{\chi_\alpha(t-s)}(\widetilde{A}_\varepsilon^{ab}))_{a,b=1}^d=(S_{\chi_\alpha(t-s)}\otimes id_{\mathbb C^{d\times d}})(\widetilde{A}_\varepsilon).
\]
Applying the amplified estimate \eqref{eq:log-spectral-cb-input} to
\(\widetilde A_\varepsilon\) and compressing yields
\[
  \norm{
    B_\varepsilon
  }_{\mathfrak S_r}
  \leq 
  \norm{
    S_{\chi_\alpha(t-s)}(\widetilde{A}_\varepsilon)
  }_{\mathfrak S_r}
  \leq 
  C_\alpha\frac{r^2}{r-1}
  \norm{
    \widetilde{A}_\varepsilon
  }_{\mathfrak S_r}
  =
  C_\alpha\frac{r^2}{r-1}
  \norm{
  (A_{ij})_{i\leq m,j\leq n}
  }_{\mathfrak S_r},
\]
where the first inequality follows from contractivity of compression and
the last equality from preservation of nonzero singular values under the isometric embeddings.
The \((i,j)\)-block of
\(B_\varepsilon\) is
\[
  (B_\varepsilon)_{ij}
  =
  c_{ij}^\varepsilon A_{ij},
  \qquad
  c_{ij}^\varepsilon
  :=
  \frac{1}{|U_i^\varepsilon||V_j^\varepsilon|}
  \int_{U_i^\varepsilon}
  \int_{V_j^\varepsilon}
  \chi_\alpha(u-v)\,dv\,du.
\]
As \(\varepsilon\to0\), we have 
\(c_{ij}^\varepsilon\to\chi_\alpha(t_i-s_j)\) by continuity of \(\chi_\alpha\).
Since the number and dimensions of the blocks are fixed, this convergence
also holds in the Schatten norm of the block matrix.
Taking the limit gives
\begin{equation}
  \left\|
    \big(
      \chi_\alpha(t_i-s_j)A_{ij}
    \big)_{i\le m,j\le n}
  \right\|_{\mathfrak S_r}
  \leq
  C_\alpha r
  \left\|
    (A_{ij})_{i\le m,j\le n}
  \right\|_{\mathfrak S_r}.
  \label{eq:log-spectral-finite-points}
\end{equation}

\emph{Step 2: finite spectral partitions.}
Let \(E\) and \(F\) be spectral measures supported on compact intervals
\(I\) and \(J\), respectively.
For finite Borel partitions
\[
  I=\bigsqcup_{i=1}^m\Delta_i,
  \qquad
  J=\bigsqcup_{j=1}^n\Gamma_j,
  \qquad
  E_i:=E(\Delta_i),
  \quad
  F_j:=F(\Gamma_j),
\]
choose representatives \(t_i\in\Delta_i\) and \(s_j\in\Gamma_j\).
For a finite-rank \(X\), set
\[
  K_{E,i}
  :=
  \operatorname{span}
  \bigcup_{j=1}^n
  \operatorname{Ran}(E_iXF_j),
  \qquad
  K_{F,j}
  :=
  \operatorname{span}
  \bigcup_{i=1}^m
  \operatorname{Ran}(F_jX^*E_i).
\]
Then 
\[
  K_E:=\bigoplus_{i=1}^m K_{E,i},
  \qquad
  K_F:=\bigoplus_{j=1}^n K_{F,j}
\]
are finite-dimensional.
Since \(X=P_{K_E}XP_{K_F}\), we may regard $X$ as an operator between these spaces.
Choose \(d\) no smaller than the dimensions of these summands, and select
isometries
\[
  \iota_i:K_{E,i}\to\mathbb C^d,
  \qquad
  \jmath_j:K_{F,j}\to\mathbb C^d.
\]
Define
\[
  J_E:\,K_E\to\mathbb C^m\otimes\mathbb C^d,
  \qquad
  \xi=(\xi_i)_i
  \mapsto
  \sum_{i=1}^m\delta_i\otimes\iota_i\xi_i,
\]
and
\[
  J_F:\,K_F\to\mathbb C^n\otimes\mathbb C^d,
  \qquad
  \eta=(\eta_j)_j
  \mapsto
  \sum_{j=1}^n\delta_j\otimes\jmath_j\eta_j.
\]
Then, the operator
\[
  B
  :=
  J_E(X|_{K_F})J_F^*
  :\,
  \mathbb C^n\otimes\mathbb C^d
  \to
  \mathbb C^m\otimes\mathbb C^d
\]
has blocks
\[
  B_{ij}
  =
  \iota_i(E_iXF_j)\jmath_j^*,
  \qquad
  \|B\|_{\mathfrak S_r}
  =
  \|X\|_{\mathfrak S_r}.
\]
Set
\[
  Y
  :=
    \sum_{i=1}^m\sum_{j=1}^n
    \chi_\alpha(t_i-s_j)E_iXF_j
  .
\]
As for \(X\), we have \(Y=P_{K_E}YP_{K_F}\).
A direct blockwise computation gives
\[
  J_E(Y|_{K_F})J_F^*
  =
  \big(
    \chi_\alpha(t_i-s_j)B_{ij}
  \big)_{i,j}.
\]
Applying \eqref{eq:log-spectral-finite-points} to $B$ gives
\begin{equation}
  \left\|
    \sum_{i=1}^m\sum_{j=1}^n
    \chi_\alpha(t_i-s_j)E_iXF_j
  \right\|_{\mathfrak S_r}
  =
  \norm{
    (\chi_\alpha(t_i-s_j)B_{ij})_{i,j}
  }_{\mathfrak S_r}
  \leq 
  C_\alpha r
  \norm{
    (B_{ij})_{i,j}
  }_{\mathfrak S_r}
  =
  C_\alpha r\|X\|_{\mathfrak S_r}.
  \label{eq:log-spectral-finite-partitions}
\end{equation}
This estimate holds for every finite-rank $X$.
Since finite-rank operators are dense in $\mathfrak S_r$, it extends to all
\(X\in\mathfrak S_r\).

\emph{Step 3: general spectral measures.}
Choose finite partitions
\((\Delta_i^{(N)})_i\) of \(I\) and
\((\Gamma_j^{(N)})_j\) of \(J\) with 
\[
  \delta_N:=\max\left\{
    \max_i\operatorname{diam}(\Delta_i^{(N)}),\,
    \max_j\operatorname{diam}(\Gamma_j^{(N)})
  \right\}\longrightarrow0
\] 
as $N\to\infty$, together with
representatives \(t_i^{(N)}\in\Delta_i^{(N)}\) and
\(s_j^{(N)}\in\Gamma_j^{(N)}\).
For
\[
  M_N(t,s)
  :=
  \sum_{i,j}
  \chi_\alpha\big(t_i^{(N)}-s_j^{(N)}\big)
  \mathbf{1}_{\Delta_i^{(N)}}(t)
  \mathbf{1}_{\Gamma_j^{(N)}}(s),
\]
the estimate for finite sums \eqref{eq:log-spectral-finite-partitions} yields,
for every \(X\in\mathfrak S_r\),
\begin{equation}\label{eq:log-spectral-simple-doi}
  \left\|
    \operatorname{DOI}_{M_N}^{E,F}(X)
  \right\|_{\mathfrak S_r}
  \leq
  C_\alpha r\|X\|_{\mathfrak S_r}.
\end{equation}
The function \(\chi_\alpha\) is uniformly continuous on the compact interval \(I-J\).
Writing \(\omega\) for its modulus of continuity on this interval, we have
\[
  \|M_N-\chi_\alpha(\cdot-\cdot)\|_\infty
  :=\sup_{t\in I,s\in J}|M_N(t,s)-\chi_\alpha(t-s)|
  \le\omega(2\delta_N)\longrightarrow0.
\]
For \(X\in\mathfrak S_2\), the Hilbert--Schmidt DOI estimate gives
\[
\begin{aligned}
  \left\|
    \operatorname{DOI}_{M_N(t,s)-\chi_\alpha(t-s)}^{E,F}(X)
  \right\|_{\mathfrak S_r}
  &\leq
  \left\|
    \operatorname{DOI}_{M_N(t,s)-\chi_\alpha(t-s)}^{E,F}(X)
  \right\|_{\mathfrak S_2}\\
  &\leq
  \|M_N(t,s)-\chi_\alpha(t-s)\|_{\infty}
  \|X\|_{\mathfrak S_2}
  \longrightarrow0.
\end{aligned}
\]
Here, \(r\geq2\) is used in the first inequality.
Thus, taking the limit in \eqref{eq:log-spectral-simple-doi} gives 
\begin{equation}
  \left\|
    \operatorname{DOI}_{\chi_\alpha(t-s)}^{E,F}(X)
  \right\|_{\mathfrak S_r}
  \leq
  C_\alpha r\|X\|_{\mathfrak S_r},\qquad X\in\mathfrak S_2.
  \label{eq:log-spectral-continuous-doi}
\end{equation}
Since \(\mathfrak S_2\) is dense in \(\mathfrak S_r\) for \(2\le r<\infty\),
this DOI has a unique bounded extension to \(\mathfrak S_r\).
The extension agrees with the Hilbert--Schmidt DOI on \(\mathfrak S_2\)
and satisfies the same estimate for every \(X\in\mathfrak S_r\).

Finally, let \(D:=\{(t,s):t=s\}\) and
\(P_D:=\operatorname{DOI}_{\mathbf{1}_D}^{E,F}\).
By \Cref{lem:contraction_on_diag},
\[
  \|P_D(X)\|_{\mathfrak S_r}\le\|X\|_{\mathfrak S_r}.
\]
Since
\[
  \chi_{\alpha,0}(t,s)
  =
  \chi_\alpha(t-s)-\chi_\alpha(0)\mathbf{1}_D(t,s),
\]
\eqref{eq:log-spectral-continuous-doi} yields
\[
  \left\|
    \operatorname{DOI}_{\chi_{\alpha,0}}^{E,F}(X)
  \right\|_{\mathfrak S_r}
  \leq
  \left(
    C_\alpha r+|\chi_\alpha(0)|
  \right)
  \|X\|_{\mathfrak S_r}
  \leq
  C_\alpha r\|X\|_{\mathfrak S_r},
\]
which completes the proof.
\end{proof}

\begin{lemma}[Bound for the interaction multiplier]
\label{lem:lipschitz-schatten-multipliers}
Under \Cref{ass:lipschitz-filter}, let
\[
  0<\alpha\le\min\{1,2\rho_\psi\}.
\]
For any \(2\le r<\infty\) and any spectral measures $E,F$ supported on $[0,\kappa^2]$, we have
\begin{equation}
  \label{eq:lipschitz-interaction-schatten-bound}
  \left\|
    (x+\lambda)(y+\lambda)
    \frac{
      \varphi_\lambda(x)\psi_\lambda(y)
      -
      \psi_\lambda(x)\varphi_\lambda(y)
    }{
      (x-y)a_{\lambda,\alpha}(y)
    }
    \mathbf{1}_{\{x\ne y\}}(x,y)
  \right\|_{\mathfrak{M}_r(E,F)}
  \le
  C_{A,\alpha,\rho_\psi} r
\end{equation}
for some constant $C_{A,\alpha,\rho_\psi}>0$. In particular, the constant $C_{A,\alpha,\rho_\psi}$ is independent of $r$, $E$, and $F$.
\end{lemma}

\begin{proof}
Set \(K_\lambda=\kappa^2/\lambda\).
For \(0\le u,v\le K_\lambda\), put \(t=\log(1+u)\) and \(s=\log(1+v)\).
On \(J_\lambda=[0,\log(1+K_\lambda)]\), define
\[
  F_\lambda(t)=e^t\phi_\lambda^\sharp(e^t-1),
  \qquad
  P_\lambda(t)=e^{\alpha t/2}\psi_\lambda^\sharp(e^t-1).
\]
The function \(F_\lambda\) is bounded and Lipschitz uniformly in \(\lambda\)
by \Cref{ass:lipschitz-filter}.
Since
\[
  P_\lambda(t)=e^{-(\rho_\psi-\alpha/2)t}P_{\lambda,\rho_\psi}(t),
  \qquad \rho_\psi-\alpha/2\ge0,
\]
the same holds for \(P_\lambda\), with constants depending only on
\(A,\alpha,\rho_\psi\).
Extend both functions to \(\mathbb R\) by composition with the nearest-point
projection onto \(J_\lambda\).
This preserves their bounds, nonnegativity, and Lipschitz constants.

Let $z=t-s$ and
\[
  D_\lambda(s)=\frac{1}{1+P_\lambda(s)}.
\]
Since $0\leq\psi_\lambda^\sharp \leq 1$, we have $0<D_\lambda \le 1$.
With all divided differences and scalar kernels assigned value zero on the
diagonal, direct algebra gives
\begin{align}
  \label{eq:log-interaction-factorization}
  &(x+\lambda)(y+\lambda)
    \frac{
      \varphi_\lambda(x)\psi_\lambda(y)
      -
      \psi_\lambda(x)\varphi_\lambda(y)
    }{
      (x-y)a_{\lambda,\alpha}(y)
    }
    \mathbf{1}_{\{x\neq y\}}(x,y)\\
  =&(1+u)(1+v)
  \frac{
    \phi_\lambda^\sharp(u)\psi_\lambda^\sharp(v)-\psi_\lambda^\sharp(u)\phi_\lambda^\sharp(v)
  }{
    (u-v)\xk{\psi_\lambda^\sharp(v)+(1+v)^{-\alpha/2}}
  }
  \mathbf{1}_{\{u\neq v\}}(u,v)\\
  =&
  D_\lambda(s)
  \zk{
    -F_\lambda(t)\vartheta_\alpha(z)(P_\lambda)_0^{[1]}(t,s)
    +
    P_\lambda(t)\vartheta_\alpha(z)(F_\lambda)_0^{[1]}(t,s)
    +
    F_\lambda(t)P_\lambda(s)\chi_{\alpha,0}(t,s)
  }.
\end{align}
Since changing variables in the spectral integrals preserves DOI multiplier norms, it suffices to estimate the \(\mathfrak M_r\) norm of the kernel on the right-hand side. By \Cref{lem:two-measure-lipschitz-schatten}, \Cref{lem:universal_schur_multiplier} and
\Cref{lem:log-spectral-schatten-multipliers}, we have
\[
    \norm{
        D_\lambda(s)F_\lambda(t)\vartheta_\alpha(z)(P_\lambda)_0^{[1]}(t,s)
    }_{\mathfrak{M}_r}
    \leq
    \norm{D_\lambda}_{L^\infty} \cdot\norm{F_\lambda}_{L^\infty} \cdot\norm{\vartheta_\alpha}_{\mathfrak{M}_r} \cdot\norm{(P_\lambda)_0^{[1]}}_{\mathfrak{M}_r}\leq C_{A,\alpha,\rho_\psi} r,
\]
\[
    \norm{
        D_\lambda(s)P_\lambda(t)\vartheta_\alpha(z)(F_\lambda)_0^{[1]}(t,s)
    }_{\mathfrak{M}_r}
    \leq
    \norm{D_\lambda}_{L^\infty} \cdot\norm{P_\lambda}_{L^\infty} \cdot\norm{\vartheta_\alpha}_{\mathfrak{M}_r}
    \cdot\norm{(F_\lambda)_0^{[1]}}_{\mathfrak{M}_r}
    \leq C_{A,\alpha,\rho_\psi} r,
\]
\[
    \norm{
        D_\lambda(s)F_\lambda(t)P_\lambda(s)\chi_{\alpha,0}(t,s)
    }_{\mathfrak{M}_r}
    \leq
    \norm{D_\lambda}_{L^\infty} \cdot\norm{F_\lambda}_{L^\infty} \cdot\norm{P_\lambda}_{L^\infty}
    \cdot\norm{\chi_{\alpha,0}}_{\mathfrak{M}_r}
    \leq C_{A,\alpha,\rho_\psi} r.
\]
Combining the estimates above yields the conclusion.
\end{proof}

Recall that \(D=\{(x,y):x=y\}\), and let
\(P_D^{T_X,T}=\operatorname{DOI}_{\mathbf1_D}^{E_{T_X},E_T}\).
For each common atom \(\tau\) of \(E_{T_X}\) and \(E_T\), the spectral identities give
\[
  E_{T_X}(\{\tau\})(T_X-T)E_T(\{\tau\})
  =(\tau-\tau)E_{T_X}(\{\tau\})E_T(\{\tau\})=0.
\]
The atomic representation in the proof of \Cref{lem:contraction_on_diag}
therefore yields \(P_D^{T_X,T}(T_X-T)=0\).
Consequently, for every \(f\in\mathfrak M_2(E_{T_X},E_T)\),
\begin{equation}\label{eq:zero-diagonal-contribution}
  \operatorname{DOI}_{f\mathbf1_D}^{E_{T_X},E_T}(T_X-T)=0.
\end{equation}
Since \(K_\lambda=T_{X,\lambda}^{-1/2}(T_X-T)T_\lambda^{-1/2}\),
left and right multiplication by these bounded spectral functions preserves
the vanishing diagonal blocks.
Thus \(P_D^{T_X,T}(K_\lambda)=0\), and
\begin{equation}
  \label{eq:zero-diagonal-mixed-perturbation}
  \operatorname{DOI}_{f\mathbf1_D}^{E_{T_X},E_T}(K_\lambda)=0.
\end{equation}
In particular, changing a bounded DOI kernel only on \(D\) does not change
its action on \(T_X-T\) or \(K_\lambda\).

\begin{lemma}[Hilbert--Schmidt bound for the shrinkage difference]
\label{lem:lipschitz-shrinkage-hs-difference}
Under \Cref{ass:lipschitz-filter}, on
\(\{\delta_\lambda \le1/4\}\),
\[
  \norm{
    q_\lambda(T_X)-q_\lambda(T)
  }_{\mathfrak{S}_2}
  \le
  C\eta_\lambda.
\]
\end{lemma}

\begin{proof}
Set
\[
  Q_\lambda(t)
  =
  q_\lambda \xk{\lambda(e^t-1)}
  =
  (1-e^{-t})F_\lambda(t).
\]
By the transformed filter assumption (\Cref{ass:lipschitz-filter}), $Q_\lambda$ is uniformly bounded and
Lipschitz.
For $x=\lambda u$, $y=\lambda v$, $u=e^t-1$, $v=e^s-1$, and $t\neq s$, we compute that 
\[
  \sqrt{\lambda+x}\sqrt{\lambda+y}
  \frac{
    q_\lambda(x)-q_\lambda(y)
  }{
    x-y
  }
  =
  \sqrt{1+u}\sqrt{1+v}
  \frac{
    q_\lambda(\lambda u)-q_\lambda(\lambda v)
  }{
    u-v
  }
  =
  \frac{t-s}{2\sinh((t-s)/2)}
  (Q_\lambda)_0^{[1]}(t,s),
\]
hence 
\[
    \left|\sqrt{\lambda+x}\sqrt{\lambda+y}
  \frac{
    q_\lambda(x)-q_\lambda(y)
  }{
    x-y
  }
    \right|
    =
    \left|
        \frac{t-s}{2\sinh((t-s)/2)}
        (Q_\lambda)_0^{[1]}(t,s)
    \right|
    <C_0<\infty
\]
for some constant $C_0>0$ independent of $\lambda$.

Since \(T_X\) and \(T\) are trace class, \(T_X-T\in\mathfrak{S}_2\).
For each fixed \(\lambda>0\), the representation
\(q_\lambda(x)=Q_\lambda(\log(1+x/\lambda))\)
shows that \(q_\lambda\) is Lipschitz on \([0,\kappa^2]\).
Applying \Cref{thm:birman-solomyak-doi} gives
\[\begin{aligned}
    q_\lambda(T_X)-q_\lambda(T)
    &=
    \operatorname{DOI}_{(q_\lambda)^{[1]}_0}^{E_{T_X},E_T}(T_X-T)\\
    &=\operatorname{DOI}_{(q_\lambda)^{[1]}_0}^{E_{T_X},E_T}(T_{X,\lambda}^{1/2} K_\lambda T_\lambda^{1/2})\\
    &=\operatorname{DOI}_{\Theta_\lambda}^{E_{T_X},E_T}(K_\lambda),
\end{aligned}\]
where 
\[
    \Theta_\lambda(x,y)=\sqrt{x+\lambda}\sqrt{y+\lambda}\frac{q_\lambda(x)-q_\lambda(y)}{x-y}\mathbf{1}_{\{x\neq y\}}(x,y).
\]
By (\ref{eq:zero-diagonal-mixed-perturbation}), the kernel values on $D$ do not affect the DOI acting on $K_\lambda$.
We may therefore set the kernel to zero on $D$.

Finally, since the $\mathfrak{M}_2$ norm of the DOI kernel is precisely its essential supremum norm,
\[
  \norm{
    q_\lambda(T_X)-q_\lambda(T)
  }_{\mathfrak{S}_2}
  \leq
  \norm{
    \operatorname{DOI}_{\Theta_\lambda}^{E_{T_X},E_T}(K_\lambda)
  }_{\mathfrak{S}_2}
  \leq 
  \|\Theta_\lambda\|_{L^\infty}\cdot \norm{K_\lambda}_{\mathfrak{S}_2}
  \leq
  C\eta_\lambda
\]
where we use \Cref{lem:mixed-perturbation} in the last inequality. 
\end{proof}

\begin{lemma}[Interaction bound in the prediction norm]
\label{lem:lipschitz-embedded-interaction}
Under \Cref{ass:lipschitz-filter}, let
\[
  0<\alpha\le\min\{1,2\rho_\psi\},
  \qquad
  G_\lambda
  =
  \varphi_\lambda(T_X)\psi_\lambda(T)
  -
  \psi_\lambda(T_X)\varphi_\lambda(T).
\]
For \(g\in\overline{\Ran(L)}\), on
\(\{\delta_\lambda \le1/4\}\), and for every \(2\le r<\infty\),
\[
  \norm{
    SG_\lambda S^*g
  }_{L^2(\mu)}
  \le
  C_{A,\alpha,\rho_\psi}
  r\norm{K_\lambda}_{\mathfrak{S}_r}
  \norm{
    a_{\lambda,\alpha}(L)g
  }_{L^2}.
\]
\end{lemma}

\begin{proof}
Recall that $T_\lambda=T+\lambda I$, $T_{X,\lambda}=T_X+\lambda I$, 
and $K_\lambda
=
T_{X,\lambda}^{-1/2}
(T_X-T)
T_\lambda^{-1/2}$. Moreover, we denote by $E_{T_X}$ and $E_T$ the spectral measures of
$T_X$ and $T$, respectively.

We first rewrite $G_\lambda$ as a
DOI. Define
\[
Q_\lambda(x,y)
=
\frac{
\varphi_\lambda(x)\psi_\lambda(y)
-
\psi_\lambda(x)\varphi_\lambda(y)}
{x-y}
\mathbf{1}_{\{x\neq y\}} .
\]
Note that 
\[
Q_\lambda(x,y)
=
(x+\lambda)^{-1}
M_{\lambda,\alpha}(x,y)
(y+\lambda)^{-1}a_{\lambda,\alpha}(y),
\]
where 
\[
M_{\lambda,\alpha}(x,y)
=
(x+\lambda)(y+\lambda)
\frac{
\varphi_\lambda(x)\psi_\lambda(y)
-
\psi_\lambda(x)\varphi_\lambda(y)}
{(x-y)a_{\lambda,\alpha}(y)}
\mathbf{1}_{\{x\neq y\}}
\]
is the $\mathfrak{S}_2$-DOI multiplier defined in \Cref{lem:lipschitz-schatten-multipliers}. Since multiplication by bounded
functions of one variable preserves the DOI multiplier property, we obtain
that $Q_\lambda$ is also an $\mathfrak{S}_2$-DOI multiplier. Then, we compute that
\[\begin{aligned}
&\iint
Q_\lambda(x,y)
\,dE_{T_X}(x)
(T_X-T)
dE_T(y)\\
=&\iint\frac{\varphi_\lambda(x)\psi_\lambda(y)-\psi_\lambda(x)\varphi_\lambda(y)}{x-y}\cdot dE_{T_X}(x)(T_X-T)dE_T(y)\\
=&\iint\frac{\varphi_\lambda(x)\psi_\lambda(y)-\psi_\lambda(x)\varphi_\lambda(y)}{x-y}\cdot dE_{T_X}(x)(x-y)dE_T(y)\\
=&\iint(\varphi_\lambda(x)\psi_\lambda(y)-\psi_\lambda(x)\varphi_\lambda(y))dE_{T_X}(x)dE_T(y)\\
=&\varphi_\lambda(T_X)\psi_\lambda(T)-\psi_\lambda(T_X)\varphi_\lambda(T)\\
=&G_\lambda,
\end{aligned}\]
where the first equality follows because the kernel values on the diagonal do not affect this DOI, by (\ref{eq:zero-diagonal-contribution}).

It follows that
\[
\begin{aligned}
T_{X,\lambda}^{1/2} G_\lambda S^*g
&=\iint T_{X,\lambda}^{1/2} Q_\lambda(x,y)dE_{T_X}(x)(T_X-T)dE_T(y)S^*g\\
&=
\iint
M_{\lambda,\alpha}(x,y)dE_{T_X}(x)
K_\lambda
dE_T(y)w,
\end{aligned}
\]
where $w=a_{\lambda,\alpha}(T)T_\lambda^{-1/2} S^*g$. By \cref{eq:spectral-intertwining}, we have
\[
    w=a_{\lambda,\alpha}(T)T_\lambda^{-1/2} S^*g
    =S^*a_{\lambda,\alpha}(L)(L+\lambda I)^{-1/2}g,
\]
then
\[
\begin{aligned}
    \norm{w}_{\mathcal{H}}^2
    &=\left\langle
      a_{\lambda,\alpha}(L)(L+\lambda I)^{-1/2}g,
      La_{\lambda,\alpha}(L)(L+\lambda I)^{-1/2}g
    \right\rangle_{L^2}\\
    &=
    \norm{
        \sqrt{\frac{L}{L+\lambda}}a_{\lambda,\alpha}(L)g
    }_{L^2}^2
    \leq\norm{a_{\lambda,\alpha}(L)g}_{L^2}^2.
\end{aligned}
\]
By Lemma \ref{lem:lipschitz-schatten-multipliers}, we obtain
\begin{equation}\label{eq:embedded-interaction-step}
\begin{aligned}
    \|T_{X,\lambda}^{1/2} G_\lambda S^*g\|_{\mathcal{H}}
    &\leq
    \norm{
        \iint
        M_{\lambda,\alpha}(x,y)dE_{T_X}(x)K_\lambda dE_T(y)
    }_{\mathfrak{B}}
    \cdot\norm{w}_{\mathcal{H}}\\
    &\leq 
    \norm{
        \iint
        M_{\lambda,\alpha}(x,y)dE_{T_X}(x)K_\lambda dE_T(y)
    }_{\mathfrak{S}_r}
    \cdot\norm{w}_{\mathcal{H}}\\
    &\leq C_{A,\alpha,\rho_\psi} r
    \norm{K_\lambda}_{\mathfrak{S}_r}
    \cdot
    \norm{a_{\lambda,\alpha}(L)g}_{L^2}.
\end{aligned}
\end{equation}

As in \Cref{lem:regularized-loewner-comparison}, on the event $\{\delta_\lambda \le 1/4\}$,
\[
    -\delta_\lambda T_\lambda \preceq T_X-T\preceq\delta_\lambda T_\lambda,
\]
hence 
\[
T\preceq T_\lambda \preceq\frac{4}{3}T_{X,\lambda} \preceq 2T_{X,\lambda},
\]
Consequently, for every $h\in\mathcal{H}$,
\[
\|Sh\|_{L^2(\mu)}^2
=
\langle h,Th\rangle_{\mathcal{H}}
\le
2\langle h,T_{X,\lambda} h\rangle_{\mathcal{H}}.
\]
Taking
\[
h=G_\lambda S^*g
\]
gives
\[
\|SG_\lambda S^*g\|_{L^2(\mu)}
\le
\sqrt2
\|T_{X,\lambda}^{1/2} G_\lambda S^*g\|_{\mathcal{H}}.
\]

Combining this estimate with \cref{eq:embedded-interaction-step} gives
\[
\|SG_\lambda S^*g\|_{L^2(\mu)}
\le
C_{A,\alpha,\rho_\psi} r
\|K_\lambda\|_{\mathfrak{S}_r}
\|a_{\lambda,\alpha}(L)g\|_{L^2(\mu)},
\]
which proves the lemma.
\end{proof}

\begin{lemma}[Logarithmic bound from Schatten interpolation]
\label{lem:logarithmic-schatten-interpolation}
For every \(H\in\mathfrak{S}_2\),
\begin{equation}
  \label{eq:optimized-schatten-endpoint}
  \inf_{2\le r<\infty}
  r\norm{H}_{\mathfrak{S}_r}
  \le
  C\norm{H}_{\mathfrak{B}}
  \log\xk{
    e+
    \frac{\norm{H}_{\mathfrak{S}_2}}{\norm{H}_{\mathfrak{B}}}
  },
\end{equation}
with the right-hand side interpreted as zero when \(H=0\).
Under \Cref{ass:basic-model,ass:lipschitz-filter} and
\cref{eq:leverage-variance-compatibility},
\[
  \inf_{2\le r<\infty}
  r\norm{K_\lambda}_{\mathfrak{S}_r}
  =
  O_{\mathbb{P}} \xk{
    \gamma_\lambda
  }.
\]
\end{lemma}

\begin{proof}
Let $a=\norm{H}_{\mathfrak{B}}$ and $b=\norm{H}_{\mathfrak{S}_2}$.
If \(a=0\), then \(H=0\) and the conclusion is immediate.
For \(a>0\), we claim that
\[
  \norm{H}_{\mathfrak{S}_r}
  \le
  a^{1-2/r} b^{2/r}.
\]
In fact, if $\{s_j\}_{j=1}^\infty$ is the sequence of the singular values of $H$, then $a=\max_j s_j$, and
\[
    \norm{H}_{\mathfrak{S}_r}^r=\sum_{j=1}^\infty s_j^r=\sum_{j=1}^\infty s_j^{r-2} \cdot s_j^2 \leq a^{r-2} \sum_{j=1}^\infty s_j^2=a^{r-2} b^2.
\]
Choose
\[
  r
  =
  2\log\xk{e+\frac{b}{a}
  },
\]
then $(b/a)^{2/r}\le e$, and hence
\[
    r\norm{H}_{\mathfrak{S}_r}
    \leq
    ra\xk{\frac{b}{a}}^{2/r}
    \leq
    2e\,a\log\xk{e+\frac{b}{a}},
\]
which proves
\cref{eq:optimized-schatten-endpoint}.

By \Cref{cor:variance-compatible-concentration},
\[
  \norm{K_\lambda}_{\mathfrak{B}}
  =
  O_{\mathbb{P}} \xk{
    \sqrt{
      \frac{\mathcal{L}_\lambda\ell_\lambda}{n}
    }
  }.
\]
By \Cref{prop:covariance-concentration,lem:mixed-perturbation},
\[
  \norm{K_\lambda}_{\mathfrak{S}_2}
  =
  O_{\mathbb{P}} \xk{
    \sqrt{\frac{\mathcal{L}_\lambda\mathcal{N}_1(\lambda)}{n}}
  }.
\]
The function $x\mapsto x\log(e+c/x)$ is increasing for $x,c>0$, hence we can
apply \cref{eq:optimized-schatten-endpoint} to $K_\lambda$ and obtain
\[
  \inf_{2\le r<\infty}
  r\norm{K_\lambda}_{\mathfrak{S}_r}
  =
  O_{\mathbb{P}} \zk{
    \sqrt{
      \frac{\mathcal{L}_\lambda\ell_\lambda}{n}
    }
    \log\xk{
      e+
      \sqrt{
        \frac{\mathcal{N}_1(\lambda)}{\ell_\lambda}
      }
    }
  }.
\]
\end{proof}

\begin{proposition}[Variance comparison for Lipschitz filters]
\label{prop:lipschitz-variance-comparison}
Under \Cref{ass:basic-model,ass:lipschitz-filter} and
\cref{eq:leverage-variance-compatibility},
\begin{equation}\label{eq:lipschitz-variance-estimate}
  \mathsf{V}_X(\varphi_\lambda)
  =
  \xk{1+o_{\mathbb{P}}(1)}
  \mathcal{V}_n(q_\lambda).
\end{equation}
Consequently,
\begin{equation}\label{eq:lipschitz-variance-additive-estimate}
  \mathsf{V}_X(\varphi_\lambda)
  =
  \mathcal{V}_n(q_\lambda)
  +
  o_{\mathbb{P}} \xk{
    \mathcal{E}_n^{\mathsf{seq}}(q_\lambda;f^*)
  }.
\end{equation} 
\end{proposition}

\begin{proof}
Write
\[
  N_\lambda=\mathcal{N}_1(\lambda),
  \qquad
  Q_\lambda=\mathcal{N}_q(\lambda),
  \qquad
  v_\lambda=\frac{\mathcal{L}_\lambda \ell_\lambda}{n}.
\]
Combining the covariance bounds in
\Cref{prop:covariance-concentration,cor:variance-compatible-concentration}
with
\Cref{lem:lipschitz-shrinkage-hs-difference,lem:empirical-effective-dimension}
gives
\[
  d_{\lambda,X}
  =
  \norm{
    q_\lambda(T_X)-q_\lambda(T)
  }_{\mathfrak{S}_2}
  =
  O_{\mathbb{P}} \xk{
    \sqrt{\frac{\mathcal{L}_\lambda N_\lambda}{n}}
  },
  \qquad
  \delta_\lambda=O_{\mathbb{P}}(\sqrt{v_\lambda}),
  \qquad
  \widehat{\mathcal{N}}_{1,X}(\lambda)
  =
  O_{\mathbb{P}}(N_\lambda),
\]
The variance condition
\cref{eq:leverage-variance-compatibility} gives
\[
  \delta_\lambda
  \frac{
    \widehat{\mathcal{N}}_{1,X}(\lambda)
  }{
    Q_\lambda
  }
  =
  o_{\mathbb{P}}(1).
\]
It also gives
\[
  \frac{d_{\lambda,X}^2}{Q_\lambda}
  =
  O_{\mathbb{P}} \xk{
    \frac{\mathcal{L}_\lambda N_\lambda}{nQ_\lambda}
  }
  =
  O_{\mathbb{P}} \zk{
    \xk{
      \frac{N_\lambda}{Q_\lambda}\sqrt{v_\lambda}
    }
    \frac{\sqrt{v_\lambda}}{\ell_\lambda}
  }
  =
  o_{\mathbb{P}}(1).
\]
Since \(\delta_\lambda=o_{\mathbb{P}}(1)\) by
\Cref{cor:variance-compatible-concentration}, these bounds imply
that the right-hand side of \cref{eq:common-variance-transfer}, divided by
\(Q_\lambda\), converges to zero in probability.
Applying \Cref{lem:common-variance-transfer} therefore proves
\cref{eq:lipschitz-variance-estimate}.

The additive comparison \cref{eq:lipschitz-variance-additive-estimate} follows because the sequence risk is at least its
variance term.
\end{proof}

\begin{proposition}[Bias comparison for Lipschitz filters]
\label{prop:lipschitz-bias-comparison}
Under \Cref{ass:basic-model,ass:lipschitz-filter} and
\Cref{ass:conditional-equivalence-scale},
\[
  \mathsf{B}_X(q_\lambda)
  =
  \mathcal{B}_n(q_\lambda;f^*)
  +
  o_{\mathbb{P}} \xk{
    \mathcal{E}_n^{\mathsf{seq}}(q_\lambda;f^*)
  }.
\]
\end{proposition}

\begin{proof}
Set
\[
  \beta_\lambda
  =
  C_{A,\alpha_\psi,\rho_\psi}
  \inf_{2\le r<\infty}
  r\norm{K_\lambda}_{\mathfrak{S}_r}.
\]
\Cref{lem:lipschitz-embedded-interaction} verifies
\cref{eq:interaction-certificate} with
\(\alpha=\alpha_\psi\), simultaneously for all
\(g\in\overline{\Ran(L)}\), and
\Cref{lem:logarithmic-schatten-interpolation} gives
\[
  \beta_\lambda
  =
  O_{\mathbb{P}} \xk{
    \gamma_\lambda
  }.
\]
Since \(\alpha_\psi/2\le\rho_\psi\), the qualification bound gives
\[
  \norm{
    a_{\lambda,\alpha_\psi}(L)f^*
  }_{L^2(\mu)}^2
  \le
  (A+1)^2\mathcal{S}_{f^*}(\lambda).
\]
The spectral balance condition for the fixed target therefore yields
\[
  \beta_\lambda^2
  \norm{
    a_{\lambda,\alpha_\psi}(L)f^*
  }_{L^2(\mu)}^2
  =
  o_{\mathbb{P}} \xk{
    \frac{\sigma^2}{n}\mathcal{N}_q(\lambda)
  }.
\]
Moreover,
\Cref{cor:variance-compatible-concentration} gives
\(\delta_\lambda=o_{\mathbb{P}}(1)\).
\Cref{prop:common-bias-transfer} proves the proposition.
\end{proof}

\begin{proof}[Proof of \Cref{thm:conditional-risk-equivalence}]
Combining
\Cref{prop:lipschitz-bias-comparison,prop:lipschitz-variance-comparison}
and using the exact risk decomposition \cref{eq:np-risk} gives
\[
  \mathcal{E}_n^{\mathsf{np}}(q_\lambda;f^*\mid X)
  =
  \mathcal{E}_n^{\mathsf{seq}}(q_\lambda;f^*)
  +
  o_{\mathbb{P}} \xk{
    \mathcal{E}_n^{\mathsf{seq}}(q_\lambda;f^*)
  },
\]
which is equivalent to
\cref{eq:conditional-risk-equivalence}.
All constants in the preceding bounds depend only on the constants in the
assumptions.
The same argument therefore gives the asserted uniform
\(o_{\mathbb{P}}(1)\) remainder whenever those assumptions hold uniformly.
\end{proof}
\clearpage{}
\clearpage{}\section{Risk Equivalence in Expectation}
\label{sec:risk-equivalence-expectation}

This section proves the risk comparison in
\Cref{thm:risk-equivalence-expectation}.
We integrate the risk comparison bounds over the event where the empirical
operator is well conditioned and use a global estimator bound on its complement.

We use \(v_\lambda\) from \Cref{sec:operator-concentration} and
\(R_\lambda,d_{\lambda,X}\) and the regularized inverse bound constant \(E\)
from \Cref{sec:common-risk-transfer}.
Let
\[
  \mathcal{G}_\lambda
  =
  \{\delta_\lambda \le1/4\}.
\]
For a fixed target, also set
\[
  D_\lambda
  =
  \mathcal{E}_n^{\mathsf{seq}}(q_\lambda;f^*)
  =
  R_\lambda^2+\frac{\sigma^2}{n}\mathcal{N}_q(\lambda).
\]
The positive noise variance and the first item of
\Cref{ass:conditional-equivalence-scale} ensure that
\(D_\lambda>0\) for all sufficiently large \(n\).

\begin{lemma}[Second-moment and tail bounds]
\label{lem:risk-expectation-moment-tail}
Under the assumptions of
\Cref{thm:conditional-risk-equivalence}, there are constants
\(c,C\in(0,\infty)\), independent of \(n\) and \(\lambda\), such that
\begin{align}
  \label{eq:bad-design-tail}
  \mathbb{P}_X(\mathcal{G}_\lambda^{\mathsf{c}})
  &\le
  C\exp\xk{-c\frac{n}{\mathcal{L}_\lambda}},
  \\
  \label{eq:operator-second-moments}
  \E_X \delta_\lambda^2
  &\le Cv_\lambda,
  &
  \E_X \eta_\lambda^2
  &\le C\frac{\mathcal{L}_\lambda \mathcal{N}_1(\lambda)}{n}.
\end{align}
Use \(A_x\) and \(A\) from the proof of
\Cref{prop:covariance-concentration}, and set \(B_i=A_{x_i}-A\).
For \(1\le i\le n\), define the leave-one-out event
\[
  \mathcal{B}_{\lambda,-i}
  =
  \dk{
    \left\|
      \frac{1}{n}\sum_{j\ne i} B_j
    \right\|_{\mathfrak{B}}
    >
    \frac{1}{8}
  }.
\]
For all sufficiently large \(n\),
\begin{equation}
  \label{eq:leave-one-out-bad-design-tail}
  \mathcal{G}_\lambda^{\mathsf{c}}
  \subseteq
  \mathcal{B}_{\lambda,-i},
  \qquad
  \max_{1\le i\le n}
  \mathbb{P}_X(\mathcal{B}_{\lambda,-i})
  \le
  C\exp\xk{-c\frac{n}{\mathcal{L}_\lambda}}.
\end{equation}
Moreover,
\begin{equation}
  \label{eq:global-delta-bound}
  \delta_\lambda \le1+\mathcal{L}_\lambda
  \quad\text{almost surely},
\end{equation}
and, on \(\mathcal{G}_\lambda\),
\begin{equation}
  \label{eq:good-empirical-dimension-bound}
  \widehat{\mathcal{N}}_{1,X}(\lambda)
  \le
  \frac{2}{n}\sum_{i=1}^n
  \norm{T_\lambda^{-1/2}k_{x_i}}_{\mathcal H}^2.
\end{equation}

For the fixed-target residual, define
\[
  U_\lambda
  =
  \norm{T_\lambda^{-1/2} Z_\lambda}_{\mathcal{H}}.
\]
Then
\begin{equation}
  \label{eq:rhs-second-moment}
  \E_X U_\lambda^2
  \le
  \frac{\mathcal{L}_\lambda}{n}R_\lambda^2.
\end{equation}
\end{lemma}

\begin{proof}
The matrix Bernstein argument in
\Cref{prop:covariance-concentration} yields, for every \(t\ge0\),
\[
  \mathbb{P}_X \dk{
    \delta_\lambda
    >
    C\zk{
      \sqrt{
        \frac{
          \mathcal{L}_\lambda\varrho_\lambda
          (\ell_\lambda^{\mathrm{int}}+t)
        }{n}
      }
      +
      \frac{
        \mathcal{L}_\lambda(\ell_\lambda^{\mathrm{int}}+t)
      }{n}
    }
  }
  \le2e^{-t}.
\]
Variance compatibility and the regularized inverse bound imply
\[
  \varrho_\lambda\ell_\lambda^{\mathrm{int}}
  \lesssim
  \ell_\lambda,
  \qquad
  \frac{\mathcal{L}_\lambda\ell_\lambda^{\mathrm{int}}}{n}
  =o(\sqrt{v_\lambda}).
\]
Integrating this tail bound yields
\(\E_X\delta_\lambda^2\le Cv_\lambda\).
Taking \(t=c_0 n/\mathcal{L}_\lambda\), with a sufficiently small fixed \(c_0>0\),
and using the two preceding comparisons gives
\cref{eq:bad-design-tail}.
The Hilbert--Schmidt second-moment bound follows from
\Cref{prop:covariance-concentration}.

By definition, \(\norm{A_x}\le\mathcal{L}_\lambda\) almost surely, while
\(0\preceq A\preceq I\).
Consequently,
\[
  \delta_\lambda
  \le
  \frac{1}{n}\sum_{i=1}^n \norm{A_{x_i}}_{\mathfrak{B}}
  +
  \norm{A}_{\mathfrak{B}}
  \le
  \mathcal{L}_\lambda+1,
\]
which proves \cref{eq:global-delta-bound}.
Because \(\mathcal{L}_\lambda/n\to0\), for all sufficiently large \(n\),
\[
  \frac{\norm{B_i}_{\mathfrak{B}}}{n}
  \le
  \frac{\mathcal{L}_\lambda+1}{n}
  \le\frac{1}{8}
  \qquad\text{almost surely}.
\]
The reverse triangle inequality therefore gives the event inclusion in
\cref{eq:leave-one-out-bad-design-tail}.
For the leave-one-out sum, the variance majorant, intrinsic dimension, and
summand bound have orders \(n\mathcal{L}_\lambda\varrho_\lambda\),
\(\mathcal{N}_1(\lambda)/\varrho_\lambda\), and \(\mathcal{L}_\lambda\), respectively.
Since \(\mathcal{L}_\lambda\ell_\lambda^{\mathrm{int}}/n=o(1)\), the prefactor arising
from the intrinsic dimension is absorbed into the exponential tail.
Applying the matrix Bernstein inequality at threshold \(n/8\) therefore gives
the probability bound in \cref{eq:leave-one-out-bad-design-tail}; replacing
\(n\) by \(n-1\) changes only the constants.

On \(\mathcal{G}_\lambda\), the Loewner argument in
\Cref{lem:empirical-effective-dimension} gives
\[
  \widehat{\mathcal{N}}_{1,X}(\lambda)
  \le
  \frac{2}{n}\sum_{i=1}^n
  \norm{T_\lambda^{-1/2} k_{x_i}}_{\mathcal{H}}^2,
\]
which is \cref{eq:good-empirical-dimension-bound}.

The second-moment bound \cref{eq:rhs-second-moment} is proved in
\Cref{lem:rhs-fluctuation}.
\end{proof}

\begin{lemma}[Expected risk comparison on the good design event]
\label{lem:risk-expectation-good-event}
Under the assumptions of \Cref{thm:conditional-risk-equivalence},
\begin{equation}
  \label{eq:good-event-l1-transfer}
  \E_X \zk{
    \mathbf{1}_{\mathcal{G}_\lambda}
    \frac{
      \left|
        \mathcal{E}_n^{\mathsf{np}}(q_\lambda;f^*\mid X)
        -D_\lambda
      \right|
    }{D_\lambda}
  }
  \longrightarrow0.
\end{equation}
\end{lemma}

\begin{proof}
We first treat the variance.
On \(\mathcal{G}_\lambda\), the Hilbert--Schmidt bounds for the shrinkage difference give
\[
  d_{\lambda,X}
  \le C\eta_\lambda.
\]
Combining
\cref{eq:common-variance-transfer,eq:good-empirical-dimension-bound} and using
\(\delta_\lambda \le1/4\), we obtain
\begin{align}
  \label{eq:good-variance-envelope}
  \mathbf{1}_{\mathcal{G}_\lambda}
  \frac{
    \left|
      \mathsf{V}_X(\varphi_\lambda)
      -
      \frac{\sigma^2}{n}\mathcal{N}_q(\lambda)
    \right|
  }{
    \frac{\sigma^2}{n}\mathcal{N}_q(\lambda)
  }
  \le
  C\mathbf{1}_{\mathcal{G}_\lambda}
  \xk{
    \delta_\lambda
    +
    \frac{\eta_\lambda}{\sqrt{\mathcal{N}_q(\lambda)}}
    +
    \frac{\eta_\lambda^2}{\mathcal{N}_q(\lambda)}
    +
    \frac{\widehat{\mathcal N}_{1,X}(\lambda)}{\mathcal{N}_q(\lambda)}\delta_\lambda
  }.
\end{align}
The variance-compatibility condition implies
\[
  \frac{\mathcal{N}_1(\lambda)}{\mathcal{N}_q(\lambda)}\sqrt{v_\lambda}\to0,
  \qquad
  \frac{\mathcal{L}_\lambda \mathcal{N}_1(\lambda)}{n\mathcal{N}_q(\lambda)}
  =
  \xk{
    \frac{\mathcal{N}_1(\lambda)}{\mathcal{N}_q(\lambda)}\sqrt{v_\lambda}
  }
  \frac{\sqrt{v_\lambda}}{\ell_\lambda}
  \to0.
\]
Moreover, \cref{eq:good-empirical-dimension-bound}, independence, and
\(\norm{T_\lambda^{-1/2}k_X}_{\mathcal H}^2\le\mathcal L_\lambda\)
give
\[
  \E_X\zk{
    \mathbf{1}_{\mathcal G_\lambda}
    \widehat{\mathcal N}_{1,X}(\lambda)^2
  }
  \lesssim
  \mathcal{N}_1(\lambda)^2+\frac{\mathcal L_\lambda \mathcal{N}_1(\lambda)}{n}.
\]
The regularized inverse bound gives
\(\mathcal{N}_q(\lambda)\le E^2\varrho_\lambda \mathcal{N}_1(\lambda)\), where
\(\varrho_\lambda\le\min\{1,\mathcal{N}_1(\lambda)\}\).
Consequently, variance compatibility implies
\[
  \frac{\mathcal L_\lambda}{n\mathcal{N}_1(\lambda)}
  =\frac{v_\lambda}{\ell_\lambda \mathcal{N}_1(\lambda)}
  \le
  E^4
  \xk{\frac{\mathcal{N}_1(\lambda)}{\mathcal{N}_q(\lambda)}\sqrt{v_\lambda}}^2
  \frac{\varrho_\lambda^2}{\ell_\lambda \mathcal{N}_1(\lambda)}
  \le
  E^4
  \xk{\frac{\mathcal{N}_1(\lambda)}{\mathcal{N}_q(\lambda)}\sqrt{v_\lambda}}^2
  =o(1).
\]
Thus Cauchy--Schwarz gives
\[
  \E_X\zk{
    \mathbf{1}_{\mathcal G_\lambda}
    \delta_\lambda
    \frac{\widehat{\mathcal N}_{1,X}(\lambda)}{\mathcal{N}_q(\lambda)}
  }
  \lesssim
  \frac{\mathcal{N}_1(\lambda)}{\mathcal{N}_q(\lambda)}\sqrt{v_\lambda}
  \longrightarrow0.
\]
Thus \Cref{lem:risk-expectation-moment-tail} and Cauchy--Schwarz show that the
expectation of the right-hand side of
\cref{eq:good-variance-envelope} tends to zero.

For the bias, let \(\alpha=\alpha_\psi\).
Since \(\alpha/2\le\rho_\psi\), the qualification bound and
\cref{eq:source-bias-weight,eq:fixed-target-spectral-tail} give
\[
  \norm{a_{\lambda,\alpha}(L)f^*}_{L^2(\mu)}
  \le
  (A+1)\sqrt{\mathcal{S}_{f^*}(\lambda)}.
\]
On \(\mathcal{G}_\lambda\), the proof of
\Cref{prop:common-bias-transfer} applies with
\[
  u_\lambda=C\frac{U_\lambda}{R_\lambda},
  \qquad
  c_\lambda
  =
  C\beta_\lambda \sqrt{\mathcal{S}_{f^*}(\lambda)},
\]
where \(u_\lambda=0\) if \(R_\lambda=0\).
By \cref{eq:rhs-second-moment},
\begin{equation}
  \label{eq:good-u-second-moment}
  \E_X u_\lambda^2
  \le C\frac{\mathcal{L}_\lambda}{n}
  \longrightarrow0.
\end{equation}

Put
\[
  m_\lambda
  =
  2\log\xk{
    e+\sqrt{\frac{\mathcal{N}_1(\lambda)}{\ell_\lambda}}
  }>2.
\]
The Schatten interpolation inequality and
\Cref{lem:mixed-perturbation} give, on \(\mathcal{G}_\lambda\),
\[
  \beta_\lambda
  \le
  Cm_\lambda
  \delta_\lambda^{1-2/m_\lambda}
  \eta_\lambda^{2/m_\lambda}.
\]
Hölder's inequality and \cref{eq:operator-second-moments} therefore yield
\begin{align}
  \label{eq:lipschitz-beta-second-moment}
  \E_X \zk{
    \mathbf{1}_{\mathcal{G}_\lambda} \beta_\lambda^2
  }
  &\le
  Cm_\lambda^2
  \xk{\E_X \delta_\lambda^2 }^{1-2/m_\lambda}
  \xk{\E_X \eta_\lambda^2 }^{2/m_\lambda}
  \\
  &\le
  C\gamma_\lambda^2.
\end{align}
Indeed,
\((\mathcal{N}_1(\lambda)/\ell_\lambda)^{2/m_\lambda}\) is uniformly bounded.

The spectral balance condition for the fixed target gives
\begin{equation}
  \label{eq:expectation-fixed-target-balance}
  \frac{n\mathcal{S}_{f^*}(\lambda)}{\mathcal{N}_q(\lambda)}
  \E_X \zk{
    \mathbf{1}_{\mathcal{G}_\lambda} \beta_\lambda^2
  }
  \longrightarrow0.
\end{equation}
Apply \Cref{lem:scalar-quadratic-risk-transfer} as in the proof of
\Cref{prop:common-bias-transfer}.
Equations
\cref{eq:good-u-second-moment,eq:expectation-fixed-target-balance}, together with
Cauchy--Schwarz for the linear and cross terms, show that
\[
  \E_X \zk{
    \mathbf{1}_{\mathcal{G}_\lambda}
    \frac{
      \left|\mathsf{B}_X(q_\lambda)-R_\lambda^2 \right|
    }{D_\lambda}
  }
  \longrightarrow0.
\]
Adding the variance comparison proves \cref{eq:good-event-l1-transfer}.
\end{proof}

\begin{lemma}[Negligible risk contribution from bad designs]
\label{lem:risk-expectation-bad-event}
Suppose the hypotheses of \Cref{thm:risk-equivalence-expectation} hold.
Then
\begin{equation}
  \label{eq:fixed-bad-event-negligible}
  \E_X \zk{
    \mathbf{1}_{\mathcal{G}_\lambda^{\mathsf{c}}}
    \frac{
      \mathcal{E}_n^{\mathsf{np}}(q_\lambda;f^*\mid X)
    }{D_\lambda}
  }
  \longrightarrow0.
\end{equation}
\end{lemma}

\begin{proof}
We first derive a global envelope for a fixed target.
Use the population solution \(h_\lambda\) and residual \(r_\lambda\)
from the proof of \Cref{prop:common-bias-transfer}, with the measurable
representative \(r_\lambda(x)=f^*(x)-h_\lambda(x)\), and set
\[
  W_\lambda
  =
  \norm{
    \frac{1}{n}\sum_{i=1}^n
    T_\lambda^{-1/2} k_{x_i} r_\lambda(x_i)
  }_{\mathcal{H}}.
\]
The identity
\[
  \widetilde{g}_X
  =
  T_X h_\lambda
  +
  \frac{1}{n}\sum_{i=1}^n k_{x_i} r_\lambda(x_i)
\]
and \(q_\lambda(T_X)=T_X \varphi_\lambda(T_X)\) imply
\[
  \mathsf{B}_X(q_\lambda)
  \le
  C\xk{
    R_\lambda^2
    +
    \norm{h_\lambda}_{\mathcal{H}}^2
    +
    \lambda^{-2} W_\lambda^2
  }.
\]
Here we used \(0\le q_\lambda \le1\),
\(\norm{S}\le\kappa\), and
\(\norm{\varphi_\lambda(T_X)}\le E/\lambda\).
The regularized inverse bound gives
\begin{equation}
  \label{eq:global-population-solution-bound}
  \norm{h_\lambda}_{\mathcal{H}}^2
  \le
  \frac{E^2}{4\lambda}
  \norm{f^*}_{L^2(\mu)}^2.
\end{equation}
Indeed, the spectral theorem and
\(t\varphi_\lambda(t)^2\le E^2t/(t+\lambda)^2
  \le E^2/(4\lambda)\)
give the claim.

The conditional variance has the global bound
\begin{align}
  \label{eq:global-variance-bound}
  \mathsf{V}_X(\varphi_\lambda)
  &\le
  \frac{\sigma^2 \norm{T}}{n}
  \Tr\xk{\varphi_\lambda(T_X)^2T_X }
  \\
  &\le
  C\frac{\sigma^2}{\lambda},
\end{align}
because \(T_X\) has rank at most \(n\) and
\(t/(t+\lambda)^2\le1/(4\lambda)\).
Set
\[
  Y_{\lambda,i}
  =
  \norm{T_\lambda^{-1/2} k_{x_i}}_{\mathcal{H}}^2
  r_\lambda(x_i)^2.
\]
Convexity gives
\[
  W_\lambda^2
  \le
  \frac{1}{n}\sum_{i=1}^n Y_{\lambda,i}.
\]
Since \(Y_{\lambda,i}\) is independent of
\(\mathcal{B}_{\lambda,-i}\), exchangeability,
\cref{eq:leave-one-out-bad-design-tail}, and the leverage bound give
\begin{align}
  \label{eq:uncentered-rhs-bad-event}
  \E_X \zk{
    W_\lambda^2 \mathbf{1}_{\mathcal{G}_\lambda^{\mathsf{c}}}
  }
  &\le
  \frac{1}{n}\sum_{i=1}^n
  \E_X \zk{
    Y_{\lambda,i} \mathbf{1}_{\mathcal{B}_{\lambda,-i}}
  }
  \\
  &\le
  C\mathcal{L}_\lambda R_\lambda^2
  \exp\xk{-c\frac{n}{\mathcal{L}_\lambda}}.
\end{align}

Set
\[
  H_\lambda
  =
  \frac{n}{\lambda \mathcal{N}_q(\lambda)}.
\]
The regularized inverse bound implies
\[
  \mathcal{N}_q(\lambda)
  \le
  E^2 \mathcal{N}_1(\lambda),
  \qquad
  \lambda \mathcal{N}_q(\lambda)
  \le
  E^2 \Tr(T).
\]
Since \(\mathcal{L}_\lambda/n\to0\), for all sufficiently large \(n\),
\begin{equation}
  \label{eq:polynomial-envelope-scale}
  \lambda^{-1} \le H_\lambda,
  \qquad
  \mathcal{N}_1(\lambda) \le n\le CH_\lambda.
\end{equation}
Let \(P_\lambda=\mathbb{P}_X(\mathcal{G}_\lambda^{\mathsf{c}})\).
If \(R_\lambda>0\), \cref{eq:uncentered-rhs-bad-event} and
\(D_\lambda \ge R_\lambda^2\) give
\[
  \frac{\lambda^{-2}}{D_\lambda}
  \E_X \zk{
    W_\lambda^2 \mathbf{1}_{\mathcal{G}_\lambda^{\mathsf{c}}}
  }
  \le
  C\lambda^{-2}\mathcal{L}_\lambda
  \exp\xk{-c\frac{n}{\mathcal{L}_\lambda}}
  \le
  CH_\lambda^3
  \exp\xk{-c\frac{n}{\mathcal{L}_\lambda}}.
\]
If \(R_\lambda=0\), then \(W_\lambda=0\) almost surely.
The remaining terms in
\cref{eq:global-population-solution-bound,eq:global-variance-bound}, divided by
\(D_\lambda\), contribute at most
\(C(1+\norm{f^*}_{L^2(\mu)}^2)(1+H_\lambda)P_\lambda\).
To control the target norm uniformly, put
\[
  b_\lambda
  =\frac{n\mathcal S_{f^*}(\lambda)\gamma_\lambda^2}{\mathcal{N}_q(\lambda)}
  =o(1).
\]
Since \(\mathcal{N}_q(\lambda)\le E^2\mathcal{N}_1(\lambda)\le E^2\mathcal L_\lambda\)
and \(n\gamma_\lambda^2\ge\mathcal L_\lambda\),
\(\mathcal S_{f^*}(\lambda)\le E^2b_\lambda\).
The spectral theorem, \(0<\alpha_\psi\le1\), and
\cref{eq:polynomial-envelope-scale} give
\[
  \norm{f^*}_{L^2(\mu)}^2
  \le
  \xk{1+\frac{\kappa^2}{\lambda}}^{\alpha_\psi}
  \mathcal S_{f^*}(\lambda)
  \le
  E^2(1+\kappa^2H_\lambda)b_\lambda.
\]
Thus the remaining contribution is bounded by
\(C(1+H_\lambda)^2P_\lambda\) for all sufficiently large \(n\),
uniformly whenever the assumptions hold uniformly.
Put \(x_\lambda=n/\mathcal{L}_\lambda\).
The expectation condition
\cref{eq:leverage-expectation-scale} gives directly
\(\log(e+H_\lambda)=o(x_\lambda)\).
Hence \cref{eq:bad-design-tail} implies
\[
  H_\lambda^a P_\lambda^b \longrightarrow0
  \qquad
  \text{for every fixed }a<\infty\text{ and }b>0.
\]
This proves \cref{eq:fixed-bad-event-negligible}.
\end{proof}

\begin{proof}[Proof of \Cref{thm:risk-equivalence-expectation}]
On \(\mathcal{G}_\lambda\), use
\Cref{lem:risk-expectation-good-event}; on its complement, use
\Cref{lem:risk-expectation-bad-event} and
\[
  \frac{
    \left|
      \mathcal{E}_n^{\mathsf{np}}(q_\lambda;f^*\mid X)-D_\lambda
    \right|
  }{D_\lambda}
  \le
  1+
  \frac{\mathcal{E}_n^{\mathsf{np}}(q_\lambda;f^*\mid X)}{D_\lambda}.
\]
Together with the tower identity \cref{eq:risk-in-expectation}, this gives
\[
  \frac{
    \left|
      \mathcal{E}_n^{\mathsf{np}}(q_\lambda;f^*)-D_\lambda
    \right|
  }{D_\lambda}
  \le
  \E_X \zk{
    \frac{
      \left|
        \mathcal{E}_n^{\mathsf{np}}(q_\lambda;f^*\mid X)-D_\lambda
      \right|
    }{D_\lambda}
  }
  \longrightarrow0,
\]
which proves \cref{eq:risk-equivalence-expectation}.
All bounds on the good design event and its complement have constants depending
only on the constants in the assumptions, and every vanishing deterministic
remainder above is controlled by the scale conditions.
The same argument therefore gives the asserted uniform \(o(1)\) remainder.
\end{proof}
\clearpage{}
\clearpage{}\section{Scale Verification and Proofs for Applications}
\label{sec:pinsker-minimax}

This appendix first verifies the scale conditions under polynomial and exponential eigenvalue decay and for high-dimensional spherical kernels.
It then derives the Pinsker constant under exact polynomial counting and proves the matching minimax lower bound.
Finally, it establishes sharp constants and optimal rates in the high-dimensional spherical model.
The application proofs use the scale verifications below and the expectation risk equivalence proved in \Cref{sec:risk-equivalence-expectation}.

\subsection{Verification of the scale conditions}

\label{subsec:scale-verification-examples}

\subsubsection{Polynomial eigenvalue decay}

\begin{proposition}[Scale conditions under polynomial eigenvalue decay]
  \label{prop:polynomial-scale-verification}
  Suppose \Cref{ass:lipschitz-filter} holds and, for some \(\beta>1\),
  \[
    \lambda_j\asymp j^{-\beta}.
  \]
  Let \(f^*\in\mathcal{F}_s(R)\) for fixed \(s,R>0\), and let
  \(\lambda\asymp n^{-\theta}\) for some \(0<\theta<\beta\).
  Suppose also that
  \(\mathcal{L}_\lambda\lesssim\mathcal{N}_1(\lambda)\).
  Then the two conditions in \Cref{ass:conditional-equivalence-scale} and
  the additional expectation condition
  \cref{eq:leverage-expectation-scale} hold.
  The conclusion is uniform over families with fixed \(\beta,s,\theta,\rho_\psi\)
  and uniform eigenvalue, source-radius, filter, leverage-comparison, and
  regularization comparison constants in \(\lambda\asymp n^{-\theta}\).
\end{proposition}

\begin{proof}
Splitting the defining sum at an index of order
\(\lambda^{-1/\beta}\) gives
\[
  \mathcal{N}_1(\lambda)
  \asymp
  \lambda^{-1/\beta},
  \qquad
  \ell_\lambda
  \asymp
  \log(1/\lambda).
\]
The regularized inverse bound implied by \Cref{ass:lipschitz-filter} yields
\[
  q_\lambda(t)
  \le
  A\frac{t}{t+\lambda},
\]
and hence
\[
  \mathcal{N}_q(\lambda)
  \le
  A^2\sum_{j\ge1}
  \xk{\frac{\lambda_j}{\lambda_j+\lambda}}^2
  \lesssim
  \mathcal{N}_1(\lambda).
\]
The qualification bound gives
\(
  \psi_\lambda(t)\le A(\lambda/t)^{\rho_\psi}
\)
for \(t>0\).
Thus, for a fixed sufficiently large \(c>0\),
\(q_\lambda(t)\ge1/2\) whenever \(t\ge c\lambda\), and consequently
\[
  \mathcal{N}_q(\lambda)
  \ge
  \frac14\#\{j:\lambda_j\ge c\lambda\}
  \gtrsim
  \lambda^{-1/\beta}.
\]
Therefore
\begin{equation}
  \label{eq:polynomial-effective-dimensions}
  \mathcal{N}_q(\lambda)
  \asymp
  \mathcal{N}_1(\lambda)
  \asymp
  \lambda^{-1/\beta}.
\end{equation}

Since \(\lambda\asymp n^{-\theta}\),
\[
  \frac{\mathcal{N}_1(\lambda)}{\mathcal{N}_q(\lambda)}
  \sqrt{\frac{\mathcal{L}_\lambda\ell_\lambda}{n}}
  \lesssim
  n^{(\theta/\beta-1)/2}\sqrt{\log n}
  =o(1).
\]

Let \(\alpha_s=\min\{s,\alpha_\psi\}>0\).
The source condition gives
\(
  \mathcal{S}_{f^*}(\lambda)\lesssim\lambda^{\alpha_s}
\), while
\[
  \gamma_\lambda^2
  \asymp
  \frac{\lambda^{-1/\beta}}{n}
  \xk{\log(1/\lambda)}^3.
\]
Using \cref{eq:polynomial-effective-dimensions}, we obtain
\[
  \frac{
    n\mathcal{S}_{f^*}(\lambda)\gamma_\lambda^2
  }{
    \mathcal{N}_q(\lambda)
  }
  \lesssim
  \lambda^{\alpha_s}
  \xk{\log(1/\lambda)}^3
  =
  o(1).
\]
This proves the two conditional scale conditions.
Finally,
\[
  \frac{\mathcal{L}_\lambda}{n}
  \log\xk{e+\frac{n}{\lambda\mathcal{N}_q(\lambda)}}
  \lesssim
  n^{\theta/\beta-1}\log n
  =
  o(1),
\]
which proves \cref{eq:leverage-expectation-scale}.
The uniform statement follows because every comparison above uses only the
corresponding uniform constants.
\end{proof}

\begin{proposition}[Polynomial eigenvalue decay under the general leverage bound]
  \label{prop:polynomial-universal-leverage-verification}
  Suppose \Cref{ass:basic-model,ass:lipschitz-filter} holds and, for some
  \(\beta>1\),
  \[
    \lambda_j\asymp j^{-\beta}.
  \]
  Let \(f^*\in\mathcal{F}_s(R)\) for fixed \(s,R>0\), and let
  \(\lambda\asymp n^{-\theta}\) for some \(0<\theta<1\).
  If
  \[
    \min\{s,\alpha_\psi\}>1-1/\beta,
  \]
  then the two conditions in \Cref{ass:conditional-equivalence-scale} and
  the additional expectation condition
  \cref{eq:leverage-expectation-scale} hold.
  The conclusion is uniform over families with fixed \(\beta,s,\theta,\rho_\psi\),
  a uniform kernel bound, and uniform eigenvalue, source-radius, filter, and
  regularization comparison constants in \(\lambda\asymp n^{-\theta}\).
\end{proposition}

\begin{proof}
The comparison \cref{eq:polynomial-effective-dimensions} still holds.  Write
\(\alpha_s=\min\{s,\alpha_\psi\}\).  The universal envelope in
\cref{eq:leverage-scale-bounds} gives
\[
  \mathcal{L}_\lambda\lesssim\lambda^{-1}.
\]
Since \(\ell_\lambda\asymp\log(1/\lambda)\),
\[
  \frac{\mathcal{N}_1(\lambda)}{\mathcal{N}_q(\lambda)}
  \sqrt{\frac{\mathcal{L}_\lambda\ell_\lambda}{n}}
  \lesssim
  n^{(\theta-1)/2}\sqrt{\log n}
  =o(1).
\]
Moreover, the source condition and the definition of \(\gamma_\lambda\)
give
\[
  \mathcal{S}_{f^*}(\lambda)\lesssim\lambda^{\alpha_s},
  \qquad
  \gamma_\lambda^2
  \lesssim
  \frac{\lambda^{-1}}{n}
  \xk{\log(1/\lambda)}^3.
\]
Thus,
\[
  \frac{
    n\mathcal{S}_{f^*}(\lambda)\gamma_\lambda^2
  }{
    \mathcal{N}_q(\lambda)
  }
  \lesssim
  \lambda^{\alpha_s-1+1/\beta}
  \xk{\log(1/\lambda)}^3
  =o(1).
\]
Finally,
\[
  \frac{\mathcal{L}_\lambda}{n}
  \log\xk{e+\frac{n}{\lambda\mathcal{N}_q(\lambda)}}
  \lesssim
  n^{\theta-1}\log n
  =o(1).
\]
The asserted uniformity follows from the corresponding uniform bounds.
\end{proof}

\subsubsection{Exponential eigenvalue decay}

\begin{proposition}[Scale conditions under exponential eigenvalue decay]
  \label{prop:exponential-scale-verification}
  Suppose \Cref{ass:lipschitz-filter} holds and, for some \(c_0,\nu>0\),
  \[
    \lambda_j\asymp \exp(-c_0j^\nu).
  \]
  Let \(f^*\in\mathcal{F}_s(R)\) for fixed \(s,R>0\), and let
  \(\lambda\asymp n^{-\theta}\) for any \(\theta>0\).
  Suppose also that
  \(\mathcal{L}_\lambda\lesssim\mathcal{N}_1(\lambda)\).
  Then the two conditions in \Cref{ass:conditional-equivalence-scale} and
  the additional expectation condition
  \cref{eq:leverage-expectation-scale} hold.
  The conclusion is uniform over families with fixed \(c_0,\nu,s,\theta,\rho_\psi\)
  and uniform eigenvalue, source-radius, filter, leverage-comparison, and
  regularization comparison constants in \(\lambda\asymp n^{-\theta}\).
\end{proposition}

\begin{proof}
Set \(m_\lambda=\xk{\log(1/\lambda)}^{1/\nu}\).
Splitting the defining sum at an index of order \(m_\lambda\) gives
\[
  \mathcal{N}_1(\lambda)
  \asymp
  m_\lambda,
  \qquad
  \ell_\lambda
  \asymp
  \log\log(1/\lambda).
\]
The upper and lower comparisons with \(\mathcal{N}_q(\lambda)\) used in the
preceding proof apply verbatim, while
\(
  \#\{j:\lambda_j\ge c\lambda\}\asymp m_\lambda
\)
for every fixed \(c>0\).
Hence
\[
  \mathcal{N}_q(\lambda)
  \asymp
  \mathcal{N}_1(\lambda)
  \asymp
  \xk{\log(1/\lambda)}^{1/\nu}.
\]
Since \(\lambda\asymp n^{-\theta}\),
\[
  \frac{\mathcal{N}_1(\lambda)}{\mathcal{N}_q(\lambda)}
  \sqrt{\frac{\mathcal{L}_\lambda\ell_\lambda}{n}}
  \lesssim
  \frac{(\log n)^{1/(2\nu)}\sqrt{\log\log n}}{\sqrt{n}}
  =o(1),
\]

Let \(\alpha_s=\min\{s,\alpha_\psi\}>0\).
The source condition and the preceding estimates of the effective dimensions give
\[
  \frac{
    n\mathcal{S}_{f^*}(\lambda)\gamma_\lambda^2
  }{
    \mathcal{N}_q(\lambda)
  }
  \lesssim
  n^{-\theta\alpha_s}(\log\log n)^3
  =o(1).
\]
Finally,
\[
  \frac{\mathcal{L}_\lambda}{n}
  \log\xk{e+\frac{n}{\lambda\mathcal{N}_q(\lambda)}}
  \lesssim
  \frac{(\log n)^{1+1/\nu}}{n}
  =o(1).
\]
This proves the three scale conditions, including their uniform version under
the stated uniform bounds.
\end{proof}

\subsubsection{High-dimensional spherical kernels}

In the spherical setting of \Cref{sec:high-dimensional-scale-verification}, let \(\mu_{d,k}\) and \(D_{d,k}\) denote the degree-\(k\) eigenvalue and multiplicity.
For each fixed degree \(k\), the spherical spectral asymptotics give
\begin{equation}
  \label{eq:high-dimensional-spectral-asymptotics}
  \mu_{d,k}\sim a_kk!\,d^{-k},
  \qquad
  D_{d,k}\sim\frac{d^k}{k!},
  \qquad
  \max_{j\ge k+1}\mu_{d,j}=O(d^{-(k+1)}).
\end{equation}
These estimates follow from \citet[Lemma 4.4]{lu2024_PinskerBound} and the corresponding spherical spectral bounds in \citet[Appendix B.1]{zhang2025_OptimalRates}.
Their convention uses \(\mathbb S^d\) rather than \(\mathbb S^{d-1}\); replacing their dimension by \(d-1\) leaves the asymptotic equivalents unchanged.
In particular, for all sufficiently large \(d\), degree \(k\) has the largest eigenvalue among degrees \(j\ge k\), for every fixed \(k\) needed below.

\begin{proposition}
  \label{prop:high-dimensional-scale-verification}
  In the spherical setting of \Cref{sec:high-dimensional-scale-verification}, fix an integer \(1\le m<\gamma\).
  There exists \(c_*>0\), depending only on \(m\), \(\Phi\), and the filter constants, such that for every fixed \(c\in(0,c_*]\), the sequence \(\lambda=c\,d^{-m}\) satisfies
  \Cref{ass:conditional-equivalence-scale,ass:risk-equivalence-expectation} uniformly over the source ball with fixed \(s,R>0\).
  More generally, the same conclusion holds for any deterministic sequence \(\lambda\asymp d^{-m}\) such that \(\liminf_{d\to\infty}q_\lambda(\mu_{d,m})>0\), where \(\mu_{d,m}\) is the eigenvalue on the spherical harmonics of degree \(m\).
\end{proposition}
\begin{proof}
  By \cref{eq:high-dimensional-spectral-asymptotics}, \(\mu_{d,k}\asymp d^{-k}\) and \(D_{d,k}\asymp d^k\) for each fixed \(k\).
  All constants below may depend on the fixed kernel, \(m\), \(c\), and the filter and source parameters, but not on \(d\) or the target in the source ball.
  Choose \(b_m>0\) such that \(\mu_{d,m}\ge b_m d^{-m}\) for all sufficiently large \(d\).
  The size bound in \Cref{ass:lipschitz-filter} implies
  \[
    0\le\psi_\lambda(t)
    \le A\xk{\frac{\lambda}{t+\lambda}}^{\rho_\psi}.
  \]
  Choose \(c_*>0\) so small that \(A(1+b_m/c_*)^{-\rho_\psi}\le1/2\).
  Then, for every fixed \(0<c\le c_*\),
  \(q_\lambda(\mu_{d,m})\ge1/2\) for all sufficiently large \(d\).

  For the general assertion, fix \(b>0\) such that \(q_\lambda(\mu_{d,m})\ge b\) for all sufficiently large \(d\).
  The preceding choice of \(c\) is a special case with \(b=1/2\).
  For an orthonormal basis \(\{Y_{d,k,j}\}_{j=1}^{D_{d,k}}\) of the degree-\(k\) spherical harmonics, the addition formula gives
  \(\sum_{j=1}^{D_{d,k}}Y_{d,k,j}(x)^2=D_{d,k}\).
  Hence the regularized leverage is constant in \(x\), and
  \[
    \mathcal L_\lambda
    =\sum_{k\ge0}\frac{\mu_{d,k}}{\mu_{d,k}+\lambda}D_{d,k}
    =\mathcal N_1(\lambda).
  \]
  The trace bound yields \(\mathcal N_1(\lambda)\le\Phi(1)/\lambda\lesssim d^m\), while the degree-\(m\) block gives the matching lower bound.
  Moreover, \Cref{ass:lipschitz-filter} gives
  \[
    q_\lambda(t)^2
    \le A^2\xk{\frac{t}{t+\lambda}}^2
    \le A^2\frac{t}{t+\lambda}.
  \]
  Therefore,
  \[
    b^2D_{d,m}\le\mathcal N_q(\lambda)
    \le A^2\mathcal N_1(\lambda),
    \qquad
    \mathcal L_\lambda=\mathcal N_1(\lambda)
    \asymp\mathcal N_q(\lambda)\asymp d^m.
  \]
  In particular, \(\ell_\lambda\asymp\log d\) and
  \(\gamma_\lambda^2\lesssim d^{m-\gamma}(\log d)^3\).

  Set \(r=\min\{s,\alpha_\psi\}>0\).
  By the source condition and spectral calculus,
  \[
    \mathcal S_{f_d^*}(\lambda)
    \le R^2\sup_{0\le t\le\kappa^2}
      t^s\xk{\frac{\lambda}{t+\lambda}}^{\alpha_\psi}
    \lesssim R^2\lambda^r
    \lesssim d^{-mr}.
  \]
  Substituting these estimates into the two conditional scale expressions gives
  \[
    \frac{\mathcal N_1(\lambda)}{\mathcal N_q(\lambda)}
    \sqrt{\frac{\mathcal L_\lambda\ell_\lambda}{n}}
    \lesssim d^{(m-\gamma)/2}(\log d)^{1/2}=o(1)
  \]
  and
  \[
    \frac{n\mathcal S_{f_d^*}(\lambda)\gamma_\lambda^2}
      {\mathcal N_q(\lambda)}
    \lesssim d^{-mr}(\log d)^3=o(1).
  \]
  Finally, \(\lambda\mathcal N_q(\lambda)\asymp1\), so
  \[
    \frac{\mathcal L_\lambda}{n}
    \log\xk{e+\frac{n}{\lambda\mathcal N_q(\lambda)}}
    \lesssim d^{m-\gamma}\log d=o(1).
  \]
  The limits follow from \(m<\gamma\) and \(r>0\), and all estimates are uniform over the stated source ball.
\end{proof}

\subsection{Pinsker bounds under polynomial eigenvalue decay}
\label{subsec:pinsker-upper}

Throughout this subsection, the Pinsker profile, oracle scale, constant, and
minimax risk are those defined in
\Cref{subsec:pinsker-constants-interpretation}.

\subsubsection{Spectral-sum asymptotics}

\begin{lemma}[Spectral sums under exact polynomial counting]
\label{lem:pinsker-spectral-sum}
Assume that \(N_L(t)\sim Dt^{-1/\beta}\) as \(t\downarrow0\), for fixed
\(D>0\) and \(\beta>1\), and set
\[
  \gamma=\frac{1}{\beta}\in(0,1).
\]
Let \(G:[0,\infty)\to\mathbb{R}\) be absolutely continuous with
\(G(0)=0\).
Suppose that, for some
\(\eta\in(0,\min\{\gamma,1-\gamma\})\),
\[
  \int_0^1
  u^{-\gamma-\eta}|G'(u)|\dd u
  +
  \int_1^\infty
  u^{-\gamma+\eta}|G'(u)|\dd u
  <
  \infty,
\]
and that
\[
  u^{-\gamma} G(u)\longrightarrow0
  \quad\text{as }u\downarrow0
  \quad\text{and as }u\uparrow\infty.
\]
Then, as \(\lambda\downarrow0\),
\begin{equation}
  \label{eq:pinsker-spectral-sum-rule}
  \lambda^\gamma
  \sum_{j\ge1}
  G(\lambda_j/\lambda)
  \longrightarrow
  D
  \int_0^\infty
  u^{-\gamma} G'(u)\dd u
  =
  D\gamma
  \int_0^\infty
  G(u)u^{-\gamma-1} \dd u.
\end{equation}
\end{lemma}

\begin{proof}
For every fixed \(u>0\),
\[
  \lambda^\gamma N_L(\lambda u)\longrightarrow Du^{-\gamma}.
\]
Potter's bound, extended outside a fixed neighborhood of zero by the
monotonicity of \(N_L\), yields, for all sufficiently small \(\lambda\),
\[
  \lambda^\gamma N_L(\lambda u)
  \le
  C_\eta
  \xk{
    u^{-\gamma-\eta} \mathbf{1}_{\{u\le1\}}
    +
    u^{-\gamma+\eta} \mathbf{1}_{\{u>1\}}
  }.
\]
The assumed derivative integrability shows in particular that
\[
  \int_0^\infty
  N_L(\lambda u)|G'(u)|\dd u
  <
  \infty
\]
for every sufficiently small \(\lambda\).
Since \(G(x)=\int_0^x G'(u)\dd u\), the preceding absolute integrability bound permits Fubini's theorem and gives
\begin{equation}
  \label{eq:pinsker-spectral-sum-tonelli}
  \sum_{j\ge1} G(\lambda_j/\lambda)
  =
  \int_0^\infty
  N_L(\lambda u)G'(u)\dd u.
\end{equation}
Strictly, Fubini first gives the counting function
\(N_L^{>}(t)=\#\{j:\lambda_j>t\}\).
It may be replaced by \(N_L\) in the integral because the two functions
differ only at the countable set
\(\{\lambda_j/\lambda:j\ge1\}\).
The same integrable bound permits dominated convergence in
\cref{eq:pinsker-spectral-sum-tonelli}.
Finally, the two boundary conditions give
\[
  \int_0^\infty u^{-\gamma} G'(u)\dd u
  =
  \gamma
  \int_0^\infty G(u)u^{-\gamma-1} \dd u
\]
by integration by parts.
\end{proof}

\subsubsection{Filter admissibility}

We verify the filter condition for the Pinsker profile before calculating its sequence risk.
Put \(r=s/2\).
For \(u=t/\lambda\), the dimensionless profiles are
\[
  \phi_{\lambda,s}^{\sharp}(u)
  =
  \begin{cases}
    0,
    &0\le u\le1,
    \\
    u^{-1}(1-u^{-s/2}),
    &u>1,
  \end{cases}
  \qquad
  \psi_{\lambda,s}^{\sharp}(u)
  =
  \begin{cases}
    1,
    &0\le u\le1,
    \\
    u^{-s/2},
    &u>1.
  \end{cases}
\]
Both profiles are continuous at \(u=1\), piecewise \(C^1\), and
absolutely continuous on every compact interval.
For \(u\ne1\),
\[
  (\phi_{\lambda,s}^{\sharp})'(u)
  =
  \mathbf{1}_{\{u>1\}}
  \xk{
    -u^{-2}
    +(r+1)u^{-r-2}
  },
  \qquad
  (\psi_{\lambda,s}^{\sharp})'(u)
  =
  -r\mathbf{1}_{\{u>1\}}u^{-r-1}.
\]
With
\[
  \rho_\psi=\min(r,1)\in[1/2,1],
\]
and because \(r\ge\rho_\psi\), there is a constant \(C_s\),
independent of \(\lambda\), such that
\[
  |\phi_{\lambda,s}^{\sharp}(u)|
  \le
  C_s
  (1+u)^{-1},
  \qquad
  |(\phi_{\lambda,s}^{\sharp})'(u)|
  \le
  C_s
  (1+u)^{-2},
\]
and
\[
  |\psi_{\lambda,s}^{\sharp}(u)|
  \le
  C_s
  (1+u)^{-\rho_\psi},
  \qquad
  |(\psi_{\lambda,s}^{\sharp})'(u)|
  \le
  C_s
  (1+u)^{-\rho_\psi-1}.
\]
The shrinkage satisfies \(0\le q_{\lambda,s}^{\mathsf{Pin}}\le1\).
If \(t>\lambda\), then
\[
  (t+\lambda)\varphi_{\lambda,s}^{\mathsf{Pin}}(t)
  =
  \xk{1+\frac{\lambda}{t}}
  \xk{
    1-\xk{\frac{\lambda}{t}}^r
  }
  \le2,
\]
while the left-hand side is zero for \(t\le \lambda\).
Thus
\[
  \sup_{0\le t\le\kappa^2}
  (t+\lambda)\varphi_{\lambda,s}^{\mathsf{Pin}}(t)
  \le2,
\]
and, for \(0\le\tau\le r\),
\[
  t^\tau \psi_{\lambda,s}^{\mathsf{Pin}}(t)
  =
  \begin{cases}
    t^\tau \le \lambda^\tau,
    &t\le \lambda,
    \\
    \lambda^r t^{\tau-r} \le \lambda^\tau,
    &t>\lambda.
  \end{cases}
\]
Consequently,
\[
  \sup_{0\le t\le\kappa^2}
  t^\tau \psi_{\lambda,s}^{\mathsf{Pin}}(t)
  \le
  \lambda^\tau.
\]
Thus \Cref{cond:smooth-filter} holds with residual order \(\rho_\psi\) and
constants independent of \(\lambda\).
In particular, the main assumption \Cref{ass:lipschitz-filter} also holds.

\subsubsection{Sequence risk and optimal regularization}

Writing \(h=\sum_{j\ge1} h_j e_j\), the sequence bias is
\[
  \mathcal{B}_n
  (q_{\lambda,s}^{\mathsf{Pin}};L^{s/2} h)
  =
  \sum_{j\ge1} w_{j,\lambda} h_j^2,
  \qquad
  w_{j,\lambda}
  =
  \lambda_j^s
  \psi_{\lambda,s}^{\mathsf{Pin}}(\lambda_j)^2.
\]
The weights have the explicit form
\[
  w_{j,\lambda}
  =
  \begin{cases}
    \lambda_j^s,
    &\lambda_j \le \lambda,
    \\
    \lambda^s,
    &\lambda_j>\lambda.
  \end{cases}
\]
Hence \(\sup_j w_{j,\lambda} \le \lambda^s\).
For all sufficiently small \(\lambda<\lambda_1\), we also have
\(w_{1,\lambda}=\lambda^s\), and therefore
\begin{equation}
  \label{eq:pinsker-exact-worst-bias}
  \sup_{f\in\mathcal{F}_s(R)}
  \mathcal{B}_n(q_{\lambda,s}^{\mathsf{Pin}};f)
  =
  \sup_{\norm{h}\le R}
  \sum_{j\ge1} w_{j,\lambda} h_j^2
  =
  R^2 \lambda^s.
\end{equation}
Equality is attained by the source coordinate \(h=Re_1\).

Set \(\gamma=1/\beta\) and define
\[
  G_{\mathsf{V}}(u)
  =
  (1-u^{-r})^2\mathbf{1}_{\{u>1\}}.
\]
This profile satisfies the conditions of
\Cref{lem:pinsker-spectral-sum}, and
\[
  q_{\lambda,s}^{\mathsf{Pin}}(\lambda_j)^2
  =
  G_{\mathsf{V}}(\lambda_j/\lambda).
\]
Consequently,
\begin{align}
  \sum_{j\ge1}
  q_{\lambda,s}^{\mathsf{Pin}}(\lambda_j)^2
  &\sim
  DK_{\mathsf{V}} \lambda^{-\gamma},
  \label{eq:pinsker-variance-sum-asymptotic}
  \\
  K_{\mathsf{V}}
  &=
  \gamma
  \int_1^\infty
  (1-u^{-r})^2u^{-\gamma-1} \dd u
  \notag
  \\
  &=
  1-\frac{2\gamma}{\gamma+r}
  +\frac{\gamma}{\gamma+2r}
  \notag
  \\
  &=
  1-\frac{4}{\rho+2}
  +\frac{1}{\rho+1}
  =
  \frac{\rho^2}{(\rho+1)(\rho+2)}.
  \notag
\end{align}
Here we used
\[
  \rho=\frac{2r}{\gamma}=s\beta,
  \qquad
  \frac{\gamma}{\gamma+r}=\frac{2}{\rho+2},
  \qquad
  \frac{\gamma}{\gamma+2r}=\frac{1}{\rho+1}.
\]
Consequently,
\begin{equation}
  \label{eq:pinsker-leading-sequence-risk}
  \sup_{f\in\mathcal{F}_s(R)}
  \mathcal{E}_n^{\mathsf{seq}}
  (q_{\lambda,s}^{\mathsf{Pin}};f)
  =
  R^2 \lambda^s
  +
  \frac{\sigma^2 D}{n}
  \frac{\rho^2}{(\rho+1)(\rho+2)}
  \lambda^{-1/\beta}
  \xk{1+o(1)}.
\end{equation}

To minimize the leading expression, set
\[
  F_n(\lambda)
  =
  R^2 \lambda^s
  +
  \frac{\sigma^2 D}{n}
  K_{\mathsf{V}} \lambda^{-\gamma}.
\]
Its unique critical point satisfies
\begin{align*}
  F_n'(\lambda)=0
  &\iff
  sR^2 \lambda^{s-1}
  =
  \gamma
  \frac{\sigma^2 D}{n}
  K_{\mathsf{V}} \lambda^{-\gamma-1}
  \\
  &\iff
  \lambda^{s+\gamma}
  =
  \frac{\sigma^2 D}{nR^2}
  \frac{K_{\mathsf{V}}}{\rho}
  =
  \frac{
    \sigma^2 D \rho
  }{
    nR^2(\rho+1)(\rho+2)
  }.
\end{align*}
Since \(s+\gamma=(\rho+1)/\beta\), this critical point is precisely
\(\lambda_n^*\) in \cref{eq:pinsker-oracle-scale}.
The derivative equation also shows that the leading variance is
\(\rho=s/\gamma\) times the worst-case bias.
It follows that
\begin{align}
  F_n(\lambda_n^*)
  &=
  (\rho+1)R^2(\lambda_n^*)^s
  \notag
  \\
  &=
  C_{\mathsf{Pin}}n^{-\rho/(\rho+1)}.
  \label{eq:pinsker-oracle-objective}
\end{align}
Combining
\cref{eq:pinsker-leading-sequence-risk,eq:pinsker-oracle-objective}
gives
\[
  \sup_{f\in\mathcal{F}_s(R)}
  \mathcal{E}_n^{\mathsf{seq}}
  (q_{\lambda_n^*,s}^{\mathsf{Pin}};f)
  \sim
  C_{\mathsf{Pin}}n^{-\rho/(\rho+1)}.
\]

\subsubsection{Attainment in kernel regression}

\begin{proof}[Proof of \Cref{cor:pinsker-upper}]
The preceding calculations verify the filter assumption and identify the exact leading sequence risk.
To verify the comparison scales, the exact counting assumption implies
\(\lambda_j\asymp j^{-\beta}\), and the oracle scale satisfies
\[
  \lambda_n^*\asymp n^{-\theta},
  \qquad \theta=\frac{\beta}{s\beta+1}\in(0,1).
\]
Since \(s\ge1\) and \(\rho_\psi=\min\{s/2,1\}\), we have
\[
  \min\{s,\alpha_\psi\}=1>1-\frac1\beta.
\]
Thus \Cref{prop:polynomial-universal-leverage-verification} gives all three scale conditions uniformly over \(\mathcal F_s(R)\).
The uniformity clause in \Cref{thm:risk-equivalence-expectation} therefore
applies to the family indexed by \(f\in\mathcal{F}_s(R)\).
Together with the preceding sequence calculation, it gives both equalities in
\cref{eq:pinsker-upper}.
\end{proof}

The Gaussian noise assumption in \Cref{ass:gaussian-pinsker-spectrum} is used in the Bayes--minimax lower bound below.
The upper bound calculation itself uses only second moments.

\subsubsection{Minimax lower bound}
\label{subsec:pinsker-lower}

The upper bound in \Cref{cor:pinsker-upper} is sharp over the full
class of measurable prediction rules.
To prove the lower half of \Cref{thm:pinsker-minimax},
we use a Gaussian prior concentrated asymptotically inside the source ball,
condition it on that ball, and compare its Bayes risk with the
orthogonal sequence experiment.
Throughout this lower-bound argument,
\Cref{ass:basic-model,ass:gaussian-pinsker-spectrum}
hold, and \(s\ge1\) and \(R>0\) are fixed.

Fix \(\delta\in(0,1)\), and write
\[
  \epsilon_n^2=\frac{\sigma^2}{n},
  \qquad
  \gamma=\frac{1}{\beta},
  \qquad
  r=\frac{s}{2}.
\]
Define the shaved threshold
\begin{equation}
  \label{eq:pinsker-shaved-threshold}
  \lambda_{n,\delta}
  =
  \lambda_n^*(1-\delta)^{-\beta/(\rho+1)}
  =
  \xk{
    \frac{
      \sigma^2 D \rho
    }{
      n(1-\delta)R^2(\rho+1)(\rho+2)
    }
  }^{\beta/(\rho+1)}.
\end{equation}
Let
\[
  J_{n,\delta}
  =
  \{j:\lambda_j>\lambda_{n,\delta}\},
  \qquad
  m_{n,\delta}
  =
  |J_{n,\delta}|,
\]
and, for \(j\in J_{n,\delta}\), set
\[
  x_{j,n}
  =
  \xk{
    \frac{\lambda_{n,\delta}}{\lambda_j}
  }^{s/2}.
\]
The active set is finite.
For every fixed \(\zeta>0\),
\[
  N_L \xk{(1+\zeta)\lambda_{n,\delta}}
  \le
  m_{n,\delta}
  \le
  N_L(\lambda_{n,\delta}).
\]
First letting \(n\to\infty\) and then \(\zeta\downarrow0\) gives
\begin{equation}
  \label{eq:pinsker-active-dimension}
  m_{n,\delta}
  \sim
  D\lambda_{n,\delta}^{-\gamma}
  \asymp
  n^{1/(\rho+1)}.
\end{equation}
Coordinates with \(\lambda_j=\lambda_{n,\delta}\) may be omitted:
there \(x_{j,n}=1\), the prior variance below is zero, and all risk and
energy profiles vanish.
Consider the centered Gaussian prior
\(\pi_{n,\delta}^{G}\) on the active population eigenspace under which
the coordinates are independent and
\begin{equation}
  \label{eq:pinsker-gaussian-prior-variance}
  f_j^*
  \sim
  \mathcal{N}(0,v_{j,n}),
  \qquad
  v_{j,n}
  =
  \epsilon_n^2(x_{j,n}^{-1}-1),
  \qquad
  j\in J_{n,\delta},
\end{equation}
with \(f_j^*=0\) outside \(J_{n,\delta}\).
The scalar Gaussian update in the orthogonal sequence experiment has
posterior shrinkage
\begin{equation}
  \label{eq:pinsker-prior-shrinkage}
  \frac{v_{j,n}}{v_{j,n}+\epsilon_n^2}
  =
  1-x_{j,n}
  =
  q_{\lambda_{n,\delta},s}^{\mathsf{Pin}}(\lambda_j).
\end{equation}
Thus the prior is the finite-dimensional Gaussian prior dual to the
Pinsker shrinkage rule at the shaved threshold.
For a prior \(\pi\), write its Bayes risk as
\[
  \mathfrak{B}_n(\pi)
  =
  \inf_{\widehat{g}_n \in\mathfrak{D}_n}
  \int
  \mathcal{R}_n(\widehat{g}_n,f)
  \pi(\dd f).
\]

\begin{lemma}[Gaussian prior shaving]
\label{lem:pinsker-prior-shaving}
Let
\[
  Q_n
  =
  \sum_{j\in J_{n,\delta}}
  \lambda_j^{-s}(f_j^*)^2
\]
and define the supported prior
\[
  \pi_{n,\delta}^{R}
  =
  \pi_{n,\delta}^{G}
  (\,\cdot\mid Q_n \le R^2).
\]
On the active eigenspace, \(Q_n\) is the squared norm of the canonical
source coordinate, and hence
\[
  Q_n \le R^2
  \quad\Longleftrightarrow\quad
  f\in\mathcal{F}_s(R).
\]
Then
\[
  \E_{\pi_{n,\delta}^{G}}Q_n
  =
  (1-\delta)R^2 \xk{1+o(1)},
  \qquad
  \pi_{n,\delta}^{G}(Q_n>R^2)
  \le
  \exp(-c_\delta m_{n,\delta})
\]
for all sufficiently large \(n\), where \(c_\delta>0\).
Moreover, there is a deterministic
\(r_{n,\delta}=o\xk{n^{-\rho/(\rho+1)} }\) such that
\begin{equation}
  \label{eq:pinsker-prior-shaving}
  \mathfrak{B}_n(\pi_{n,\delta}^{R})
  \ge
  \mathfrak{B}_n(\pi_{n,\delta}^{G})
  -
  r_{n,\delta}.
\end{equation}
\end{lemma}

\begin{proof}
Writing
\[
  b_{j,n}
  =
  \lambda_j^{-s} v_{j,n}
  =
  \epsilon_n^2 \lambda_{n,\delta}^{-s}
  x_{j,n}(1-x_{j,n}),
\]
where
\(\lambda_j^{-s}=\lambda_{n,\delta}^{-s} x_{j,n}^2\), we have
\[
  Q_n
  =
  \sum_{j\in J_{n,\delta}}
  b_{j,n} Z_j^2,
  \qquad
  Z_j \stackrel{\mathrm{i.i.d.}}{\sim}\mathcal{N}(0,1).
\]
If \(q_{j,n}=1-x_{j,n}\), then
\begin{equation}
  \label{eq:pinsker-prior-duality-identity}
  \epsilon_n^2
  \sum_{j\in J_{n,\delta}}q_{j,n}
  =
  \epsilon_n^2
  \sum_{j\in J_{n,\delta}}q_{j,n}^2
  +
  \lambda_{n,\delta}^s
  \E_{\pi_{n,\delta}^{G}}Q_n.
\end{equation}
Indeed,
\(q_{j,n}=q_{j,n}^2+x_{j,n}(1-x_{j,n})\), while
\(\epsilon_n^2 x_{j,n}(1-x_{j,n})=\lambda_{n,\delta}^s b_{j,n}\).

For
\[
  G_{\mathsf{E}}(u)
  =
  (u^{-r}-u^{-2r})\mathbf{1}_{\{u>1\}},
\]
\Cref{lem:pinsker-spectral-sum} gives
\begin{align}
  \sum_{j\in J_{n,\delta}}
  x_{j,n}(1-x_{j,n})
  &\sim
  D K_{\mathsf{E}} \lambda_{n,\delta}^{-\gamma},
  \label{eq:pinsker-prior-energy-sum}
  \\
  K_{\mathsf{E}}
  &=
  \gamma
  \int_1^\infty
  (u^{-r}-u^{-2r})u^{-\gamma-1} \dd u
  \notag
  \\
  &=
  \frac{\gamma}{\gamma+r}
  -
  \frac{\gamma}{\gamma+2r}
  \notag
  \\
  &=
  \frac{2}{\rho+2}
  -
  \frac{1}{\rho+1}
  =
  \frac{\rho}{(\rho+1)(\rho+2)}.
  \notag
\end{align}
Consequently,
\begin{align*}
  \E Q_n
  &\sim
  \epsilon_n^2 D
  K_{\mathsf{E}}
  \lambda_{n,\delta}^{-s-\gamma}
  \\
  &=
  (1-\delta)R^2.
\end{align*}
The equality of the leading terms follows from
\cref{eq:pinsker-shaved-threshold}, which is equivalent to
\[
  \lambda_{n,\delta}^{s+\gamma}
  =
  \frac{
    \epsilon_n^2 D K_{\mathsf{E}}
  }{
    (1-\delta)R^2
  }.
\]
In particular,
\(\E Q_n \le(1-\delta/2)R^2\) for all sufficiently large \(n\).

Since \(x(1-x)\le1/4\) on \([0,1]\),
\[
  \max_j b_{j,n}
  \le
  \frac{1}{4}
  \epsilon_n^2 \lambda_{n,\delta}^{-s}.
\]
The threshold identity above gives
\(\epsilon_n^2 \lambda_{n,\delta}^{-s}=O(\lambda_{n,\delta}^{\gamma})\),
whereas
\(\sum_j b_{j,n}=\E Q_n \to(1-\delta)R^2\).
Therefore
\[
  \frac{\max_j b_{j,n}}{\sum_j b_{j,n}}
  =
  O(\lambda_{n,\delta}^{\gamma}),
\]
and, using
\(\sum_j b_{j,n}^2 \le(\max_j b_{j,n})\sum_j b_{j,n}\),
\[
  \frac{\sum_j b_{j,n}^2}{(\sum_j b_{j,n})^2}
  =
  O(\lambda_{n,\delta}^{\gamma}).
\]

For completeness, if
\(0<\theta<(2\max_j b_{j,n})^{-1}\), independence gives
\begin{align*}
  \log
  \E\exp\xk{\theta(Q_n-\E Q_n)}
  &=
  \frac{1}{2}
  \sum_j
  \zk{
    -\log(1-2\theta b_{j,n})
    -2\theta b_{j,n}
  }
  \\
  &\le
  \frac{
    \theta^2 \sum_j b_{j,n}^2
  }{
    1-2\theta\max_j b_{j,n}
  }.
\end{align*}
Chernoff optimization yields the weighted chi-square bound
\[
  \mathbb{P}\xk{
    Q_n-\E Q_n
    \ge
    2\sqrt{t\sum_j b_{j,n}^2}
    +
    2t\max_j b_{j,n}
  }
  \le
  e^{-t}
\]
for every \(t>0\).
Take \(t=c_\delta m_{n,\delta}\).
By \cref{eq:pinsker-active-dimension}, both
\[
  t
  \frac{\sum_j b_{j,n}^2}{(\sum_j b_{j,n})^2}
  =
  O(c_\delta)
  \quad\text{and}\quad
  t
  \frac{\max_j b_{j,n}}{\sum_j b_{j,n}}
  =
  O(c_\delta).
\]
After division by \(\sum_j b_{j,n}\), the two terms in the deviation
threshold are therefore \(O(\sqrt{c_\delta})\) and \(O(c_\delta)\),
respectively.
Choosing the fixed \(c_\delta>0\) sufficiently small makes the
deviation threshold at most
\(\delta R^2/2\) for all large \(n\).
Since \(R^2-\E Q_n \ge\delta R^2/2\), we obtain
\[
  \mathbb{P}(Q_n>R^2)
  \le
  \exp(-c_\delta m_{n,\delta}).
\]

It remains to compare the two Bayes risks.
Let
\[
  A_n=\{Q_n \le R^2\},
  \qquad
  p_{n,\delta}=\pi_{n,\delta}^{G}(A_n),
  \qquad
  q_{n,\delta}=1-p_{n,\delta}.
\]
Let \(\widehat{g}_{n,R}\) be the posterior mean under
\(\pi_{n,\delta}^{R}\).
It is a jointly measurable finite-dimensional decision rule.
Since \(\mathcal{F}_s(R)\) is convex, the rule takes values in
\(\mathcal{F}_s(R)\), and the posterior mean identity for squared loss
shows that it attains \(\mathfrak{B}_n(\pi_{n,\delta}^{R})\).
Evaluating this rule under the Gaussian prior gives
\[
  \mathfrak{B}_n(\pi_{n,\delta}^{G})
  \le
  p_{n,\delta}
  \mathfrak{B}_n(\pi_{n,\delta}^{R})
  +
  T_n,
\]
where
\[
  T_n
  =
  \int_{A_n^c}
  \mathcal{R}_n(\widehat{g}_{n,R},f)
  \pi_{n,\delta}^{G}(\dd f).
\]
For every \(f\) in the active subspace,
\[
  \norm{f}_{L^2(\mu)}^2
  \le
  \lambda_1^s Q_n,
\]
and
\(\norm{\widehat{g}_{n,R}(X,y)}_{L^2(\mu)}^2 \le\lambda_1^s R^2\).
Moreover,
\[
  \E Q_n^2
  =
  (\E Q_n)^2
  +
  2\sum_j b_{j,n}^2
  =
  O(1).
\]
Using
\[
  \norm{\widehat{g}_{n,R}(X,y)-f}_{L^2(\mu)}^2
  \le
  2\lambda_1^s(R^2+Q_n)
\]
and then Cauchy--Schwarz gives
\begin{align*}
  T_n
  &\le
  2\lambda_1^s R^2 q_{n,\delta}
  +
  2\lambda_1^s
  \E\zk{
    Q_n \mathbf{1}_{A_n^c}
  }
  \\
  &\le
  2\lambda_1^s R^2 q_{n,\delta}
  +
  2\lambda_1^s
  (\E Q_n^2)^{1/2}
  q_{n,\delta}^{1/2}
  \\
  &\le
  C_\delta q_{n,\delta}^{1/2}.
\end{align*}
Rearranging the preceding Bayes risk inequality gives
\[
  \mathfrak{B}_n(\pi_{n,\delta}^{R})
  \ge
  \frac{
    \mathfrak{B}_n(\pi_{n,\delta}^{G})-T_n
  }{
    p_{n,\delta}
  }.
\]
Since \(\mathfrak{B}_n(\pi_{n,\delta}^{G})\ge0\),
\[
  \mathfrak{B}_n(\pi_{n,\delta}^{R})
  \ge
  \mathfrak{B}_n(\pi_{n,\delta}^{G})
  -
  \frac{
    C_\delta q_{n,\delta}^{1/2}
  }{
    p_{n,\delta}
  }.
\]
The exponential tail bound and
\cref{eq:pinsker-active-dimension} imply
\[
  q_{n,\delta}
  \le
  \exp\xk{
    -c_\delta n^{1/(\rho+1)}
  },
  \qquad
  p_{n,\delta} \longrightarrow1.
\]
Thus one may take
\[
  r_{n,\delta}
  =
  \frac{
    C_\delta q_{n,\delta}^{1/2}
  }{
    p_{n,\delta}
  }
  =
  o\xk{
    n^{-\rho/(\rho+1)}
  },
\]
which proves \cref{eq:pinsker-prior-shaving}.
\end{proof}

\begin{lemma}[Gaussian prior Bayes lower bound]
\label{lem:pinsker-gaussian-bayes-risk}
For each fixed \(\delta\in(0,1)\),
\begin{equation}
  \label{eq:pinsker-gaussian-bayes-lower}
  \mathfrak{B}_n(\pi_{n,\delta}^{G})
  \ge
  \xk{
    C_{\mathsf{Pin}}
    (1-\delta)^{1/(\rho+1)}
    -
    o(1)
  }
  n^{-\rho/(\rho+1)}.
\end{equation}
\end{lemma}

\begin{proof}
Given \(X=(x_1,\dots,x_n)\), let
\[
  \Phi_X
  =
  \bigl(e_j(x_i)\bigr)_{
    1\le i\le n,\,
    j\in J_{n,\delta}
  },
  \qquad
  G_X
  =
  \frac{1}{n}\Phi_X^\top \Phi_X,
\]
and let
\[
  V_n
  =
  \operatorname{diag}
  (v_{j,n}:j\in J_{n,\delta}).
\]
Conditional on the design, the experiment on the active subspace is the
Gaussian linear model
\[
  y=\Phi_X \theta+\varepsilon,
  \qquad
  \theta=(f_j^*:j\in J_{n,\delta}),
  \qquad
  \varepsilon\sim\mathcal{N}(0,\sigma^2 I_n).
\]
Since
\[
  \sigma^{-2} \Phi_X^\top \Phi_X
  =
  \epsilon_n^{-2} G_X,
\]
Gaussian conjugacy gives
\[
  \theta\mid X,y
  \sim
  \mathcal{N}\xk{
    C_X \sigma^{-2} \Phi_X^\top y,
    C_X
  },
\]
where
\[
  C_X
  =
  \xk{
    V_n^{-1}
    +
    \epsilon_n^{-2} G_X
  }^{-1},
\]
and the posterior mean is jointly measurable in \((X,y)\).
Because the active \(e_j\)'s are orthonormal in \(L^2(\mu)\), prediction
loss on this subspace is Euclidean coefficient loss.
The posterior mean identity for squared loss therefore gives
\[
  \mathfrak{B}_n(\pi_{n,\delta}^{G})
  =
  \E_X \Tr(C_X).
\]

The Bayes risk in the corresponding orthogonal sequence experiment is
\begin{align}
  B_{n,\delta}^{\mathsf{seq}}
  &=
  \Tr
  \xk{
    V_n^{-1}
    +
    \epsilon_n^{-2} I
  }^{-1}
  \notag
  \\
  &=
  \epsilon_n^2
  \sum_{j\in J_{n,\delta}}
  (1-x_{j,n}).
  \label{eq:pinsker-sequence-bayes-risk}
\end{align}
The last equality follows coordinatewise from
\[
  \xk{
    v_{j,n}^{-1}+\epsilon_n^{-2}
  }^{-1}
  =
  \frac{
    v_{j,n} \epsilon_n^2
  }{
    v_{j,n}+\epsilon_n^2
  }
  =
  \epsilon_n^2(1-x_{j,n}).
\]

For
\[
  G_{\mathsf{B}}(u)
  =
  (1-u^{-r})\mathbf{1}_{\{u>1\}},
\]
\Cref{lem:pinsker-spectral-sum} gives
\begin{align}
  \sum_{j\in J_{n,\delta}}
  (1-x_{j,n})
  &\sim
  D K_{\mathsf{B}} \lambda_{n,\delta}^{-\gamma},
  \label{eq:pinsker-bayes-sum}
  \\
  K_{\mathsf{B}}
  &=
  \gamma
  \int_1^\infty
  (1-u^{-r})u^{-\gamma-1} \dd u
  =
  1-\frac{\gamma}{\gamma+r}
  \notag
  \\
  &=
  \frac{\rho}{\rho+2}.
  \notag
\end{align}
The constants governing the variance in the upper bound, the prior energy, and
the Bayes risk satisfy
\begin{equation}
  \label{eq:pinsker-constant-relations}
  K_{\mathsf{V}}
  =
  \rho K_{\mathsf{E}},
  \qquad
  K_{\mathsf{B}}
  =
  K_{\mathsf{V}}+K_{\mathsf{E}}
  =
  (\rho+1)K_{\mathsf{E}}.
\end{equation}
Combining
\cref{eq:pinsker-bayes-sum,eq:pinsker-shaved-threshold} gives
\begin{align*}
  B_{n,\delta}^{\mathsf{seq}}
  &\sim
  \epsilon_n^2 D
  K_{\mathsf{B}} \lambda_{n,\delta}^{-\gamma}
  \\
  &=
  (\epsilon_n^2 D)^{\rho/(\rho+1)}
  \xk{(1-\delta)R^2}^{1/(\rho+1)}
  K_{\mathsf{B}}
  K_{\mathsf{E}}^{-1/(\rho+1)}
  \xk{1+o(1)}.
\end{align*}
Using
\cref{eq:pinsker-constant-relations,eq:pinsker-upper-constant}
reduces the preceding asymptotic expression to
\begin{equation}
  \label{eq:pinsker-sequence-bayes-constant}
  B_{n,\delta}^{\mathsf{seq}}
  \sim
  C_{\mathsf{Pin}}
  (1-\delta)^{1/(\rho+1)}
  n^{-\rho/(\rho+1)}.
\end{equation}

Because the active eigenfunctions are orthonormal in \(L^2(\mu)\),
\[
  \E_X G_X=I.
\]
The inverse map is operator convex on the positive-definite cone.
Jensen's inequality therefore gives
\begin{align*}
  \mathfrak{B}_n(\pi_{n,\delta}^{G})
  &=
  \E_X \Tr
  \xk{
    V_n^{-1}+\epsilon_n^{-2} G_X
  }^{-1}
  \\
  &\ge
  \Tr
  \xk{
    V_n^{-1}+\epsilon_n^{-2} \E_X G_X
  }^{-1}
  \\
  &=
  B_{n,\delta}^{\mathsf{seq}}.
\end{align*}
Combining this exact comparison with
\cref{eq:pinsker-sequence-bayes-constant} proves
\cref{eq:pinsker-gaussian-bayes-lower}.
\end{proof}

\begin{proof}[Proof of \Cref{thm:pinsker-minimax}]
The prior \(\pi_{n,\delta}^{R}\) is supported on
\(\mathcal{F}_s(R)\).
For every \(\widehat{g}_n \in\mathfrak{D}_n\),
\[
  \sup_{f\in\mathcal{F}_s(R)}
  \mathcal{R}_n(\widehat{g}_n,f)
  \ge
  \int
  \mathcal{R}_n(\widehat{g}_n,f)
  \pi_{n,\delta}^{R}(\dd f).
\]
Taking the infimum over the same decision class gives
\begin{equation}
  \label{eq:pinsker-bayes-minimax}
  \mathfrak{R}_n \zk{\mathcal{F}_s(R)}
  \ge
  \mathfrak{B}_n(\pi_{n,\delta}^{R}).
\end{equation}
Thus
\Cref{lem:pinsker-prior-shaving,lem:pinsker-gaussian-bayes-risk} give, for
each fixed \(\delta\in(0,1)\),
\[
  \mathfrak{R}_n \zk{\mathcal{F}_s(R)}
  \ge
  \xk{
    C_{\mathsf{Pin}}
    (1-\delta)^{1/(\rho+1)}
    -
    o(1)
  }
  n^{-\rho/(\rho+1)}.
\]
Letting first \(n\to\infty\) and then \(\delta\downarrow0\) yields
\[
  \liminf_{n\to\infty}
  n^{\rho/(\rho+1)}
  \mathfrak{R}_n \zk{\mathcal{F}_s(R)}
  \ge C_{\mathsf{Pin}}.
\]

Conversely, the prediction rule
\(S\widehat{f}_{\lambda_n^*,s}^{\mathsf{Pin}}\)
is admissible in \(\mathfrak{D}_n\), so
\Cref{cor:pinsker-upper} gives
\[
  \mathfrak{R}_n \zk{\mathcal{F}_s(R)}
  \le
  \sup_{f\in\mathcal{F}_s(R)}
  \mathcal{E}_n^{\mathsf{np}}
  (q_{\lambda_n^*,s}^{\mathsf{Pin}};f)
  =
  \xk{C_{\mathsf{Pin}}+o(1)}
  n^{-\rho/(\rho+1)}.
\]
This proves \cref{eq:pinsker-minimax}.
Both deterministic quantities in \cref{eq:pinsker-minimax-attainment} are
positive and have leading term
\(C_{\mathsf{Pin}}n^{-\rho/(\rho+1)}\), so their ratio converges to one.
\end{proof}

\subsection{High-dimensional rates and sharp constants}
\label{subsec:high-dimensional-pinsker-proof}

This subsection proves \Cref{cor:high-dimensional-pinsker,cor:high-dimensional-krr-rate} in the spherical model of \Cref{subsec:high-dimensional-optimality}.
We use the notation \(p,\delta,A_j,B_p\) from that subsection and the spherical spectral estimates in \cref{eq:high-dimensional-spectral-asymptotics}.

\subsubsection{Sharp constants for the smoothed cutoff}

\begin{proof}[Proof of \Cref{cor:high-dimensional-pinsker}]
  We first verify the filter condition.
  The dimensionless profiles are
  \[
    \phi^\sharp(u)=\frac{\chi_w(u)}{u},
    \qquad \psi^\sharp(u)=1-\chi_w(u),
  \]
  where \(\phi^\sharp=0\) near zero.
  On \([0,1/2]\) the profiles are constant; on \([1/2,2]\) they and their derivatives are bounded; on \([2,\infty)\) they equal \(u^{-1}\) and zero, respectively.
  Thus \Cref{cond:smooth-filter} holds with \(\rho_\psi=1\), with constants depending on the fixed profile but not on \(d\).

  Since \(\lambda_d=\mu_{d,p}\asymp d^{-p}\) and \(q_{\lambda_d}(\mu_{d,p})=w>0\), the general assertion in \Cref{prop:high-dimensional-scale-verification} verifies all three scale conditions.
  Indeed, \(p\ge1\) and \(\gamma=p(s+1)+\delta>p\).
  We may therefore apply \Cref{thm:risk-equivalence-expectation} uniformly over \(\mathcal F_{s,d}(R)\).

  The spectral gaps in \cref{eq:high-dimensional-spectral-asymptotics} imply, for all sufficiently large \(d\),
  \[
    q_{\lambda_d}(\mu_{d,k})=
    \begin{cases}
      1,&k<p,\\
      w,&k=p,\\
      0,&k>p.
    \end{cases}
  \]
  The variance is independent of the target, and the supremum of the squared bias over the source ball is \(R^2\max_{k\ge0}\mu_{d,k}^s\psi_{\lambda_d}(\mu_{d,k})^2\).
  Hence the worst-case sequence risk is exactly
  \begin{equation}
    \label{eq:high-dimensional-block-risk}
    \frac{\sigma^2}{n}
      \left(\sum_{k=0}^{p-1}D_{d,k}+w^2D_{d,p}\right)
    +R^2\max\left\{(1-w)^2\mu_{d,p}^s,\mu_{d,p+1}^s\right\}.
  \end{equation}

  If \(\delta>0\), then \(w=1\), so \cref{eq:high-dimensional-block-risk} is
  \[
    (1+o(1))\left(B_p d^{p-\gamma}+A_{p+1}d^{-(p+1)s}\right).
  \]
  Comparing the two exponents gives the three stated constants for \(0<\delta<s\), \(\delta=s\), and \(s<\delta<s+1\).
  If \(\delta=0\), then \(\gamma=p(s+1)\) and \(0<w<1\).
  The lower-degree variance and the bias from degrees above \(p\) are negligible, leaving
  \[
    \left(B_pw^2+A_p(1-w)^2+o(1)\right)d^{-ps}.
  \]
  The choice \(w=A_p/(A_p+B_p)\) minimizes this quadratic and gives \(A_pB_p/(A_p+B_p)\).
  Uniform expectation risk equivalence transfers each of these sequence-risk equivalents to the empirical spectral estimator.

  For the matching lower bound, apply \citet[Theorem 3.1 and Appendix D]{lu2024_PinskerBound}.
  Their source ball has squared radius denoted by \(R\), so their radius parameter is \(R^2\) in our notation, and their sample-size constant is our \(\eta\).
  Their lower-bound argument restricts to independent Gaussian noise with variance \(\sigma^2\), so it applies to the noise model here.
  With these substitutions, their lower bound is
  \[
    \mathfrak R_{n,d}(\mathcal F_{s,d}(R))
    \ge \bigl(C_{s,\gamma}-o(1)\bigr)d^{-\zeta}.
  \]
  The empirical spectral estimator is an admissible prediction rule and supplies the matching upper bound just proved.
  Since \(C_{s,\gamma}>0\), this also gives the ratio statement in \cref{eq:high-dimensional-pinsker-attainment}.
\end{proof}

\subsubsection{Optimal rates for kernel ridge regression}

\begin{proof}[Proof of \Cref{cor:high-dimensional-krr-rate}]
  Put \(h=\tfrac12\min\{\delta,s\}\), so \(a=p+h\) and \(0\le h\le1/2\).
  For KRR, the effective-dimension estimates in \citet[Lemma 23]{zhang2025_OptimalRates} and the spherical addition formula yield
  \[
    \mathcal L_\lambda=\mathcal N_1(\lambda)\asymp d^a,
    \qquad
    \mathcal N_q(\lambda)
    \asymp d^p+d^{2a-p-1}
    \asymp d^p.
  \]
  To bound the population bias, split the spectrum after degree \(p\).
  On the first \(p+1\) subspaces, \(\mu_{d,k}\gtrsim d^{-p}\), and \(s-2<0\), so
  \[
    \mu_{d,k}^s\left(\frac{\lambda}{\mu_{d,k}+\lambda}\right)^2
    \le\lambda^2\mu_{d,k}^{s-2}
    \lesssim d^{-ps-2h}.
  \]
  On the remaining subspaces, the residual is at most one and the eigenvalues are \(O(d^{-(p+1)})\), so the squared bias is bounded by \(CR^2d^{-(p+1)s}\).
  Combining the bias and variance bounds gives
  \[
    \sup_{f\in\mathcal F_{s,d}(R)}
      \mathcal E_n^{\mathsf{seq}}(q_\lambda^{\mathsf{KR}};f)
    \lesssim d^{p-\gamma}+d^{-ps-2h}+d^{-(p+1)s}
    \lesssim d^{-\zeta},
  \]
  because \(ps+2h=ps+\min\{\delta,s\}=\zeta\).
  KRR satisfies \Cref{ass:lipschitz-filter} with \(\alpha_\psi=1\).
  Since \(0<s\le1\), the uniform source bound gives
  \(\mathcal S_f(\lambda)\lesssim\lambda^s\).
  The conditional scale expressions are therefore bounded by
  \[
    \frac{\mathcal N_1}{\mathcal N_q}
      \sqrt{\frac{\mathcal L_\lambda\ell_\lambda}{n}}
    \lesssim d^{(3h-ps-\delta)/2}(\log d)^{1/2}
  \]
  and
  \[
    \sup_{f\in\mathcal F_{s,d}(R)}
      \frac{n\mathcal S_f(\lambda)\gamma_\lambda^2}{\mathcal N_q}
    \lesssim d^{-ps+h(1-s)}(\log d)^3.
  \]
  Here \(3h-\delta\le s/2\), so the first exponent is at most \(-s/4<0\), since \(p\ge1\).
  Also \(-ps+h(1-s)\le-ps+s(1-s)/2<0\).
  Finally, \(\gamma-a=ps+\delta-h>0\), and the expectation scale expression is \(O(d^{a-\gamma}\log d)=o(1)\).
  Thus uniform expectation risk equivalence applies at the stated KRR scale.

  Risk equivalence yields the same upper bound for KRR.
  The matching minimax lower bound used in \Cref{cor:high-dimensional-pinsker} applies to every estimator, including KRR, and completes the proof.
\end{proof}
\clearpage{}


\addcontentsline{toc}{section}{References}


\begin{thebibliography}{27}
\providecommand{\natexlab}[1]{#1}
\providecommand{\url}[1]{\texttt{#1}}
\expandafter\ifx\csname urlstyle\endcsname\relax
  \providecommand{\doi}[1]{doi: #1}\else
  \providecommand{\doi}{doi: \begingroup \urlstyle{rm}\Url}\fi

\bibitem[Johnstone(2017)]{johnstone2017_GaussianEstimation}
Iain~M. Johnstone.
\newblock Gaussian estimation: Sequence and wavelet models.
\newblock Unpublished manuscript, 2017.

\bibitem[Li et~al.(2024)Li, Gan, Shi, and Lin]{li2024_GeneralizationError}
Yicheng Li, Weiye Gan, Zuoqiang Shi, and Qian Lin.
\newblock Generalization error curves for analytic spectral algorithms under
  power-law decay, January 2024.
\newblock URL \url{https://doi.org/10.2139/ssrn.4633576}.
\newblock arXiv: 2401.01599.

\bibitem[Conde-Alonso et~al.(2023)Conde-Alonso, Gonz{\'a}lez-P{\'e}rez, Parcet,
  and Tablate]{condeAlonso2023_SchurMultipliersSchatten}
Jos{\'e}~M. Conde-Alonso, Adri{\'a}n~M. Gonz{\'a}lez-P{\'e}rez, Javier Parcet,
  and Eduardo Tablate.
\newblock Schur multipliers in {Schatten}--von {Neumann} classes.
\newblock \emph{Annals of Mathematics}, 198\penalty0 (3):\penalty0 1229--1260,
  October 2023.
\newblock \doi{10.4007/annals.2023.198.3.5}.
\newblock URL \url{https://doi.org/10.4007/annals.2023.198.3.5}.

\bibitem[Lo~Gerfo et~al.(2008)Lo~Gerfo, Rosasco, Odone, De~Vito, and
  Verri]{gerfo2008_SpectralAlgorithms}
L.~Lo~Gerfo, Lorenzo Rosasco, Francesca Odone, Ernesto De~Vito, and Alessandro
  Verri.
\newblock Spectral algorithms for supervised learning.
\newblock \emph{Neural Computation}, 20\penalty0 (7):\penalty0 1873--1897,
  2008.
\newblock \doi{10.1162/neco.2008.05-07-517}.
\newblock URL
  \url{https://www.semanticscholar.org/paper/cb8c9ff361a33445cfba85fbb48dc39315bd166c}.

\bibitem[Bauer et~al.(2007)Bauer, Pereverzev, and
  Rosasco]{bauer2007_RegularizationAlgorithms}
Frank Bauer, Sergei Pereverzev, and Lorenzo Rosasco.
\newblock On regularization algorithms in learning theory.
\newblock \emph{Journal of Complexity}, 23\penalty0 (1):\penalty0 52--72, 2007.
\newblock \doi{10.1016/j.jco.2006.07.001}.
\newblock URL \url{https://doi.org/10.1016/j.jco.2006.07.001}.

\bibitem[Caponnetto and De~Vito(2007)]{caponnetto2007_OptimalRates}
Andrea Caponnetto and Ernesto De~Vito.
\newblock Optimal rates for the regularized least-squares algorithm.
\newblock \emph{Foundations of Computational Mathematics}, 7\penalty0
  (3):\penalty0 331--368, 2007.
\newblock \doi{10.1007/s10208-006-0196-8}.
\newblock URL \url{https://doi.org/10.1007/s10208-006-0196-8}.

\bibitem[Blanchard and M{\"u}cke(2018)]{blanchard2018_OptimalRates}
Gilles Blanchard and Nicole M{\"u}cke.
\newblock Optimal rates for regularization of statistical inverse learning
  problems.
\newblock \emph{Foundations of Computational Mathematics}, 18\penalty0
  (4):\penalty0 971--1013, 2018.
\newblock \doi{10.1007/s10208-017-9359-7}.
\newblock URL \url{https://doi.org/10.1007/s10208-017-9359-7}.

\bibitem[Brown and Low(1996)]{brown1996_AsymptoticEquivalence}
Lawrence~D. Brown and Mark~G. Low.
\newblock Asymptotic equivalence of nonparametric regression and white noise.
\newblock \emph{The Annals of Statistics}, 24\penalty0 (6):\penalty0
  2384--2398, 1996.
\newblock \doi{10.1214/aos/1032181159}.
\newblock URL \url{https://doi.org/10.1214/aos/1032181159}.

\bibitem[Brown et~al.(2002)Brown, Cai, Low, and
  Zhang]{brown2002_AsymptoticEquivalence}
Lawrence~D. Brown, T.~Tony Cai, Mark~G. Low, and Cun-Hui Zhang.
\newblock Asymptotic equivalence theory for nonparametric regression with
  random design.
\newblock \emph{The Annals of Statistics}, 30\penalty0 (3):\penalty0 688--707,
  2002.
\newblock \doi{10.1214/aos/1028674838}.
\newblock URL \url{https://doi.org/10.1214/aos/1028674838}.

\bibitem[Rei{\ss}(2008)]{reiss2008_AsymptoticEquivalence}
Markus Rei{\ss}.
\newblock Asymptotic equivalence for nonparametric regression with multivariate
  and random design.
\newblock \emph{The Annals of Statistics}, 36\penalty0 (4):\penalty0
  1957--1982, 2008.
\newblock \doi{10.1214/07-AOS525}.
\newblock URL \url{https://doi.org/10.1214/07-AOS525}.

\bibitem[Sollich and Halees(2002)]{sollich2002_LearningCurves}
Peter Sollich and Anason Halees.
\newblock Learning curves for gaussian process regression: Approximations and
  bounds.
\newblock \emph{Neural Computation}, 14\penalty0 (6):\penalty0 1393--1428,
  2002.
\newblock \doi{10.1162/089976602753712990}.
\newblock URL \url{https://doi.org/10.1162/089976602753712990}.

\bibitem[Bordelon et~al.(2020)Bordelon, Canatar, and
  Pehlevan]{bordelon2020_SpectrumDependent}
Blake Bordelon, Abdulkadir Canatar, and Cengiz Pehlevan.
\newblock Spectrum dependent learning curves in kernel regression and wide
  neural networks.
\newblock In \emph{Proceedings of the 37th International Conference on Machine
  Learning}, pages 1024--1034. PMLR, 2020.
\newblock URL \url{https://proceedings.mlr.press/v119/bordelon20a.html}.

\bibitem[Cheng et~al.(2024)Cheng, Lucchi, Kratsios, and
  Belius]{cheng2024_ComprehensiveAnalysis}
Tin~Sum Cheng, Aurelien Lucchi, Anastasis Kratsios, and David Belius.
\newblock A comprehensive analysis on the learning curve in kernel ridge
  regression.
\newblock In \emph{Advances in Neural Information Processing Systems 37}.
  Neural Information Processing Systems Foundation, 2024.
\newblock \doi{10.52202/079017-0778}.
\newblock URL \url{https://doi.org/10.52202/079017-0778}.

\bibitem[Li et~al.(2023)Li, Zhang, and Lin]{li2023_AsymptoticLearning}
Yicheng Li, Haobo Zhang, and Qian Lin.
\newblock On the asymptotic learning curves of kernel ridge regression under
  power-law decay.
\newblock In \emph{37th Conference on Neural Information Processing Systems
  (NeurIPS 2023)}, 2023.
\newblock URL \url{https://openreview.net/forum?id=E4P5kVSKlT}.

\bibitem[Velikanov et~al.(2024)Velikanov, Panov, and
  Yarotsky]{velikanov2024_GeneralizationError}
Maksim Velikanov, Maxim Panov, and Dmitry Yarotsky.
\newblock Generalization error of spectral algorithms, March 2024.
\newblock URL \url{https://arxiv.org/abs/2403.11696}.

\bibitem[Zhang et~al.(2024)Zhang, Li, and
  Lin]{zhang2024_OptimalityMisspecified}
Haobo Zhang, Yicheng Li, and Qian Lin.
\newblock On the optimality of misspecified spectral algorithms.
\newblock \emph{Journal of Machine Learning Research}, 25\penalty0
  (188):\penalty0 1--50, 2024.
\newblock URL \url{https://www.jmlr.org/papers/v25/23-0383.html}.

\bibitem[Lu et~al.(2024{\natexlab{a}})Lu, Zhang, Li, and
  Lin]{lu2024_SaturationEffects}
Weihao Lu, Haobo Zhang, Yicheng Li, and Qian Lin.
\newblock On the saturation effects of spectral algorithms in large dimensions.
\newblock In \emph{38th Conference on Neural Information Processing Systems
  (NeurIPS 2024)}, September 2024{\natexlab{a}}.
\newblock \doi{10.52202/079017-0225}.
\newblock URL \url{https://openreview.net/forum?id=kJzecLYsRi}.

\bibitem[Birman and Solomyak(2003)]{birman2003_DoubleOperator}
Mikhail~Sh. Birman and Michael Solomyak.
\newblock Double operator integrals in a {Hilbert} space.
\newblock \emph{Integral Equations and Operator Theory}, 47\penalty0
  (2):\penalty0 131--168, October 2003.
\newblock \doi{10.1007/s00020-003-1157-8}.
\newblock URL \url{https://doi.org/10.1007/s00020-003-1157-8}.

\bibitem[Peller(2016)]{peller2016_MultipleOperator}
V.~V. Peller.
\newblock Multiple operator integrals in perturbation theory.
\newblock \emph{Bulletin of Mathematical Sciences}, 6\penalty0 (1):\penalty0
  15--88, April 2016.
\newblock \doi{10.1007/s13373-015-0073-y}.
\newblock URL \url{http://link.springer.com/10.1007/s13373-015-0073-y}.

\bibitem[Aleksandrov and Peller(2016)]{aleksandrov2016_OperatorLipschitz}
Alexei Aleksandrov and Vladimir Peller.
\newblock Operator {Lipschitz} functions ({English} translation), November
  2016.
\newblock URL \url{http://arxiv.org/abs/1611.01593}.
\newblock arXiv: 1611.01593.

\bibitem[Potapov and Sukochev(2011)]{potapov2011_OperatorLipschitzFunctions}
Denis Potapov and Fedor Sukochev.
\newblock Operator-{Lipschitz} functions in {Schatten}--von {Neumann} classes.
\newblock \emph{Acta Mathematica}, 207\penalty0 (2):\penalty0 375--389, 2011.
\newblock \doi{10.1007/s11511-012-0072-8}.
\newblock URL \url{http://projecteuclid.org/euclid.acta/1485892583}.

\bibitem[Caspers et~al.(2014)Caspers, Montgomery-Smith, Potapov, and
  Sukochev]{caspers2014_BestConstants}
M.~Caspers, S.~Montgomery-Smith, D.~Potapov, and F.~Sukochev.
\newblock The best constants for operator {Lipschitz} functions on {Schatten}
  classes.
\newblock \emph{Journal of Functional Analysis}, 267\penalty0 (10):\penalty0
  3557--3579, November 2014.
\newblock \doi{10.1016/j.jfa.2014.08.018}.
\newblock URL
  \url{https://linkinghub.elsevier.com/retrieve/pii/S0022123614003450}.

\bibitem[Zhang et~al.(2025)Zhang, Li, Lu, and Lin]{zhang2025_OptimalRates}
Haobo Zhang, Yicheng Li, Weihao Lu, and Qian Lin.
\newblock Optimal rates of kernel ridge regression under source condition in
  large dimensions.
\newblock \emph{Journal of Machine Learning Research}, 26\penalty0
  (219):\penalty0 1--63, 2025.
\newblock ISSN 1533-7928.
\newblock URL \url{http://jmlr.org/papers/v26/23-1679.html}.

\bibitem[Pinsker(1980)]{pinsker1980_OptimalFiltering}
M.~S. Pinsker.
\newblock Optimal filtering of square-integrable signals in gaussian noise.
\newblock \emph{Problems of Information Transmission}, 16\penalty0
  (2):\penalty0 120--133, 1980.
\newblock URL \url{https://www.mathnet.ru/eng/ppi/v16/i2/p52}.

\bibitem[Lu et~al.(2024{\natexlab{b}})Lu, Ding, Zhang, and
  Lin]{lu2024_PinskerBound}
Weihao Lu, Jialin Ding, Haobo Zhang, and Qian Lin.
\newblock On the {Pinsker} {Bound} of {Inner} {Product} {Kernel} {Regression}
  in {Large} {Dimensions}, September 2024{\natexlab{b}}.
\newblock URL \url{https://arxiv.org/abs/2409.00915}.
\newblock arXiv: 2409.00915v2.

\bibitem[Tropp(2015)]{tropp2015_IntroductionMatrix}
Joel~A. Tropp.
\newblock An introduction to matrix concentration inequalities.
\newblock \emph{Foundations and Trends in Machine Learning}, 8\penalty0
  (1--2):\penalty0 1--230, May 2015.
\newblock \doi{10.1561/2200000048}.
\newblock URL
  \url{https://www.emerald.com/ftmal/article/8/1-2/1/1332387/An-Introduction-to-Matrix-Concentration}.

\bibitem[Zhang et~al.(2023)Zhang, Li, Lu, and
  Lin]{zhang2023_OptimalityMisspecified}
Haobo Zhang, Yicheng Li, Weihao Lu, and Qian Lin.
\newblock On the optimality of misspecified kernel ridge regression.
\newblock In \emph{40th International Conference on Machine Learning (ICML
  2023)}, 2023.
\newblock URL \url{https://openreview.net/forum?id=Kg2al3GXBR}.

\end{thebibliography}
\end{document}